\documentclass[11pt]{article}

\usepackage{url}
\usepackage{amsmath}
\usepackage{dialogue} 
\usepackage{graphicx}
\usepackage{geometry}
\usepackage{enumitem}
\usepackage{float}
\usepackage{amsfonts} 
\usepackage{bbm}
  
\usepackage{parskip}

\usepackage{lastpage}
\usepackage{fancyhdr}
\usepackage{xcolor}
\usepackage[utf8]{inputenc}
 
\usepackage{url}
\usepackage{multirow}
\usepackage{hhline}
\usepackage{float}
\usepackage{caption}
\usepackage[english]{babel}
\usepackage{import}
\usepackage{graphicx}
\usepackage{subcaption}
\usepackage[toc,page]{appendix}
\usepackage{algcompatible}
\usepackage{algorithm}
\usepackage[noend]{algpseudocode}
\usepackage{commath}
\usepackage{amssymb}
\usepackage{mathtools}
\usepackage{amsthm}

\usepackage{titling}
\newcommand{\subtitle}[1]{%
  \posttitle{%
    \par\end{center}
    \begin{center}\Large#1\end{center}
    \vskip0.5em}%
}

\usepackage{indentfirst}
\usepackage{amsfonts}
\usepackage{amsmath}
\usepackage[section]{placeins}
\usepackage{afterpage}
\usepackage{esvect}
\usepackage{verbatim}
\usepackage[none]{hyphenat}

\newtheorem{theorem}{Theorem}
\newtheorem{lemma}[theorem]{Lemma}
\newtheorem{proposition}[theorem]{Proposition}
\newtheorem{corollary}[theorem]{Corollary}

\theoremstyle{definition}
\newtheorem{definition}{Definition}
\newtheorem{example}{Example}
\newtheorem{assumption}{Assumption}
\newtheorem{problem}{Problem}

\DeclareMathOperator{\sgn}{sgn}

\theoremstyle{remark}
\newtheorem*{remark}{Remark}

\usepackage{hyperref}
\hypersetup{
    colorlinks=true,
    linkcolor=blue,
    filecolor=magenta,      
    urlcolor=cyan,
}
 
\title{Lebesgue Measure Preserving Thompson's Monoid}
\author{William Li\\ Delbarton School
\\ \texttt{liwilliam2021@gmail.com} 
}
\date{}

\begin{document}
\maketitle

\begin{abstract}

This paper defines Lebesgue measure preserving Thompson's monoid, denoted by $\mathbb{G}$, which is modeled on Thompson's group $\mathbb{F}$ except that the elements of $\mathbb{G}$ are non-invertible. Moreover, it is required that the elements of $\mathbb{G}$ preserve Lebesgue measure. Monoid $\mathbb{G}$ exhibits very different properties from Thompson's group $\mathbb{F}$. The paper studies a number of algebraic (group-theoretic) and dynamical properties of $\mathbb{G}$ including approximation, mixing, periodicity, entropy, decomposition, generators, and topological conjugacy. 
\end{abstract}

\section{Introduction}

In this paper we define Lebesgue measure preserving Thompson's monoid and study its algebraic and dynamical properties. This study is at an intersection of two subjects of research. 

The first subject is concerned with Lebesgue measure preserving interval maps of $[0,1]$ onto itself, which studies dynamical properties such as transitivity, mixing, periodic points and metric entropy and finds important applications in the abstract formulation of dynamical systems, chaos theory and ergodic theory \cite{ruette2015chaos, notesonergodictheory}. The author in \cite{ruette2015chaos} motivates the study of interval maps by stating that the ``most interesting'' part of some higher-dimensional systems can be of lower dimensions, which allows, in some cases, to boil down to systems in dimension one. In particular, a recent paper \cite{2019arXiv190607558B} studies a special form of interval maps, namely, piecewise affine maps.

The second subject is concerned with Thompson's group $\mathbb{F}$ \cite{1996CFP, introductiontoTF}, which is the group of piecewise affine maps from $[0,1]$ onto itself whose derivatives are integer powers of $2$ and points at which the derivatives are discontinuous are dyadic numbers. As the derivatives are always positive, the orientation of maps is preserved. Thompson group $\mathbb{F}$ has a collection of unusual algebraic properties that make it appealing in many different and diverse areas of mathematics such as group theory, combinatorics \cite{Cleary2002} and cryptography \cite{10.1007/11496137_11}.

Except for the identity map, any Thompson's group $\mathbb{F}$ map does not preserve Lebesgue measure and any Lebesgue measure preserving interval map does not preserve orientation and thus not belong to Thompson's group $\mathbb{F}$. Thus these two subjects do not naturally intersect. We intend to build on Thompson's group $\mathbb{F}$ by making important changes to preserve Lebesgue measure. More precisely, we define Lebesgue measure preserving Thompson's monoid, denoted by $\mathbb{G}$. Monoid $\mathbb{G}$ is similar to $\mathbb{F}$ except that the derivatives of piecewise affine maps can be negative. As a result, the maps in $\mathbb{G}$ are non-invertible except for some trivial maps and exhibit very different properties from those in $\mathbb{F}$. 

To the best of our knowledge, Lebesgue measure preserving Thompson's monoid has not been proposed or studied in the literature. Unless explicitly mentioned, all the results presented and proved in this paper are original. 

The goal of this paper is to study unique properties of $\mathbb{G}$. The main results of this paper are summarized as follows. 
\begin{itemize}
\item
We show that any continuous measure preserving map can be approximated by a map in $\mathbb{G}$ with any required precision. Moreover, we show that the approximating map in $\mathbb{G}$ can be locally eventually onto (LEO) and achieve any target value of entropy that is at least $2$.
\item
We show that for any element of $\mathbb{G}$, topological mixing (TM) is equivalent to LEO and every dyadic point is preperiodic. Thus, any map in $\mathbb{G}$ is Markov. We show that for maps in a subset of $\mathbb{G}$ there exist periodic points with period of $3$, an essential feature of chaotic maps. We characterize periods of periodic points of other maps in $\mathbb{G}$.
\item 
We show that unlike $\mathbb{F}$, $\mathbb{G}$ is not finitely generated. We define equivalence classes for maps in $\mathbb{G}$ and construct a monoid by sets of equivalence classes such that the monoid is finitely generated and any map in $\mathbb{G}$ is an element of an equivalence class in the monoid.
\item
We derive sufficient conditions for a continuous map to be topologically conjugate to a measure preserving piecewise affine continuous map and in particular a map in $\mathbb{G}$.
\end{itemize}

The main results of this paper improve several results of \cite{2019arXiv190607558B}. For example, we  show that $\mathbb{G}$ that is both LEO and Markov is dense in the set of continuous measure preserving maps. Because $\mathbb{G}$ is a subset of piecewise affine continuous measure preserving maps, this result is stronger than \cite[Proposition.\ 7]{2019arXiv190607558B}, which shows that piecewise affine continuous measure preserving maps that are both LEO and Markov is dense in the set of continuous measure preserving maps. 

At an intersection of these two subjects of research, the paper demonstrates an interesting interplay between algebraic (group-theoretic) and dynamical settings. For example, in general, LEO implies TM and the converse does not hold; however, we show that for any element of $\mathbb{G}$, TM is equivalent to LEO and any map in $\mathbb{G}$ is Markov. As another example, we show that the algebraic structure of $\mathbb{G}$ leads to a simple characterization of periods of periodic points of maps in $\mathbb{G}$ and allows the use of Markov partition to study measure preserving topological conjugate maps.

The remainder of the paper is organized as follows. Section~\ref{sec:basic} reviews the basic properties of measure preserving interval maps and Thompson's group $\mathbb{F}$ and defines measure preserving Thompson's monoid $\mathbb{G}$. Section~\ref{sec:approximation} shows that a map in $\mathbb{G}$ can approximate any continuous measure preserving map with any required precision. While locally eventually onto (LEO) implies topological mixing (TM) for any maps, Section~\ref{sec:mixing} shows that TM is equivalent to LEO for any element of $\mathbb{G}$ and that a map in $\mathbb{G}$ that is LEO can approximate any continuous measure preserving map with any required precision. Section~\ref{sec:Markov} shows a salient feature of $\mathbb{G}$ that every dyadic point is preperiodic. As a result, any map in $\mathbb{G}$ is Markov. Section~\ref{sec:Markov} furthermore characterizes the periods of periodic points of maps in $\mathbb{G}$. Section~\ref{sec:entropy} investigates the entropy properties of $\mathbb{G}$ and shows that any entropy greater than or equal to $2$ can be achieved by $\mathbb{G}$. Section~\ref{sec:generators} shows that any map in $\mathbb{G}$ can be expressed as a composition of a finite number of basic maps in $\mathbb{G}$ and the generators in $\mathbb{F}$. Section~\ref{sec:equivalence} shows that unlike $\mathbb{F}$, $\mathbb{G}$ is not finitely generated. Section~\ref{sec:equivalence} defines the notions of equivalence classes and sets of equivalence classes, constructs a monoid of sets of equivalence classes and shows that the monoid has a finite number of generators and that any map in $\mathbb{G}$ is an element of an equivalence class in the monoid. Section~\ref{sec:equivalence} furthermore introduces a metric to characterize equivalence classes. Section~\ref{sec:conjugacy} studies topological conjugacy under the measure preservation constraint and uses Markov partition to characterize continuous maps that are conjugate to measure preserving maps. Finally, Section~\ref{sec:future} proposes a few areas for future study.

\section{Basic Definitions and Properties}\label{sec:basic}

\subsection{Notations}

Consider continuous interval maps from $[0,1]$ to $[0,1]$. Let $h_1$ and $h_2$ be two maps. Denote by $h_1\circ h_2$ the composition of $h_1$ and $h_2$ where $h_1\circ h_2(x) = h_1(h_2(x))$. The composition of more than two maps can be recursively defined with this definition. For any $y\in[0,1]$, define $h^{-1}(y)=\{ x\in[0,1]: h(x)=y\}$. 

Interval map $h$ defines a topological dynamical system whose evolution is given by successive iterations of the map. For any positive integer $n$, $h^n = \underbrace{h \circ h \circ \cdots \circ h}_{\text{$n$ times}}$. By convention, $h^0$ is the identity map. Point $x$ is \emph{preperiodic} if positive integers $n>m$ exist such that $h^n(x)=h^m(x)$. If $m=0$, then $x$ is \emph{periodic}.

Define trivial maps $g_{0,+}(x)=x$ and $g_{0,-}(x)=1-x$ for $x\in[0,1]$.

Let $A$ be a point in the plane of $[0,1]\times[0,1]$. Denote by $A_x$ and $A_y$ the $x$- and $y$-coordinates of point $A$, respectively. If $A$ is on the graph of map $h$, $A_y=h(A_x)$. 

Let $\mathcal{I}$ be a interval in $[0,1]$. Let $\mathcal{I}^{\circ}$ represent the interior of $\mathcal{I}$. The left and right endpoints of $\mathcal{I}$ are denoted by $\mathcal{I}^0, \mathcal{I}^1$, respectively.
If $\mathcal{I}$ is closed, then $\mathcal{I}=[\mathcal{I}^0, \mathcal{I}^1]$. Let $|\mathcal{I}|$ represent the measure of the interval: $|\mathcal{I}|=\mathcal{I}^1-\mathcal{I}^0$. For two distinct intervals $\mathcal{I}_1$ and $\mathcal{I}_2$, $\mathcal{I}_1< \mathcal{I}_2$ if $x_1\le x_2$, $\forall x_1 \in \mathcal{I}_1, x_2 \in \mathcal{I}_2$.

Let $\mathcal{I}, \mathcal{J}$ be two closed intervals of $[0,1]$ and $f_1, f_2$ be two maps. Let $f_1(\mathcal{I})\simeq f_2(\mathcal{J})$ if $f_2$ can be linearly transformed from $f_1$. That is, if $x_1=\mathcal{I}^{0}+\alpha (\mathcal{I}^{1}-\mathcal{I}^{0})$, and $x_2=\mathcal{J}^{0}+\alpha (\mathcal{J}^{1}-\mathcal{J}^{0})$ for some $\alpha\in[0,1]$, then $f_1(x_1)=f_2(x_2)$. When $f_2$ is a trivial map, $f_1(\mathcal{I})\simeq\mathcal{J}$ if $f_1$ is an affine map.

A set of distinct closed intervals $\{\mathcal{I}_1, \ldots, \mathcal{I}_n\}$ is a partition of $[0,1]$ if $\mathcal{I}^{\circ}_i \cap \mathcal{I}^{\circ}_j=\emptyset$ for any $i\neq j$ and $\bigcup_{i=1}^n \mathcal{I}_i=[0,1]$. If $\mathcal{I}_1< \cdots< \mathcal{I}_n$, then $\{|\mathcal{I}_i|\}$ completely determines $\{\mathcal{I}_i\}$. A subset of $\{\mathcal{I}_i\}$ may be a single point, i.e., $\mathcal{I}_j^0=\mathcal{I}_j^1$ where some $j$.

Denote by $\langle a,b\rangle$ interval $[a,b]$ if $a\le b$ and interval $[b,a]$ if $b<a$.

\subsection{$\lambda$-Preserving Interval Maps}\label{sec:basic-1}

Denote by $\lambda$ the Lebesgue measure on $[0,1]$ and $\mathcal{B}$ all Borel sets on $[0,1]$. 
\begin{definition}[\textsl{$\lambda$-Preserving Interval Maps}]
Continuous interval map $h$ is $\lambda$-preserving if $\forall A \in \mathcal{B}, \lambda(A) = \lambda(h^{-1}(A))$.
\label{definition:lambda-preserving}
\end{definition}
\begin{remark}
Definition~\ref{definition:lambda-preserving} does not imply $\lambda(A) = \lambda(h(A))$ for $\lambda$-preserving $h$. In fact, one can easily show that if $h$ is $\lambda$-preserving, $\lambda(A) \le \lambda(h(A))$ for any $A\in \mathcal{B}$. Except for the trivial maps of $g_{0,+}$ and $g_{0,-}$, $h$ is not invertible and $\exists A \in \mathcal{B}$ such that $\lambda(A) < \lambda(h(A))$. 
\end{remark}
For simplicity, $\lambda(A)$ is also written as $|A|$.

Let $C(\lambda)$ be the set of all continuous $\lambda$-preserving interval maps. For each map $h\in C(\lambda)$, the set of periodic points is dense on $[0,1]$ because of the Poincar\'e Recurrence Theorem and the fact that the closures of recurrent points and periodic points coincide \cite{2019arXiv190607558B}.

Let $PA(\lambda)$ be the subset of $C(\lambda)$ consisting of all piecewise affine maps.

\subsection{Thompson's Group $\mathbb{F}$}\label{sec:basic-2}

Thompson's group $\mathbb{F}$ has a few different representations such as group presentations, rectangle diagrams and piecewise linear homeomorphisms. The following focuses on the representation of piecewise linear homeomorphisms because it is closely related to $\lambda$-preserving Thompson's monoid to be introduced in the next section.

\begin{definition}[\textsl{Thompson's Group $\mathbb{F}$}]
A homeomorphism $f$ from $[0,1]$ onto $[0,1]$ is an element of Thompson's group $\mathbb{F}$ if 
\begin{itemize}
\item
$f$ is piecewise affine;
\item
$f$ is differentiable except at finitely many points;
\item
The $x$-coordinate of each of these points of non-differentiability is a dyadic number, i.e., a rational number whose denominator is an integer power of $2$;
\item
On the intervals where $f$ is differentiable, the derivatives are integer powers of $2$.
\end{itemize} 
\label{definition:ThompsonGroupF}
\end{definition}
In the remainder of this paper, $f$ is referred to an element in Thompson's group $\mathbb{F}$.
\begin{remark}
It is easily to see that $f(0)=0, f(1)=1$ and $f$ is strictly increasing on $[0,1]$ and is thus invertible. Except for the trivial map of $f=g_{0,+}$, $f$ is not $\lambda$-preserving.
\end{remark}
\begin{example}
Define the following two maps in $\mathbb{F}$.
\begin{equation}
f_A(x)=\left\{
\begin{array}{ll}
\frac{x}{2}, & 0\le x\le \frac{1}{2},\\
x-\frac{1}{4}, & \frac{1}{2}\le x \le \frac{3}{4},\\
2x-1, & \frac{3}{4} \le x \le 1,
\end{array}
\right.
f_B(x)=\left\{
\begin{array}{ll}
x, & 0\le x\le \frac{1}{2},\\
\frac{x}{2}+\frac{1}{4}, & \frac{1}{2}\le x \le \frac{3}{4},\\
x-\frac{1}{8}, & \frac{3}{4} \le x \le \frac{7}{8},\\
2x-1, & \frac{7}{8} \le x \le 1.
\end{array}
\right.
\label{eq:fAfB}
\end{equation}
\end{example}
The significance of $f_A$ and $f_B$ is that Thompson's group $\mathbb{F}$ is generated by the two maps. That is, any $f\in \mathbb{F}$ can be represented by a composition of possibly multiple $f_A$ and $f_B$ in certain order \cite{1996CFP}.

\subsection{$\lambda$-Preserving Thompson's Monoid $\mathbb{G}$}\label{sec:basic-3}

\begin{definition}[\textsl{$\lambda$-Preserving Thompson's Monoid $\mathbb{G}$}]
Continuous interval map $g$ from $[0,1]$ onto $[0,1]$ is an element of $\lambda$-preserving Thompson's monoid $\mathbb{G}$ if  
\begin{itemize}
\item
$g$ is $\lambda$-preserving;
\item
$g$ is piecewise affine;
\item
$g$ is differentiable except at finitely many points;
\item
The $x$-coordinate of each of these points of non-differentiability is a dyadic number;
\item
On an interval where $g$ is differentiable, the derivative is positive or negative and the absolute value of the derivative is an integer power of $2$.
\end{itemize} 
\label{definition:ThompsonMonoidG}
\end{definition}

\begin{remark}
The difference between $\mathbb{G}$ and $\mathbb{F}$ is that the derivatives can be negative in the maps of $\mathbb{G}$, which makes it possible for them to be $\lambda$-preserving. 
\end{remark}

\begin{remark}
It is easy to see that if $g_1, g_2, g_3\in\mathbb{G}$, then $g_1\circ g_2\in\mathbb{G}$ and $(g_1\circ g_2)\circ g_3 =g_1\circ (g_2\circ g_3)$. Trivial map $g_{0,+}$ is the identity element of $\mathbb{G}$. However, an inverse may not always exist for any given $g\in\mathbb{G}$. This is the reason that the set of maps satisfying these conditions is a monoid. 
\end{remark}

\begin{remark}
$\mathbb{G} \subset PA(\lambda)$. From Section~\ref{sec:basic-1}, the set of periodic points  of $g\in \mathbb{G}$ is dense on $[0,1]$.
\end{remark}

In the remainder of this paper, $g$ is referred to an element in $\lambda$-preserving Thompson's monoid $\mathbb{G}$. When $g$ is an affine segment on an interval, for simplicity, refer the derivative of $g$ on the interval to as the slope of the affine segment.

\begin{definition}[\textsl{Breakpoints}] 
Let $g\in \mathbb{G}$. A breakpoint of $g$ is either an endpoint at $x=0$ or $x=1$ or a point at which the derivative of $g$ is discontinuous. A breakpoint that is not an endpoint is referred to as interior breakpoint. An interior breakpoint is further categorized into type I and type II. At a type I breakpoint, the left and right derivatives are of the same sign. At a type II breakpoint, the left and right derivatives are of the opposite signs. 
\label{def:breakpoints}
\end{definition}

Point $(x,y)$ is said to be dyadic if both $x$ and $y$ are dyadic. 
\begin{lemma}
For any point $(x,y)$ of $g\in\mathbb{G}$, $y$ is dyadic if and only if $x$ is dyadic.
\label{lemma:alldyadic}
\end{lemma}
\begin{proof}
Let $(x_0, y_0)$ be a breakpoint of $g$. By definition, $x_0$ is dyadic. If $y_0=1$, then $(x_0,y_0)$ is already dyadic. Otherwise, let $c\in g^{-1}(1)$. Suppose $c>x_0$. (The case of $c<x_0$ can be proven analogously.) Let $x_0<x_1<\cdots <x_n=c$ be the set of breakpoints between $x_0$ and $c$ and the slope of the affine segment on $[x_{i-1}, x_i]$ be $(-1)^{p_i} 2^{k_i}$, with $p_i$ equal to $0$ or $1$ and $k_i$ an integer, for $i=1, \ldots, n$. Thus, $g(x_n)-g(x_0)=1-y_0=\sum_{i=1}^n (-1)^{p_i} 2^{k_i} (x_i-x_{i-1})$, which is a dyadic number. Hence, $(x_0,y_0)$ is dyadic. Therefore, any breakpoint of $g$ is dyadic. The lemma follows immediately because both endpints of an affine segment are dyadic and the derivative is in the form of $\pm2^{k}$ with integer $k$.
\end{proof}
\begin{lemma}
For $y\in[0,1]$, suppose that $g^{-1}(y)=\{x_1, \ldots, x_n\}$ and none of $x_1, \ldots, x_n$ are breakpoints. Map $g$ is $\lambda$-preserving if and only if
\begin{equation}
\sum_{i=1}^n \frac{1}{|g'(x_i)|}=\sum_{i=1}^n 2^{-k_i}=1,
\label{eq:sum2^-k=1}
\end{equation}
where $k_i$ is integer and $|g'(x_i)| = 2^{k_i}$ is the absolute value of the slope of the affine segment on which $x_i$ resides. 
\label{lemma:basicderivativeinverse}
\end{lemma}
\begin{proof}
Let $\mathcal{Y}=[y-\delta, y+\delta]$ for $\delta>0$. For a sufficiently small $\delta$, $g^{-1}(\mathcal{Y})=\bigcup_{i=1}^n I_i$, where intervals $\mathcal{I}_i$ are disjoint, $x_i \in \mathcal{I}_i$, and $g(\mathcal{I}_i)=\mathcal{Y}$ for $i=1, \ldots, n$. $\lambda(\mathcal{Y})=\lambda\left(g(\mathcal{I}_i)\right)=|g'(x_i)|\lambda(\mathcal{I}_i)$ as $\delta \to 0$. By $\lambda$-preservation and because $\mathcal{I}_i$ are disjoint, $\lambda(\mathcal{Y})=\lambda\left(g^{-1}(\mathcal{Y})\right)=\sum_{i=1}^n \lambda(\mathcal{I}_i)$. (\ref{eq:sum2^-k=1}) follows immediately. 
\end{proof}

To satisfy (\ref{eq:sum2^-k=1}), $k_i$ must be non-negative for any $i$. In contrast, for $f$, a derivative can be a negative integer power of $2$. Moreover, if $n>1$, $g'(x_i)$ has alternating signs: $g'(x_i)g'(x_{i+1})<0$ for $i=1, \ldots, n-1$. Unlike $f$, $g$ is not orientation-preserving except for the trivial maps.

\begin{definition}[\textsl{Legs and Affine Legs}]
Let interval $\mathcal{Y} \subset [0,1]$. If except for a finite number, $\forall y \in \mathcal{Y}$, set $g^{-1}(y)$ has $m$ elements, then $g^{-1}(\mathcal{Y})$ is said to have $m$ legs. When $g^{-1}(\mathcal{Y})$ has $m$ legs, $m$ intervals $\mathcal{I}_1,\ldots, \mathcal{I}_m$ with mutually disjoint interiors exist such that $g^{-1}(\mathcal{Y})=\bigcup_{i=1}^m \mathcal{I}_i$, and $g$ is monotone on every $\mathcal{I}_i$ and $\mathcal{Y}=g(\mathcal{I}_i)$ for any $i$. The graph of $g$ on $\mathcal{I}_i$ is referred to as the $i$-th leg. If $g$ is affine on every $\mathcal{I}_i$, then $g^{-1}(\mathcal{Y})$ is said to have $m$ affine legs.  
\end{definition}

\begin{definition}[\textsl{Window Perturbation}]
When $g^{-1}(\mathcal{Y})$ has $m$ affine legs for interval $\mathcal{Y}$, if $\bigcup_{i=1}^m \mathcal{I}_i$ is an interval $\mathcal{I}$, $g$ is said to be an $m$-fold window perturbation on $\mathcal{I}$.
\end{definition}

Figure~\ref{fig:WeeklyDiscussion_20200217Page-11} illustrates the definitions of legs, affine legs and window perturbation.
\begin{figure}
 \centering
  \includegraphics[width=10cm]{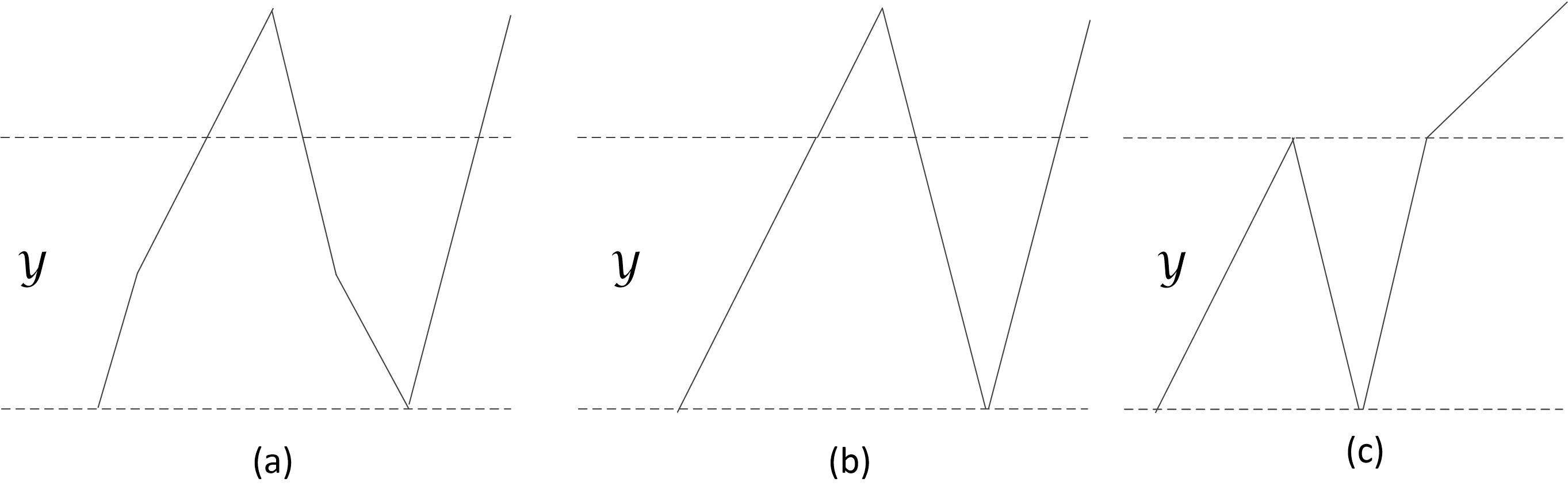}
  \caption{Illustration of the definitions of legs (a), affine legs (b) and window perturbation (c). $m=3$ in the figure.}
  \label{fig:WeeklyDiscussion_20200217Page-11}
\end{figure} 

\section{Approximation \label{sec:approximation}}

Define metric $\rho$ by $\rho(h_1,h_2)=\sup_{x\in[0,1]}|h_1(x)-h_2(x)|$ for any two continuous interval maps $h_1$ and $h_2$. This section will show that $\mathbb{G}$ has rich approximation capability in the sense that any $\lambda$-preserving continuous map $h$ can be approximated by a map $g\in\mathbb{G}$ within any $\epsilon>0$ neighborhood, i.e., $\rho(h,g)<\epsilon$. 

\begin{lemma}
Let $(x_1,y_1)$ and $(x_2,y_2)$ be two dyadic points where $x_1<x_2$ and $y_1<y_2$. Suppose that $y_2-y_1\ge x_2-x_1$. If $\frac{y_2-y_1}{x_2-x_1}\neq{2^k}$ for any integer $k$, then a dyadic point $(x_3, y_3)$ exists with $x_1<x_3<x_2$, $y_1<y_3<y_2$ such that the slopes between $(x_1,y_1), (x_3,y_3)$ and between $(x_2,y_2), (x_3,y_3)$ are both in the form of $2^{k}$ for non-negative integer $k$.
\label{lemma:x1x2x3}
\end{lemma}
\begin{proof}
Let 
\[
\frac{y_3-y_1}{x_3-x_1}=2^{k_1}, \frac{y_3-y_2}{x_3-x_2}=2^{k_2}.
\]
Then 
\[
x_3 =x_1+ \frac{2^{-k_2}(y_2-y_1)-(x_2-x_1)}{2^{k_1-k_2}-1}.
\]
Two integer solutions are given by
\[
\left\{
\begin{array}{ccc}
k_2&=&\left\lfloor{\log_2 \frac{y_2-y_1}{x_2-x_1}}\right\rfloor\\
k_1&=&k_2+1,
\end{array}
\right.
\mbox{ and }
\left\{
\begin{array}{ccc}
k_2&=&\left\lceil{\log_2 \frac{y_2-y_1}{x_2-x_1}}\right\rceil\\
k_1&=&k_2-1,
\end{array}
\right.
\]
It is easy to verify that in either solution, $(x_3,y_3)$ is dyadic and $x_1<x_3<x_2$ and $y_1<y_3<y_2$.
\end{proof}
Point $(x_3,y_3)$ in Lemma~\ref{lemma:x1x2x3} is referred to as a \emph{partition point} between points $(x_1,y_1)$ and $(x_2,y_2)$.

\begin{proposition}
For any increasing continuous map $a: [0,1]\rightarrow[0,1]$ and any $\epsilon>0$, map $f\in\mathbb{F}$ exists such that $\rho(a,f)<\epsilon$.
\end{proposition}
\begin{proof}
Because the set of dyadic points is dense and map $a$ is increasing, a set of dyadic points $(x_i,y_i)$, for $i=0, 1, \ldots, n$, exist such that $x_0=0, x_n=1$, $x_i<x_j$ and $y_i<y_j$ if $i<j$, and $a(x_i)-a(x_{i-1})<\frac{\epsilon}{3}$ and $|y_i-a(x_i)|<\frac{\epsilon}{3}$ for all $i$. Connect point $(x_{i-1},y_{i-1})$ and point $(x_{i},y_{i})$ directly if the slope between them is in the form of $2^k$ for integer $k$ or otherwise via a partition point between them defined in Lemma~\ref{lemma:x1x2x3}. The resultant map is $f\in\mathbb{F}$. For $x\in[x_{i-1},x_i]$, 
\begin{eqnarray*}
|a(x)-f(x)|&\le&\max(a(x_i), y_{i})-\min(a(x_{i-1}), y_{i-1})\\
&\le&|a(x_i)-y_{i}|+|a(x_{i-1})-y_{i-1}|+|a(x_i)-a(x_{i-1})|\\
&<&\frac{\epsilon}{3}+\frac{\epsilon}{3}+\frac{\epsilon}{3}=\epsilon.
\end{eqnarray*}
Hence, $\rho(a,f)<\epsilon$.
\end{proof}

\begin{theorem}
$\mathbb{G}$ is dense in $C(\lambda)$. That is, for any $b\in C(\lambda)$ and $\epsilon>0$, map $g\in\mathbb{G}$ exists such that $\rho(b,g)<\epsilon$.
\label{theorem:approximationproperty}
\end{theorem}
\begin{proof}
Given $b\in C(\lambda)$ and $\epsilon>0$, it has been shown in \cite{10.2307/44153684} that $h\in PA(\lambda)$ exists such that $\rho(b,h)<\epsilon/2$. In this proof, let $g=h$ initially and then perturb $g$ in the following three steps such that $\rho(g_{\text{old}}, g_{\text{new}})<\epsilon/6$ in each step, where $g_{\text{old}}$ and $g_{\text{new}}$ represent map $g$ before and after each step of perturbation, respectively, and eventually make $g$ an element of $\mathbb{G}$. 

Let $\{(x_0,y_0), (x_1, y_1),\ldots, (x_n, y_n)\}$ be the set of breakpoints of $g$, where $x_0=0, x_n=1$. Points $(x_0,y_0)$ and $(x_n,y_n)$ are the two endpoints. Not all $x_i, y_i$ are dyadic. The first step is to eliminate $(x_i,y_i)$ if $y_i$ is not dyadic. The second step is to eliminate $(x_i,y_i)$ if $x_i$ is not dyadic. The third step is to eliminate segments whose slopes are not in the form of $\pm2^k$ with integer $k$. Measure is preserved in each step.

\begin{figure}
 \centering
  \includegraphics[width=16cm]{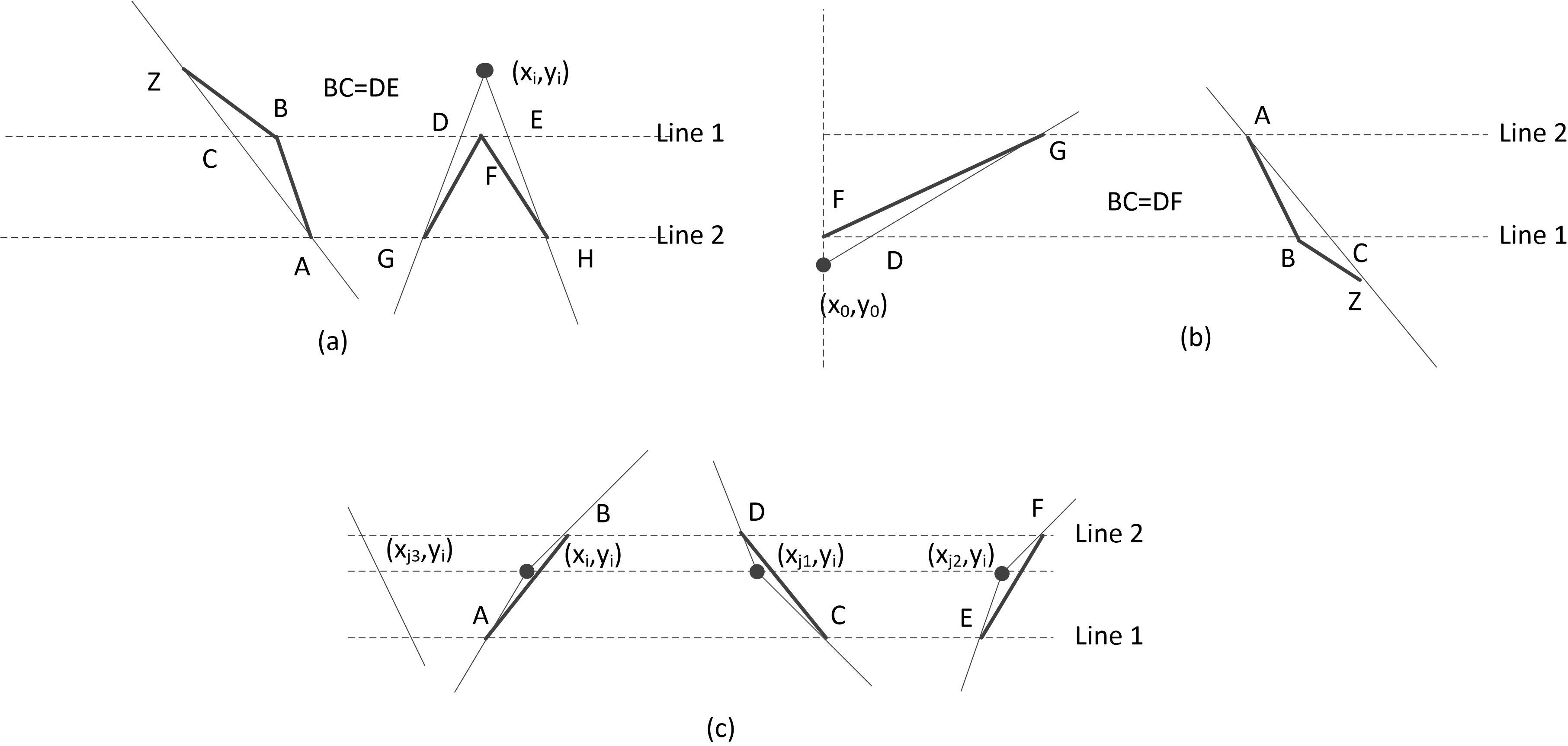}
  \caption{Step $1$ of the proof of Theorem~\ref{theorem:approximationproperty}: eliminate any breakpoint whose $y$-coordinate is not dyadic. In the figure, the thin and thick segments represent the graph of $g$ before and after the replacement respectively. In (a), type II breakpoint $(x_i,y_i)$ is replaced by point $F$ while segment $ZA$ is replaced by $ZB$ and $BA$ to preserve $\lambda$. In (b), endpoint $(x_0,y_0)$ is replaced by point $F$ while segment $ZA$ is replaced by $ZB$ and $BA$. In (c), type I breakpoints $(x_i,y_i), (x_{j_1}, y_i), (x_{j_2}, y_i)$ are eliminated by directly connecting $AB, CD, EF$.
}
  \label{fig:WeeklyDiscussion_20200217Page-7}
\end{figure} 

\emph{Step $1$}. First, suppose that point $(x_i,y_i)$ is a type II breakpoint and $y_i$ is not dyadic. Figure~\ref{fig:WeeklyDiscussion_20200217Page-7}(a) shows how $(x_i,y_i)$ is eliminated. Because $g$ is continuous and onto $[0,1]$, there exists another point $Z$ on $g$ such that $Z_y=y_i$ and the left side and right side derivatives of $Z$ are of the same sign. Point $Z$ is not necessarily a breakpoint itself. Without loss of generality, suppose that the right side derivative of $(x_i,y_i)$ is negative. Let line $2$ be a horizontal line $y=\hat{y}_i$ where $\hat{y}_i$ is dyadic, $0<\hat{y}_i<y_i$ and $y_i-\hat{y}_i<\frac{\epsilon}{12}$. Choose $\hat{y}_i$ sufficiently close to $y_i$ such that no breakpoint exists whose $y$-coordinate falls in $(\hat{y}_i, y_i)$. Let line $1$ be a horizontal line $y=\tilde{y}_i$ where $\tilde{y}_i$ is dyadic, $\hat{y}_i<\tilde{y}_i<y_i$. Let $D$ and $E$ be the two points where the left and right side affine segments of $g$ connecting $(x_i,y_i)$ intersect line $1$, and $G$ and $H$ be the two points where they intersect line $2$. Let $C$ and $A$ be the two points where $g$ connecting $Z$ intersects lines $1$ and $2$ respectively. Choose $\tilde{y}_i$ sufficiently close to $y_i$ such that $|C_x-Z_x|+E_x-D_x<|A_x-Z_x|$. Such a $\tilde{y}_i$ exists because $C_x-Z_x\to0$ and $E_x-D_x \to 0$ as $\tilde{y}_i \to y_i$. Let $F$ be any point on line $1$ with $D_x<F_x<E_x$. Let point $B$ be on line $1$ where $BC=DE$ and $|Z_x-B_x|>|Z_x-C_x|$. Replace $(x_i,y_i)$ with $F$ by connecting $G, F$ and $H, F$. Replace the portion of $g$ between $Z$ and $A$ with segments $ZB$ and $BA$. Because $BC=DE$, measure is preserved between the horizontal line $y=y_i$ and line $1$ and between line $1$ and line $2$. Therefore, a type II breakpoint $(x_i,y_i)$ is eliminated while six new breakpoints are added: $A, B, F, G, H$ all have dyadic $y$-coordinate and $Z$ is a type I breakpoint with non-dyadic $y$-coordinate. $\rho(g_{\text{old}}, g_{\text{new}})<\frac{\epsilon}{12}$. Repeat the preceding procedure, one can eliminate all type II breakpoints with non-dyadic $y$-coordinates.

Endpoints $(x_0,y_0)$ and $(x_n,y_n)$ where $y_0$ or $y_n$ is not dyadic can be eliminated analogously as shown Figure~\ref{fig:WeeklyDiscussion_20200217Page-7}(b).

After the preceding procedure, the only remaining breakpoints that have non-dyadic $y$-coordinates are of type I. If $(x_i,y_i)$ is one such breakpoint, then there exists at least another breakpoint $(x_j,y_i)$ where $x_j\neq x_i$. Figure~\ref{fig:WeeklyDiscussion_20200217Page-7}(c) illustrates an example where two such breakpoints $(x_{j_1},y_i)$ and $(x_{j_2},y_i)$ exist. It is possible that points with $y$-coordinate equal to $y_i$ exist and are not a breakpoint, such as $(x_{j_3},y_i)$ in the figure. Let line $1$ and line $2$ be horizontal lines $y=\hat{y}_i$ and $y=\tilde{y}_i$, respectively, where $\hat{y}_i$ and $\tilde{y}_i$ are both dyadic, $\hat{y}_i<y_i<\tilde{y}_i$ and $\tilde{y}_i-\hat{y}_i<\frac{\epsilon}{12}$. Let $\hat{y}_i$ and $\tilde{y}_i$ be sufficiently close to $y_i$ that no breakpoint exists whose $y$-coordinate is not equal to $y_i$ and falls in $(\hat{y}_i,\tilde{y}_i)$. Let $A, C, E$ be the points where $g$ intersects line $1$ and $B, D, F$ be the points where $g$ intersects line $2$. Replace the portion $g$ between $A$ and $B$ with segment $AB$, between $C$ and $D$ with segment $CD$, and between $E$ and $F$ with segment $EF$. Measure is preserved between line $1$ and line $2$ after the replacement, because measure is preserved before the replacement between line $1$ and the horizontal line $y=y_i$ and between line $2$ and the horizontal line $y=y_i$. Therefore, the three type I breakpoints $(x_i,y_i)$, $(x_{j_1},y_i)$, $(x_{j_2},y_i)$ are eliminated while six new breakpoints are added: $A, B, C, D, E, F$ all have dyadic $y$-coordinates. $\rho(g_{\text{old}}, g_{\text{new}})<\frac{\epsilon}{12}$. Repeat the preceding procedure, one can eliminate all type I breakpoints with non-dyadic $y$-coordinates.

In step $1$, $g(x)$ for any $x$ is perturbed at most twice. For example, in Figure~\ref{fig:WeeklyDiscussion_20200217Page-7}(a) $x$ close to and greater than $Z_x$ is perturbed once and will be perturbed again when type I breakpoint $Z$ is to be eliminated as in Figure~\ref{fig:WeeklyDiscussion_20200217Page-7}(c). On the other hand, point $(x_i,y_i)$ is perturbed only once shown in Figure~\ref{fig:WeeklyDiscussion_20200217Page-7}(a). Hence, at the end of step $1$, $\rho(g,h)<\frac{\epsilon}{12}+\frac{\epsilon}{12}=\frac{\epsilon}{6}$.

\emph{Step $2$}. Eliminate $(x_i,y_i)$ if $x_i$ is not dyadic. Because $g$ is piecewise affine on $[0,1]$, $g$ satisfies a Lipschitz condition, i.e., $|g(x_1)-g(x_2)|<K |x_1-x_2|$ for some fixed number $K$ and any $x_1, x_2\in[0,1]$. Consider a set of horizontal equally-spaced dyadic lines, $y=i\cdot 2^{-M}$ for $i=0, 1, \ldots, 2^M$, where the spacing between any adjacent dyadic lines is equal to $\Delta y=2^{-M}$. Let $2^{-M}<\frac{\epsilon}{6}$ and is sufficiently small that all the breakpoints of $g$ are on the lines. The lines are referred to as lines $1$, $2$, $3$ and so on, as shown in Figure~\ref{fig:WeeklyDiscussion_20200217Page-5}. 

\begin{figure}
 \centering
  \includegraphics[width=16cm]{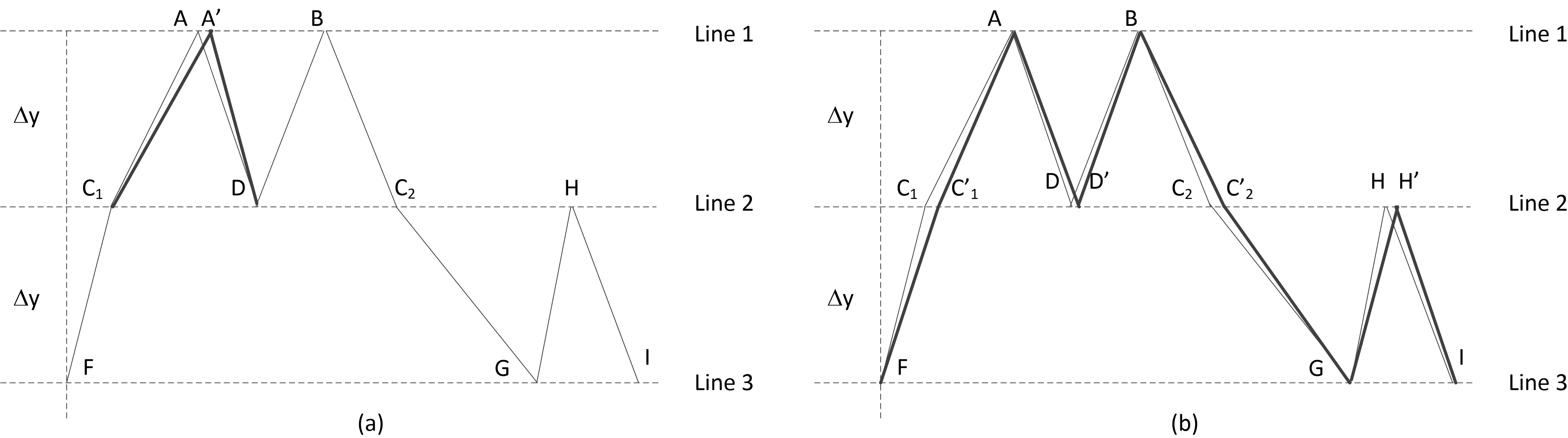}
  \caption{Step $2$ of the proof of Theorem~\ref{theorem:approximationproperty}: eliminate any breakpoint whose $x$-coordinate is not dyadic. In the figure, the thin and thick segments represent the graph of $g$ before and after the replacement respectively. In (a), point $A$, which is not dyadic, is replaced by dyadic point $A'$. In (b), $C_1, D, C_2, H$ are not dyadic and are replaced by dyadic points $C'_1, D', C'_2, H'$ respectively.}
  \label{fig:WeeklyDiscussion_20200217Page-5}
\end{figure} 

Because $g$ is onto $[0,1]$, one or multiple points of $g$ exist on line $1$. Figure~\ref{fig:WeeklyDiscussion_20200217Page-5}(a) shows two such points $A, B$. If $A_x$ is not dyadic, then replace it with sufficiently close dyadic number $A'_x$, i.e., with $|A_x-A'_x|<\frac{\epsilon}{6K}$. Because $A_x$ is not dyadic, $A$ cannot be an endpoint and thus must be a type II breakpoint and connect to two points of the graph of $g$ on line $2$, one to the left, $C_1$, and one to the right, $D$. Replace the original affine segments $C_1A, DA$ with $C_1A', DA'$. Because $D_x-A'_x+A'_x-C_{1,x}= D_x-A_x+A_x-C_{1,x}$, measure is preserved between line $1$ and line $2$ after the replacement. Therefore, all the points on line $1$ are now dyadic.

Now consider the set of points of the graph of $g$ on line $2$ that are not dyadic. If a point is a type II breakpoint, e.g., $D$ and $H$ shown in Figure~\ref{fig:WeeklyDiscussion_20200217Page-5}(b), it can be replaced like $A$ on line $1$. Otherwise, it must connect to one point of $g$ on line $1$ and another point of $g$ on line $3$, one to the left and one to the right. Let $C_1, C_2, \ldots$ be these points. Figure~\ref{fig:WeeklyDiscussion_20200217Page-5}(b) shows $C_1, C_2$ on line $2$. Point $C_1$ connects to the right to $A$ on line $1$ and to the left to $F$ on line $3$. Point $C_2$ connects to the left to $B$ on line $1$ and to the right to $G$ on line $3$. Let $\sgn_{C_i}=1$ (or respectively, $\sgn_{C_i}=-1$) if $C_i$ connects to the left (or respectively, right) on line $1$. In Figure~\ref{fig:WeeklyDiscussion_20200217Page-5}(b), $\sgn_{C_1}=-1$ and $\sgn_{C_2}=1$. If any $C_{i,x}$ are not dyadic, then replace them with sufficiently close dyadic numbers $C'_{i,x}$, i.e., with $|C_{i,x}-C'_{i,x}|<\frac{\epsilon}{6K}$, such that $\sgn_{C'_i}=\sgn_{C_i}$ and $\sum_i \sgn_{C'_i} C'_{i,x}=\sum_i \sgn_{C_i} C_{i,x}$. Such dyadic numbers $C'_{i,x}$ exist because of $\lambda$-preservation and because $\Delta y = 2^{-M}$ and all points on line $1$ are dyadic, and therefore $\sum_i \sgn_{C_i} C_{i,x}$ must be dyadic even though individual $C_{i,x}$ are not dyadic. Measure is thus preserved between line $1$ and line $2$ and between line $2$ and line $3$ after the replacement. Therefore, all the points on line $2$ are now dyadic.

Repeat the same procedure for all the $2^M+1$ lines. Therefore all the breakpoints are now dyadic. In step $2$, a breakpoint is perturbed horizontally at most by $\frac{\epsilon}{6K}$. Hence, at the end of step $2$, $\rho(g,h)<\frac{\epsilon}{6}+K\cdot \frac{\epsilon}{6K}=\frac{\epsilon}{3}$.

\emph{Step $3$}. The final step of the proof is to connect these breakpoints between adjacent horizontal dyadic lines with piecewise affine segments whose slopes are in the form of $\pm 2^k$ for integer $k$. Let $\mathcal{Y}$ be the interval between two adjacent horizontal dyadic lines. Because no breakpoint exists on $\mathcal{Y}^{\circ}$, $g^{-1}(\mathcal{Y})$ can be written as $\bigcup_i \mathcal{I}_i$, where $\{\mathcal{I}_i\}$ have mutually disjoint interiors, $g(\mathcal{I}_i)=\mathcal{Y}$ for all $i$, and $g$ is affine on all $\mathcal{I}_i$. From step $2$, the endpoints of any $\mathcal{I}_i$ are dyadic. Let $|\mathcal{I}_i|=l_i\cdot 2^{-N_i}$ for integers $l_i$ and $N_i$ where $l_i$ is an odd number. Because of $\lambda$-preservation, $|\mathcal{Y}|=\sum_i |\mathcal{I}_i|$, i.e., $2^{-M}=\sum_i l_i\cdot 2^{-N_i}$. Thus, $N_i\ge M$.

Replace the affine segment of $g$ on $\mathcal{I}_i$ with a window perturbation, as illustrated in Figure~\ref{fig:WeeklyDiscussion_20200217Page-9}. The window perturbation is $l_i$-fold where each leg has slope of $\pm 2^{N_i-M}$ and covers the same interval as the original affine segment. Therefore measure is preserved after the replacement.

\begin{figure}
 \centering
  \includegraphics[width=8cm]{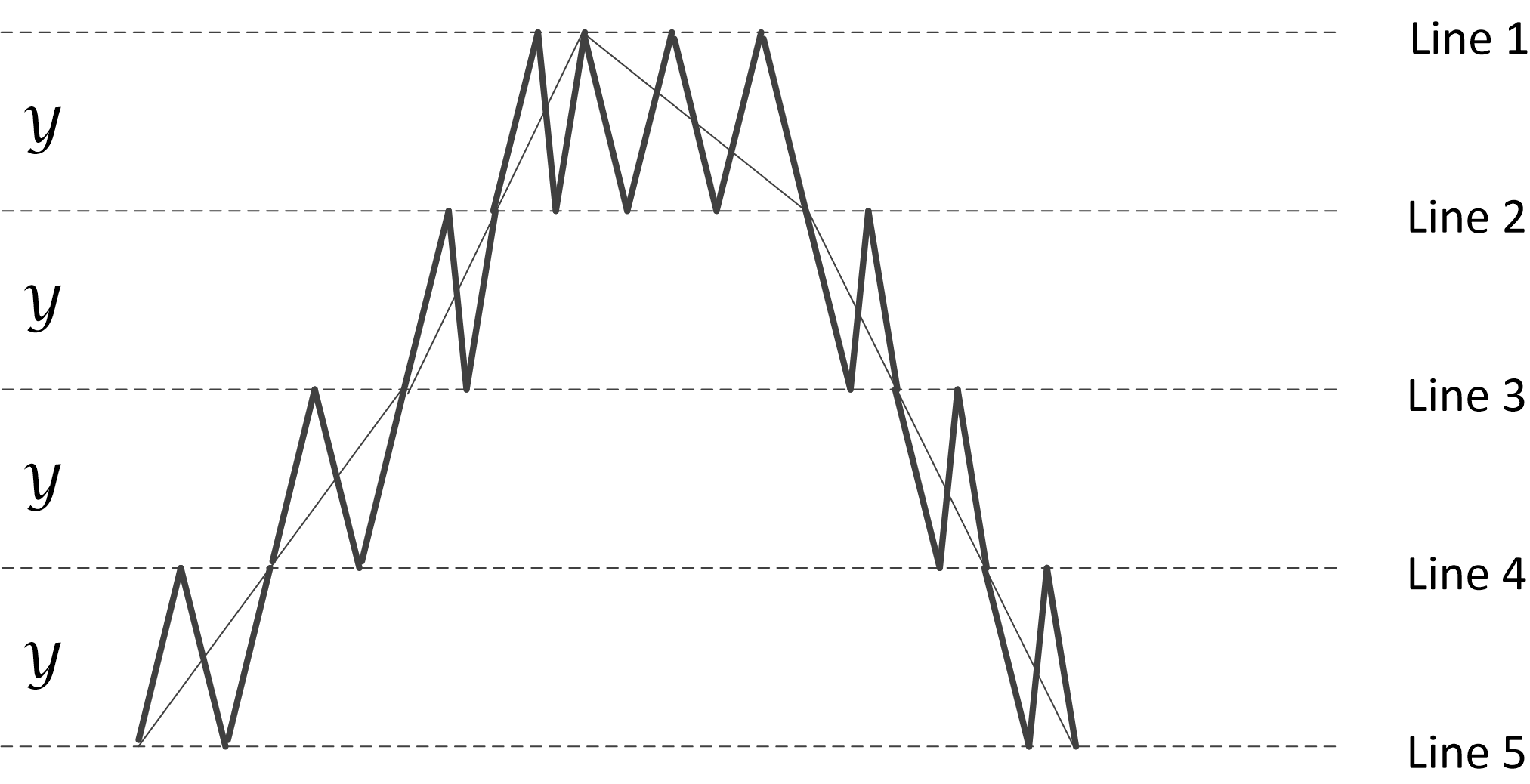}
  \caption{Step $3$ of the proof of Theorem~\ref{theorem:approximationproperty}: connect breakpoints on adjacent horizontal dyadic lines. In the figure, the thin and thick segments represent the graph of $g$ before and after the replacement respectively. The slopes of the thin segments are not in the form of  $\pm 2^k$ for integer $k$. The slopes of the thick segments are in the form of  $\pm 2^k$ for integer $k$. }
  \label{fig:WeeklyDiscussion_20200217Page-9}
\end{figure} 

In step $3$, the maximum vertical perturbation cannot exceed the spacing between any adjacent dyadic lines, which is $2^{-M}<\frac{\epsilon}{6}$.

Hence, at the end of step $3$, the resultant $g$ is an element of $\mathbb{G}$ and $\rho(h,g)<\frac{\epsilon}{3}+\frac{\epsilon}{6}=\frac{\epsilon}{2}$. $\rho(b,g)\le\rho(b,h)+\rho(h,g)<\epsilon$. This completes the proof.
\end{proof}

\section{Mixing \label{sec:mixing}}

Topological transitivity and mixing properties of dynamical systems are widely studied in the literature. There are several different versions of such properties, including topological transitivity (TT), strong transitivity (ST), exact transitivity (ET), weak mixing (WM), topological mixing (TM), and locally eventually onto (LEO). It is shown in \cite{2016Akin} that many of these versions, such as TT, ET, WM, and TM, are implied by LEO but not the other way around. Thus, LEO is the strongest version among them. This section focuses on LEO and TM.

\begin{definition}[\textsl{Topological Mixing (TM)}]
An interval map $h$ is TM if for all nonempty open sets $U, V$ in $[0,1]$, there exists an integer $N\ge0$ such that $\forall n\ge N$, $f^n(U)\cap V \neq \emptyset$.
\end{definition}
\begin{definition}[\textsl{Locally Eventually Onto (LEO)}]
An interval map $h$ is LEO if for every nonempty open set $U$ in $[0,1]$ there is an integer $N$ such that $h^N(U)=[0,1]$.
\end{definition}
\begin{remark}
It can be shown (\cite[Proposition.\ 2.8]{ruette2015chaos}) that $h$ is TM if and only if $\forall \epsilon>0$ and open $U \subset [0,1]$, there is an integer $N$ such that $h^n(U)\supset [\epsilon, 1-\epsilon]$ for any $n\ge N$. However, it is not necessary that $h^n(U)=[0,1]$. Clearly, LEO implies TM. The difference between LEO and TM lies at two endpoints.
\end{remark}
\begin{remark}
If $h$ is LEO or TM, then its trajectory is sensitive to initial conditions in the sense that two arbitrarily close initial conditions $x_1$ and $x_2$ lead to divergent trajectories as $h^n(x_1)$ and $h^n(x_2)$ eventually spread over the entire interval $(0,1)$ for $n\ge 0$.
\end{remark}
\begin{lemma}[\textsl{Barge and Martin, 1985, \cite{10.2307/2044896}}]
If a continuous interval map $h$ has a dense set of periodic points, then a collection of intervals $\{\mathcal{J}_1, \mathcal{J}_2, \ldots\}$ of $[0,1]$ exist with mutually disjoint interiors such that for each $i$, $h^2(\mathcal{J}_i)=\mathcal{J}_i$, $h(\mathcal{J}_i)= \mathcal{J}_j$ for some $j\ge 1$, $h^{-1}(h(\mathcal{J}_i))=\mathcal{J}_i$, and $h^2(x)=x$ on $[0,1]\setminus \bigcup_{i\ge1}\mathcal{J}_i^{\circ}$. If $|\{\mathcal{J}_1, \mathcal{J}_2, \ldots\}|>1$, then there can only be two cases. In the first case,  $h(\mathcal{J}_i)=\mathcal{J}_i$ for all $i$ and $h(x)=x$ on $[0,1]\setminus \bigcup_{i\ge1}\mathcal{J}_i^{\circ}$. In the second case, $h(\mathcal{J}_i)>h(\mathcal{J}_j), \forall \mathcal{J}_i<\mathcal{J}_j$, $h(\mathcal{J}_i)=\mathcal{J}_i$ for at most one $i$ and $h(x)=1-x$ on $[0,1]\setminus \bigcup_{i\ge1}\mathcal{J}_i^{\circ}$.
\label{lemma:Jcollection1}
\end{lemma}

\begin{figure}
 \centering
  \includegraphics[width=16cm]{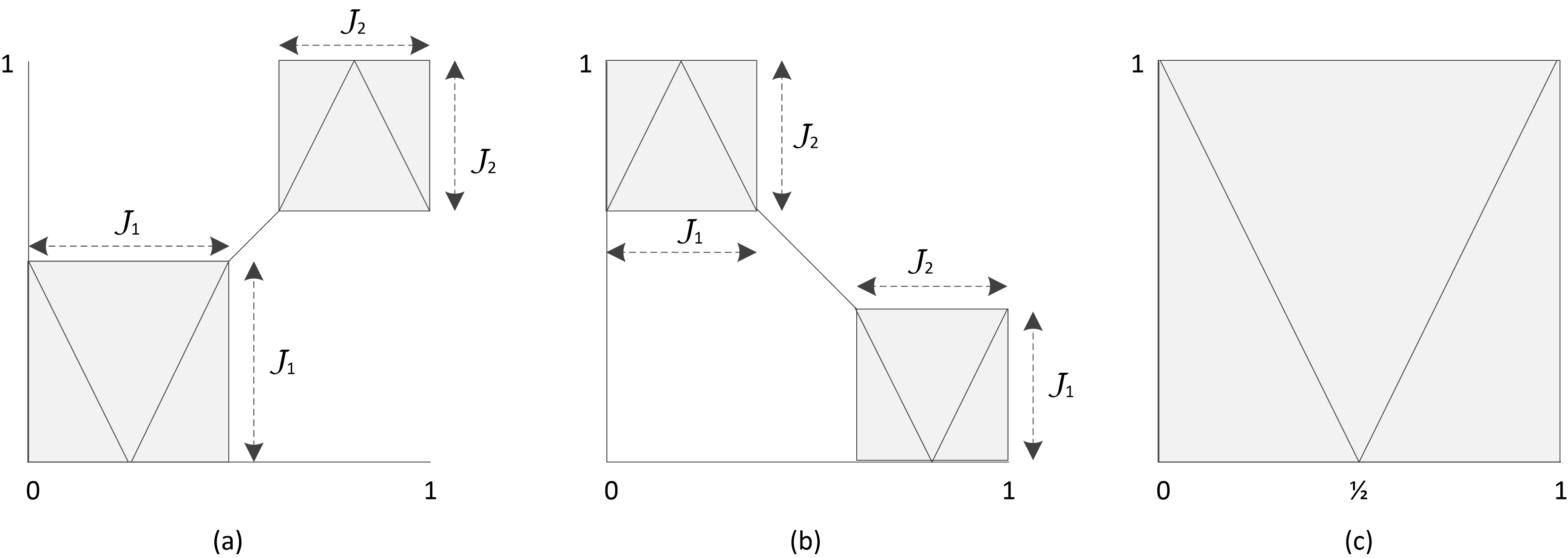}
  \caption{Examples of $g\in\mathbb{G}$. Map $g$ is not TM or LEO in (a) and (b) and is TM and LEO in (c). In (a) $g(\mathcal{J}_1)=\mathcal{J}_1$ and $g(\mathcal{J}_2)=\mathcal{J}_2$. In (b) $g(\mathcal{J}_1)=\mathcal{J}_2$ and $g(\mathcal{J}_2)=\mathcal{J}_1$.}
  \label{fig:WeeklyDiscussion_20200420Page-5}
\end{figure} 

Recall from Section~\ref{sec:basic-3} that for each map $g\in\mathbb{G}$, the set of periodic points is dense on $[0,1]$. A collection of $\{\mathcal{J}_1, \mathcal{J}_2, \ldots\}$ of $g$ exist to have the properties stated in Lemma~\ref{lemma:Jcollection1}.

The two cases of Lemma~\ref{lemma:Jcollection1} are illustrated in Figure~\ref{fig:WeeklyDiscussion_20200420Page-5}(a) and (b). Set $\{\mathcal{J}_1, \mathcal{J}_2, \ldots\}$ is not unique when $[0,1]\setminus \bigcup_{i\ge1}\mathcal{J}_i^{\circ}\neq\emptyset$. For example, in Figure~\ref{fig:WeeklyDiscussion_20200420Page-5}(a) $g$ is an affine segment with slope $1$ on interval $[\mathcal{J}^1_1, \mathcal{J}^0_2]$. Let $\mathcal{J}'_1=[\mathcal{J}^0_1, \mathcal{J}^1_1+\delta_1]$ and $\mathcal{J}'_2=[\mathcal{J}^0_2-\delta_2, \mathcal{J}^1_2]$ with $\delta_1, \delta_2>0$ and $\delta_1+\delta_2\le\mathcal{J}^0_2-\mathcal{J}^1_1$. In Figure~\ref{fig:WeeklyDiscussion_20200420Page-5}(b) $g$ is an affine segment with slope $-1$ on interval $[\mathcal{J}^1_1, \mathcal{J}^0_2]$. Let $\mathcal{J}'_1=[\mathcal{J}^0_1, \mathcal{J}^1_1+\delta]$ and $\mathcal{J}'_2=[\mathcal{J}^0_2-\delta, \mathcal{J}^1_2]$ with $0<\delta\le\frac{\mathcal{J}^0_2-\mathcal{J}^1_1}{2}$. In either case, $\{\mathcal{J}'_1,\mathcal{J}'_2\}$ has the same properties as $\{\mathcal{J}_1,\mathcal{J}_2\}$ as far as Lemma~\ref{lemma:Jcollection1} is concerned. Intervals $\mathcal{J}'_1,\mathcal{J}'_2$ can shrink in length to become $\mathcal{J}_1,\mathcal{J}_2$. More precisely, a collection of intervals $\{\mathcal{J}_1, \mathcal{J}_2, \ldots\}$ in Lemma~\ref{lemma:Jcollection1} is said to be of minimum length if there does not exist a distinct collection $\{\mathcal{J}'_1, \mathcal{J}'_2, \ldots\}$, also satisfying Lemma~\ref{lemma:Jcollection1}, such that $\mathcal{J}'_i \subseteq \mathcal{J}_i$ for $i=1, 2, \ldots$.

\begin{lemma}
Let $g\in\mathbb{G}$. Suppose that $\{\mathcal{J}_1, \mathcal{J}_2, \ldots\}$ is a collection of $g$ in Lemma~\ref{lemma:Jcollection1}. If $\{\mathcal{J}_1, \mathcal{J}_2, \ldots\}$ is of minimum length, then the endpoints of each $\mathcal{J}_i$ are dyadic. 
\label{lemma:dyadicJenpoints}
\end{lemma}
\begin{proof}
First suppose that $|\{\mathcal{J}_1, \mathcal{J}_2, \ldots\}|=1$. That is, $\{\mathcal{J}_1, \mathcal{J}_2, \ldots\}=\{\mathcal{J}_1\}$. If $\mathcal{J}_1=[0,1]$, then the proof is already done. If $\mathcal{J}^0_1>0$ and $\mathcal{J}^1_1=1$, then by Lemma~\ref{lemma:Jcollection1} $g(x)=x$ on $[0,\mathcal{J}^0_1]$. A type II breakpoint $(x_0,y_0)$ must exist such that $y_0=\mathcal{J}^0_1$, for otherwise $\delta>0$ exists such that the set $g^{-1}(y)$ consists of only one element for $y\in[\mathcal{J}^0_1, \mathcal{J}^0_1+\delta]$ and because of $\lambda$-preservation, $g(x)=x$ on $[\mathcal{J}^0_1, \mathcal{J}^0_1+\delta]$. Thus, 
$\mathcal{J}_1$ can shrink to $[\mathcal{J}^0_1+\delta, \mathcal{J}^1_1]$. Therefore, $\mathcal{J}^0_1$ is dyadic by Lemma~\ref{lemma:alldyadic}. Similarly, if $\mathcal{J}^0_1=0$ and $\mathcal{J}^1_1<1$, then $\mathcal{J}^1_1$ can be shown to be dyadic. Now suppose $\mathcal{J}^0_1>0$ and $\mathcal{J}^1_1<1$. By Lemma~\ref{lemma:Jcollection1} there can only be two cases. In the first case, $g(x)=x$ on $[0,1]\setminus\mathcal{J}_1^{\circ}$. Type II breakpoints $(x_0,y_0)$ and $(x_1,y_1)$ must exist such that $y_0=\mathcal{J}^0_1$ and $y_1=\mathcal{J}^1_1$. Therefore, $\mathcal{J}^0_1$ and $\mathcal{J}^1_1$ are both dyadic by Lemma~\ref{lemma:alldyadic}. In the second case, $g(x)=1-x$ on $[0,1]\setminus\mathcal{J}_1^{\circ}$. A type II breakpoint $(x_0,y_0)$ must exist such that $y_0=\mathcal{J}^0_1$ or $y_0=\mathcal{J}^1_1$, for otherwise $\delta>0$ exists such that the set $g^{-1}(y)$ consists of only one element for $y\in[\mathcal{J}^0_1, \mathcal{J}^0_1+\delta]$ and $y\in[\mathcal{J}^1_1-\delta, \mathcal{J}^1_1]$ and thus $\mathcal{J}_1$ can shrink to $[\mathcal{J}^0_1+\delta, \mathcal{J}^1_1-\delta]$. Therefore, at least one of $\mathcal{J}^0_1$ and $\mathcal{J}^1_1$ are dyadic. By Lemma~\ref{lemma:alldyadic}, both of them are dyadic because $g(\mathcal{J}^0_1)=\mathcal{J}^1_1$ and $g(\mathcal{J}^1_1)=\mathcal{J}^0_1$.

Next suppose that $|\{\mathcal{J}_1, \mathcal{J}_2, \ldots\}|>1$. Consider the two cases of Lemma~\ref{lemma:Jcollection1}. In the first case, $g(x)=x$ on $[0,1]\setminus\bigcup_{i\ge1}\mathcal{J}_i^{\circ}$. Type II breakpoints $(x_0,y_0)$ and $(x_1,y_1)$ must exist such that $y_0=\mathcal{J}^0_i$ and $y_1=\mathcal{J}^1_i$ for each $i$, for otherwise $\mathcal{J}_i$ can shrink similar to what is shown above. Numbers $\mathcal{J}^0_i$ and $\mathcal{J}^1_i$ are both dyadic. In the second case, $g(x)=1-x$ on $[0,1]\setminus\bigcup_{i\ge1}\mathcal{J}_i^{\circ}$. For each $i$, $h(\mathcal{J}_i)=\mathcal{J}_j$. If $j=i$, type II breakpoint $(x_0,y_0)$ must exist such that $y_0=\mathcal{J}^0_i$ or $y_0=\mathcal{J}^1_i$, for otherwise $\mathcal{J}_i$ can shrink similar to what is shown above. Numbers $\mathcal{J}^0_i$ and $\mathcal{J}^1_i$ are both dyadic. Now consider $j\neq i$. Without loss of generality, suppose that $\mathcal{J}_j < \mathcal{J}_i$. Thus, $g(\mathcal{J}^0_i)=\mathcal{J}^1_j$. A type II breakpoints $(x_0,y_0)$ must exist such that $y_0=\mathcal{J}^0_i$ or $y_0=\mathcal{J}^1_j$, for otherwise $\delta>0$ exists such that the set $g^{-1}(y)$ consists of only one element for $y\in[\mathcal{J}^0_i, \mathcal{J}^0_i+\delta]$ and $y\in[\mathcal{J}^1_j-\delta, \mathcal{J}^1_j]$ and thus $\mathcal{J}_i$ can shrink to $[\mathcal{J}^0_i+\delta, \mathcal{J}^1_i]$ and $\mathcal{J}_j$ can shrink to $[\mathcal{J}^0_j, \mathcal{J}^1_j-\delta]$. Thus, at least one of $\mathcal{J}^0_i$ and $\mathcal{J}^1_j$ are dyadic. By Lemma~\ref{lemma:alldyadic}, both of them are dyadic because $g(\mathcal{J}^0_i)=\mathcal{J}^1_j$. Analogously, both $\mathcal{J}^1_i$ and $\mathcal{J}^0_j$ can be shown to be dyadic.
\end{proof}

\begin{lemma}[\textsl{Bobok and Troubetzkoy, 2019, \cite[Lemma.\ 5]{2019arXiv190607558B}}]
In Lemma~\ref{lemma:Jcollection1}, $h$ is TM if and only if the collection of intervals satisfy $\{\mathcal{J}_1, \mathcal{J}_2, \ldots\}=\{[0,1]\}$, and $h$ is LEO if and only if in addition both of the sets $h^{-2}(0)\cap (0,1)$ and $h^{-2}(1)\cap (0,1)$ are non-empty.
\label{lemma:Jcollection2}
\end{lemma}

Figure~\ref{fig:WeeklyDiscussion_20200420Page-5} provides three examples of $g\in\mathbb{G}$, two of which are not TM or LEO and the third one is LEO. In (a) and (b), $\mathcal{J}_1\cup\mathcal{J}_2 \subset [0,1]$ with $\mathcal{J}_1^{\circ} \cap \mathcal{J}_2^{\circ}=\emptyset$ and $g^2(\mathcal{J}_i)=(\mathcal{J}_i)$ for $i=1,2$. As a result, if $U \subset \mathcal{J}_i$, then given $n$,  $g^n(U)\subset \mathcal{J}_j$ for either $j=1$ or $j=2$ but it is impossible that $g^n(U)=[0,1]$. Therefore, $g$ in (a) and (b) is not TM or LEO. In (c), such a partition of separate $\mathcal{J}_1$ and $\mathcal{J}_2$ does not exist and $\{\mathcal{J}_1, \mathcal{J}_2, \ldots\}=\{[0,1]\}$. Thus, $g$ in (c) is TM by Lemma~\ref{lemma:Jcollection2}. In addition, because $\frac{1}{2}\in g^{-2}(0)$ and $\frac{1}{2}\in g^{-2}(1)$, $g$ in (c) is LEO.

As remarked, in general, LEO implies TM and the converse does not hold. However, the two are equivalent for $g\in\mathbb{G}$ as stated in the following theorem.

\begin{theorem}
If $g\in \mathbb{G}$ is TM, then $g$ is LEO.
\label{theorem:GLEOMT}
\end{theorem}
\begin{proof}
We will prove the theorem by contradiction. Assume that $g^{-2}(0)\cap (0,1)=\emptyset$. First we will show by contradiction that $g^{-1}(0)\cap (0,1)=\emptyset$. Assume that $c\in g^{-1}(0)$ and $0<c<1$. Because $g^{-1}(c)\subseteq g^{-2}(0)$, $g^{-1}(c)\cap (0,1)=\emptyset$, which is impossible because $g$ is continuous and onto $[0,1]$. Therefore, $g^{-1}(0)\subseteq\{0,1\}$. 

If $g^{-1}(0)=\{0\}$, then $g(x)$ is an affine segment with slope $1$ on $[0,\delta]$ for some sufficiently small $\delta>0$. Thus $g(x)$ is not TM because $g^n\left((0,\delta)\right)=(0,\delta)$ for any $n$ and does not mix with $(\delta,1)$. Contradiction with the hypothesis.

If $g^{-1}(0)=\{1\}$, then $g(x)$ is an affine segment with slope $-1$ on $[1-\delta_1,1]$ for some sufficiently small $\delta_1>0$. Moreover, for $g^{-2}(0)\cap (0,1)=\emptyset$, it follows that $g^{-1}(1)=\{0\}$. The graph of $g(x)$ is an affine segment with slope $-1$ on $[0,\delta_2]$ for some sufficiently small $\delta_2>0$. Let $\delta=\min(\delta_1,\delta_2)$. $g^n\left((0,\delta)\right)=(0,\delta)$ for even $n$ and $g^n\left((0,\delta)\right)=(1-\delta,1)$ for odd $n$. Thus $g(x)$ is not TM because $g^n\left((0,\delta)\right)$ for any $n$ and does not mix with $(\delta,1-\delta)$. Contradiction with the hypothesis.

If $g^{-1}(0)=\{0,1\}$, then $g^{-1}(1) \subset (0,1)$. Therefore, $g^{-2}(0)\cap(0,1)\neq \emptyset$. Contradiction with the assumption.

Hence, $g^{-2}(0)\cap (0,1)\neq\emptyset$. We can analogously show that $g^{-2}(1)\cap (0,1)\neq\emptyset$. By Lemma~\ref{lemma:Jcollection2}, $g$ is LEO.
\end{proof}

\begin{theorem}
Denote by $\mathbb{G}_{\text{LEO}}$ the subset of $\mathbb{G}$ whose elements are LEO.
$\mathbb{G}_{\text{LEO}}$ is dense in $C(\lambda)$.
\label{theorem:GLEOdense}
\end{theorem}
\begin{proof}

The idea is to further perturb $g$ obtained in the proof of Theorem~\ref{theorem:approximationproperty} to meet the conditions required in Lemma~\ref{lemma:Jcollection2}, thereby making $g$ LEO. Specifically, map $g$ obtained in the proof of Theorem~\ref{theorem:approximationproperty} is an element of $\mathbb{G}$, and thus the set of periodic points of $g$ is dense on $[0,1]$ from Section~\ref{sec:basic-3}. A collection of intervals $\{\mathcal{J}_1, \mathcal{J}_2, \ldots\}$ of $[0,1]$ exist to have the properties stated in Lemma~\ref{lemma:Jcollection1}. 

\begin{figure}
 \centering
  \includegraphics[width=12cm]{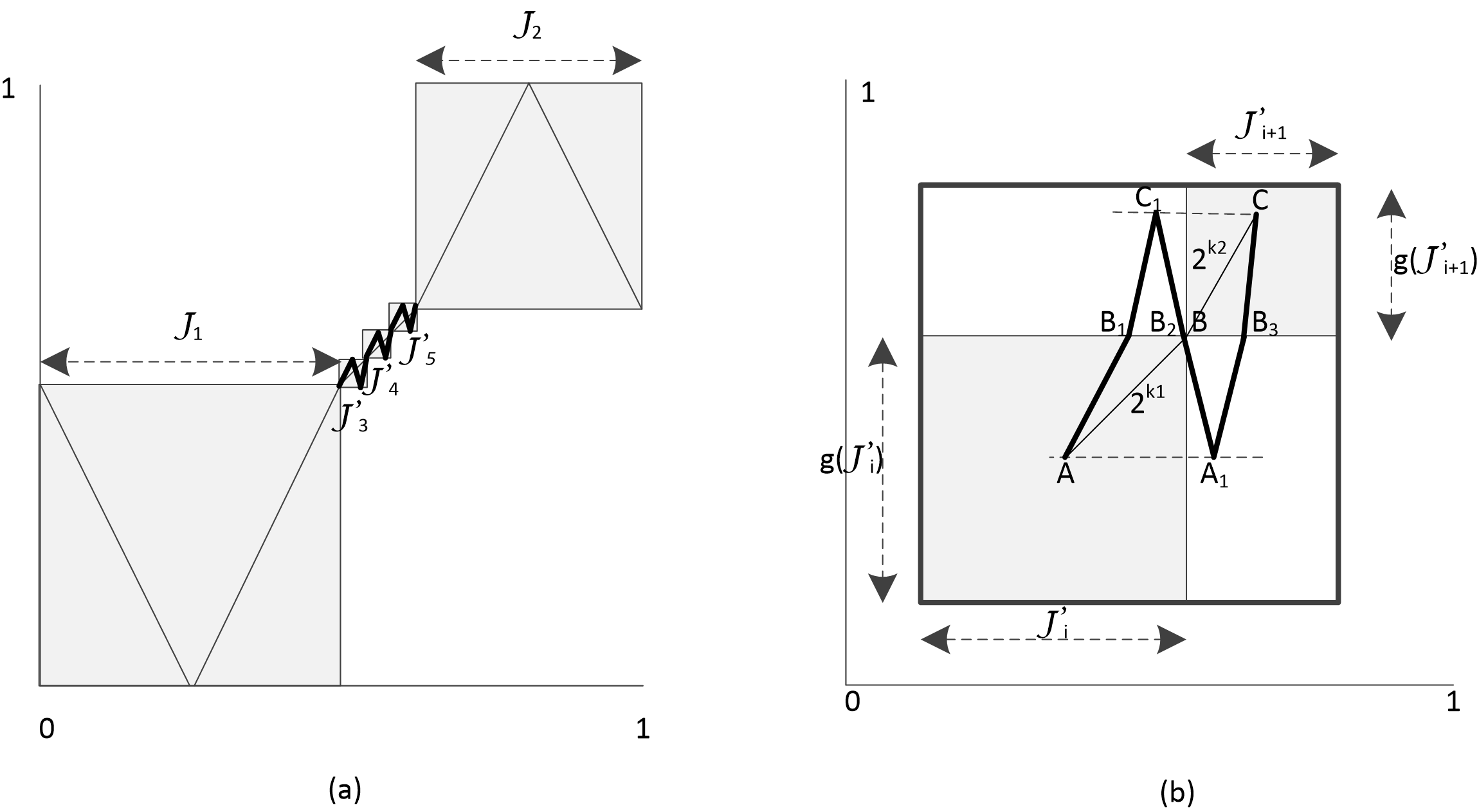}
  \caption{Perturbation in the proof of Theorem~\ref{theorem:GLEOdense}. (a) Add new $\mathcal{J}'_3, \mathcal{J}'_4, \mathcal{J}'_5$ such that $[0,1]=\bigcup_{i\ge1} \mathcal{J}'_i$. (b) Merge $\mathcal{J}_i$ and $\mathcal{J}_{i+1}$ into one interval by replacing the thin segments $AB, BC$ with thick segments $AB_1, B_1C_1, C_1B_2, B_2A_1, A_1B_1, B_1C$.}
  \label{fig:WeeklyDiscussion_20200420Page-4}
\end{figure} 

\emph{Step $1$}. Add new intervals to the set $\{\mathcal{J}_1, \mathcal{J}_2, \ldots\}$ such that the new set $\{\mathcal{J}'_1, \mathcal{J}'_2, \ldots\}$ covers $[0,1]$ in $g_{\text{new}}$ as illustrated in Figure~\ref{fig:WeeklyDiscussion_20200420Page-4}(a). By Lemma~\ref{lemma:Jcollection1}, if $x\in[0,1]\setminus \bigcup_{i\ge1}\mathcal{J}_i^{\circ}$, then the derivative of $g$ is either $1$ or $-1$. Divide  $[0,1]\setminus \bigcup_{i\ge1}\mathcal{J}_i^{\circ}$ into a number of intervals with mutually disjoint interiors and each has an maximum length smaller than $\frac{\epsilon}{2}$ and dyadic endpoints. Such a division exists because the endpoints of $\{\mathcal{J}_i\}$ are all dyadic by Lemma~\ref{lemma:dyadicJenpoints}. If the affine segment on one of the intervals has the slope $1$, then replace it with a $3$-fold window perturbation of slopes $2^{1}, -2^{2}, 2^{2}$ on the three legs respectively; otherwise, the affine segment has the slope $-1$ and replace it with a $3$-fold window perturbation of slopes $-2^{1}, 2^{2}, -2^{2}$ on the three legs respectively. 

After step $1$, $g_{\text{new}}\in\mathbb{G}$ and $\rho(g_{\text{old}}, g_{\text{new}})<\frac{\epsilon}{2}$. Combine the original $\mathcal{J}_1, \mathcal{J}_2, \ldots$ and the newly added intervals to become $\{\mathcal{J}'_1, \mathcal{J}'_2, \ldots\}$. $\bigcup_{i\ge1}\mathcal{J}'_i=[0,1]$.

\emph{Step $2$}. Merge adjacent disjoint intervals. Consider adjacent disjoint intervals $\mathcal{J}'_i$ and $\mathcal{J}'_{i+1}$. Let $B$ be on the graph of $g$ at the boundary point between $\mathcal{J}'_i$ and $\mathcal{J}'_{i+1}$. The graph of $g$ is an affine segment in a sufficiently small left and right neighborhood of $B$ and the left and the right derivatives are of the same sign by Lemma~\ref{lemma:Jcollection1}. Suppose that the derivatives are both positive as shown in Figure~\ref{fig:WeeklyDiscussion_20200420Page-4}(b). (The case of the derivatives being both negative can be proven analogously.) Let the left derivative be $2^{k_1}$ and the right derivative be $2^{k_2}$. Let $AB$ be the affine segment of $g$ in $\mathcal{J}'_i$ and $BC$ be the affine segment of $g$ in $\mathcal{J}'_{i+1}$ where $B_y-A_y=C_y-B_y=2^{-M}$ for a large positive integer $M$ such that $2^{-M}<\min\left(\frac{\epsilon}{4},\frac{|\mathcal{J}'_{i}|}{2},\frac{|\mathcal{J}'_{i+1}|}{2}\right)$ and no breakpoint exists on $(A_x, B_x)$ or $(B_x, C_x)$. $B_x-A_x = (B_y-A_y)\cdot 2^{-k_1}=2^{-M-k_1}, C_x-B_x = (C_y-B_y)\cdot 2^{-k_1}=2^{-M-k_2}$. Thus $A$ and $C$ are both dyadic, because $B$ is dyadic.

Replace segments $AB$ and $BC$ by the following six affine segments to merge $\mathcal{J}'_{i}$ and $\mathcal{J}'_{i+1}$: $AB_1$, $B_1C_1$, $C_1B_2$, $B_2A_1$, $A_1B_3$ and $B_3C$. The connecting points $A_1, C_1, B_1, B_2, B_3$ are defined as follows: $C_{1,y}=C_y, A_{1,y}=A_y, B_{1,y}=B_{2,y}=B_{3,y}=B_{y}$ and $B_{1,x}-A_x=2^{-M-k_1-1}, C_{1,x}-B_{1,x}=2^{-M-k_2-1}, B_{2,x}-C_{1,x}=2^{-M-k_2-2}, A_{1,x}-B_{2,x}=2^{-M-k_1-2}, B_{3,x}-A_{1,x}=2^{-M-k_1-2}, C_{x}-B_{3,x}=2^{-M-k_2-2}$. It is easy to verify that $g$ is still $\lambda$-preserving, the absolute values of the slopes of the affine segments are $2^{k_1+1}, 2^{k_1+2}, 2^{k_2+1}, 2^{k_2+2}$, and the newly added breakpoints $A, A_1, B_1, B_2, B_3, C, C_1$ are all dyadic. Therefore, $g_{\text{new}}\in\mathbb{G}$.

Repeat the preceding procedure for all $i$. The choice of $M$ ensures that the perturbation done for all $i$ does not overlap and $\rho(g_{\text{old}}, g_{\text{new}})<\frac{\epsilon}{2}$ in step $2$. After step $2$, $\{\mathcal{J}'_1, \mathcal{J}'_2, \ldots\}$ are all merged into $\{[0,1]\}$.

Hence, after the preceding two steps of perturbation, $\rho(g_{\text{old}}, g_{\text{new}})<\epsilon$. By Lemma~\ref{lemma:Jcollection2}, $g$ is TM, and by Theorem~\ref{theorem:GLEOMT}, $g$ is LEO.
\end{proof}

\begin{figure}
 \centering
  \includegraphics[width=8cm]{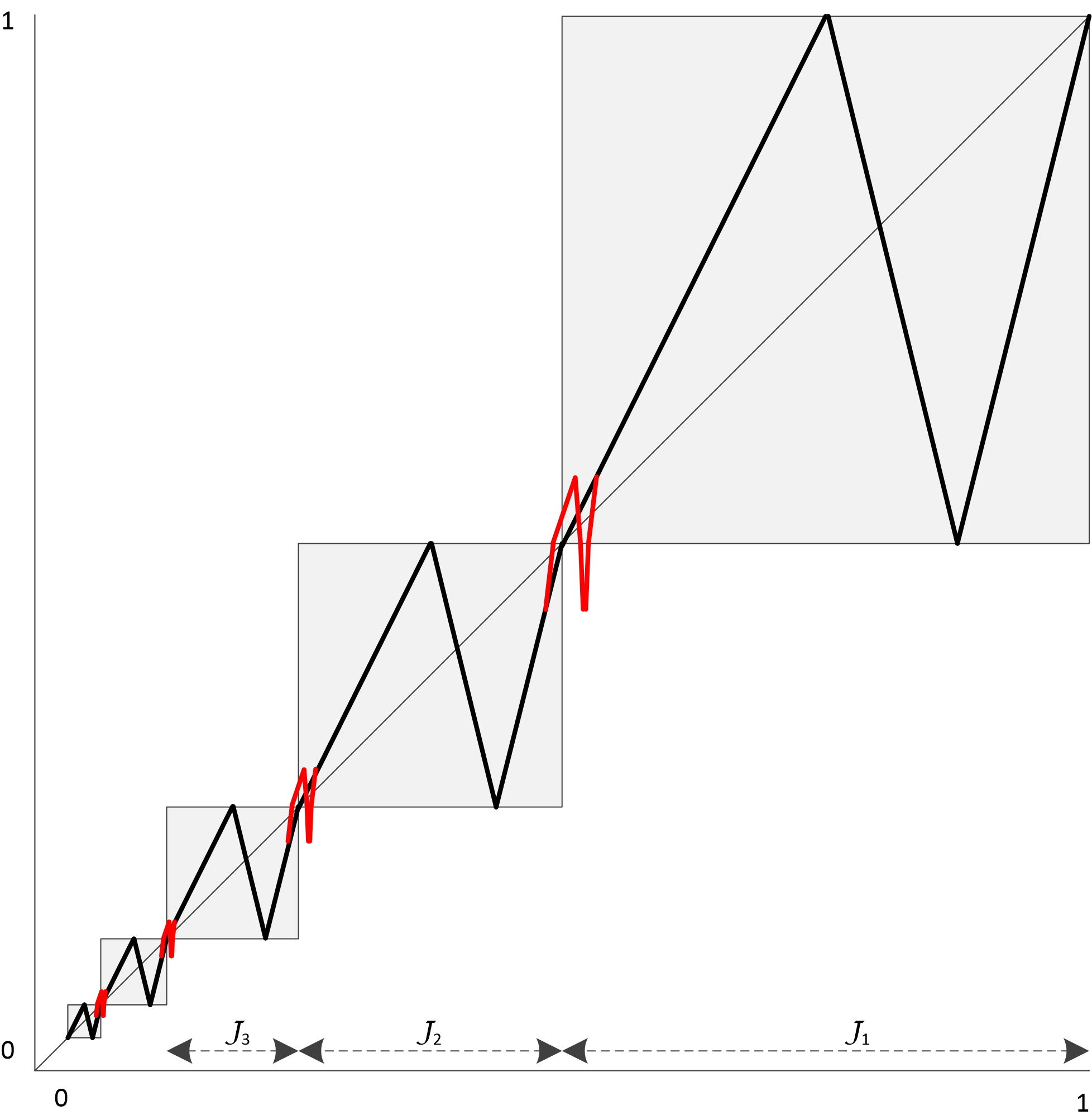}
  \caption{An example of $h$ that is TM but is not LEO. First, replace the thin segment $h(x)=x$ with the black thick segments, which are the window perturbations on intervals $\mathcal{J}_1, \mathcal{J}_2, \ldots$. Then, replace the black thick segments with the red thick segments at the boundaries between  $\mathcal{J}_i$ and $\mathcal{J}_{i+1}$ for $i=1, 2, \ldots$ to be TM.}
  \label{fig:WeeklyDiscussion_20200420Page-7}
\end{figure} 

Figure~\ref{fig:WeeklyDiscussion_20200420Page-7} shows an example of $h$ that is TM but is not LEO. Specifically, partition $[0,1]$ into countably infinitely many subintervals $\{\mathcal{J}_i\}$, where $\mathcal{J}_i=[2^{-i},2^{-i+1}]$ for $i=1, 2, \ldots$. Start with $h(x)=x$. Next replace $h(x)$ on $\mathcal{J}_i$ with a $3$-fold window perturbation, shown as the thick black segments. Then merge $\mathcal{J}_i$ and $\mathcal{J}_{i+1}$ as in step $2$ of the proof of Theorem~\ref{theorem:GLEOdense} illustrated in Figure~\ref{fig:WeeklyDiscussion_20200420Page-4}(b). The merge is shown as the thick red segments in Figure~\ref{fig:WeeklyDiscussion_20200420Page-7}. Let $h(0)=0$. $h^{-1}(0)=\{0\}$. The endpoint $x=0$ is not accessible and therefore $h$ is not LEO. Recall that endpoint $x=0$ is accessible if there exists $x\in(0,1)$ such that $h^n(x)=0$ for some $n>0$. Note that this map $h$ is not an element of $\mathbb{G}$, because there are infinitely many points at which $h$ is not differentiable, although $h$ meets all the other conditions of $\mathbb{G}$. Therefore this counterexample does not contradict Theorem~\ref{theorem:GLEOMT}.

\section{Periodicity \label{sec:Markov}}

Theorem~\ref{theorem:allperiodicpoints} states a salient feature of $\mathbb{G}$.
\begin{theorem}
Let $g\in \mathbb{G}$ and $c$ be a dyadic number. Then point $(c,g(c))$ is preperiodic under the diagonal action $(x,y)\rightarrow (g(x), g(y))$.
\label{theorem:allperiodicpoints}
\end{theorem}
\begin{proof}
Let $0=x_0<\cdots<x_n=1$ be all the breakpoints of $g$. Let $x_i=\frac{l_i}{2^M}$ for $i=0, 2, \ldots, n$, $c=\frac{p_{0}}{2^M}$ and $g(c)=\frac{p_{1}}{2^M}$ for integers $M$ and $l_i, p_{0}, p_{1}$. Integers $l_i, p_{0}, p_{1}$ are not necessarily odd. 

Let $\{x_{b_1}, x_{b_2}, \ldots, x_{b_m}\}$ be the subset of breakpoints between $c$ and $g(c)$, inclusive, where $m\ge0$. Between $x=c$ and $x=g(c)$ there are $m+1$ affine segments, each with a horizontal length in the form of $\frac{l}{2^M}$ and a slope in the form of $\pm 2^{k}$ for some integers $l,k$. The vertical displacement of any affine segment is in the form of $\pm \frac{l\cdot2^{k}}{2^M}$. The sum of the vertical displacements of these $m+1$ affine segments, equal to $g(g(c))-g(c)$, is in the form of $\frac{l}{2^M}$ for some integer $l$. Given that $g(c)=\frac{p_{1}}{2^M}$, it follows that $g(g(c))=\frac{p_{2}}{2^M}$ for some integer $p_{2}$. 

Repeating the preceding argument, it follows that $g^i(c)=\frac{p_{i}}{2^M}$ for integer $p_{i}$ for all $i=0, 1, 2, \ldots$, with $0\le p_i\le 2^M$. Because $M$ is a finite number, the total number of distinct $\frac{p_{i}}{2^M}$ in $[0,1]$ is finite. Hence, $g^{i_1}(c)=g^{i_2}(c)$ for some $i_1 \neq i_2$, and $(c,g(c))$ is a preperiodic point.
\end{proof}

\begin{definition}[\textsl{Markov Map}]
A piecewise affine interval map is a Markov map if all breakpoints are preperiodic.
\end{definition}
By definition, any breakpoint of $g\in\mathbb{G}$ is dyadic and thus preperiodic by Theorem~\ref{theorem:allperiodicpoints}. The following corollary follows immediately.
\begin{corollary}
Any $g\in \mathbb{G}$ is a Markov map.
\label{corollary:markovmap}
\end{corollary}

By Theorem~\ref{theorem:GLEOdense} and Corollary~\ref{corollary:markovmap}, $\mathbb{G}$ that is both LEO and Markov is dense in $C(\lambda)$. Because $\mathbb{G}$ is a subset of $PA(\lambda)$, this result is stronger than \cite[Proposition.\ 7]{2019arXiv190607558B}, which shows that $PA(\lambda)$ that are both LEO and Markov is dense in $C(\lambda)$. Corollary~\ref{corollary:markovmap} provides an essential basis of the study of topological conjugacy in Section~\ref{sec:conjugacy}.

\begin{definition}[\textsl{Period of a Point}]
Suppose that $x$ is a periodic point. The period of $x$ is the least positive integer $p$ such that $h^p(x)=x$. 
\end{definition}
\begin{definition}[\textsl{Chaotic Function}]
A map $h$ is called chaotic if there exists a point $x$ of period $k$ for any positive integer $k$. 
\end{definition}
Li-Yorke theorem \cite{10.2307/2318254} states that if a periodic point $x$ of period $3$ exists, then $h$ is chaotic. Periodic points of period $3$ are thus of particular importance. The remainder of this section is to characterize the periods of periodic points of $g\in\mathbb{G}$. 

\begin{theorem}
Consider a continuous map $h$ from $[0,1]$ onto itself. If intervals $\mathcal{I}_0, \mathcal{I}_1, \mathcal{I}_2 \subseteq[0,1]$ exist such that $\mathcal{I}_1\subset \mathcal{I}_0, \mathcal{I}_2\subset \mathcal{I}_0$, $\mathcal{I}^{\circ}_1\cap \mathcal{I}^{\circ}_2=\emptyset$, and $h(\mathcal{I}_1)=h(\mathcal{I}_2)=\mathcal{I}_0$, then a periodic point $x_0\in\mathcal{I}_0$ of period $3$ exists.
\label{theorem:period3}
\end{theorem}
\begin{proof}
Because $\mathcal{I}_1 \subset \mathcal{I}_0=h(\mathcal{I}_2)$, an interval $\mathcal{I}_3\subset \mathcal{I}_2$ exists such that $\mathcal{I}_1=h(\mathcal{I}_3)$. Because $\mathcal{I}_3 \subset h(\mathcal{I}_2)$, an interval $\mathcal{I}_4 \subset \mathcal{I}_2$ exists such that $\mathcal{I}_3=h(\mathcal{I}_4)$. Therefore, $\mathcal{I}_0=h^3(\mathcal{I}_4)$. Because $\mathcal{I}_4\subset\mathcal{I}_0$, by the Intermediate Value Theorem, $x_0\in\mathcal{I}_4$ exists such that $h^3(x_0)=x_0$. Specifically, $x_1\in\mathcal{I}_3$ and $x_2\in\mathcal{I}_1$ exist such that $x_1=h(x_0), x_2=h(x_1)$ and $x_0=h(x_2)$. Moreover, $\mathcal{I}_3\cap\mathcal{I}_1=\emptyset$ because $\mathcal{I}_3\subset \mathcal{I}_2$ and $\mathcal{I}^{\circ}_1\cap \mathcal{I}^{\circ}_2=\emptyset$. Then, it follows that $\mathcal{I}_3\cap\mathcal{I}_4=\emptyset$ because $\mathcal{I}_1=h(\mathcal{I}_3)$ and $\mathcal{I}_3=h(\mathcal{I}_4)$. Thus $x_0, x_1, x_2$ are all distinct. Hence, the period of $x_0$ is $3$. The proof is illustrated in Figure~\ref{fig:WeeklyDiscussion_20200420Page-8}.
\end{proof}

\begin{figure}
 \centering
  \includegraphics[width=8cm]{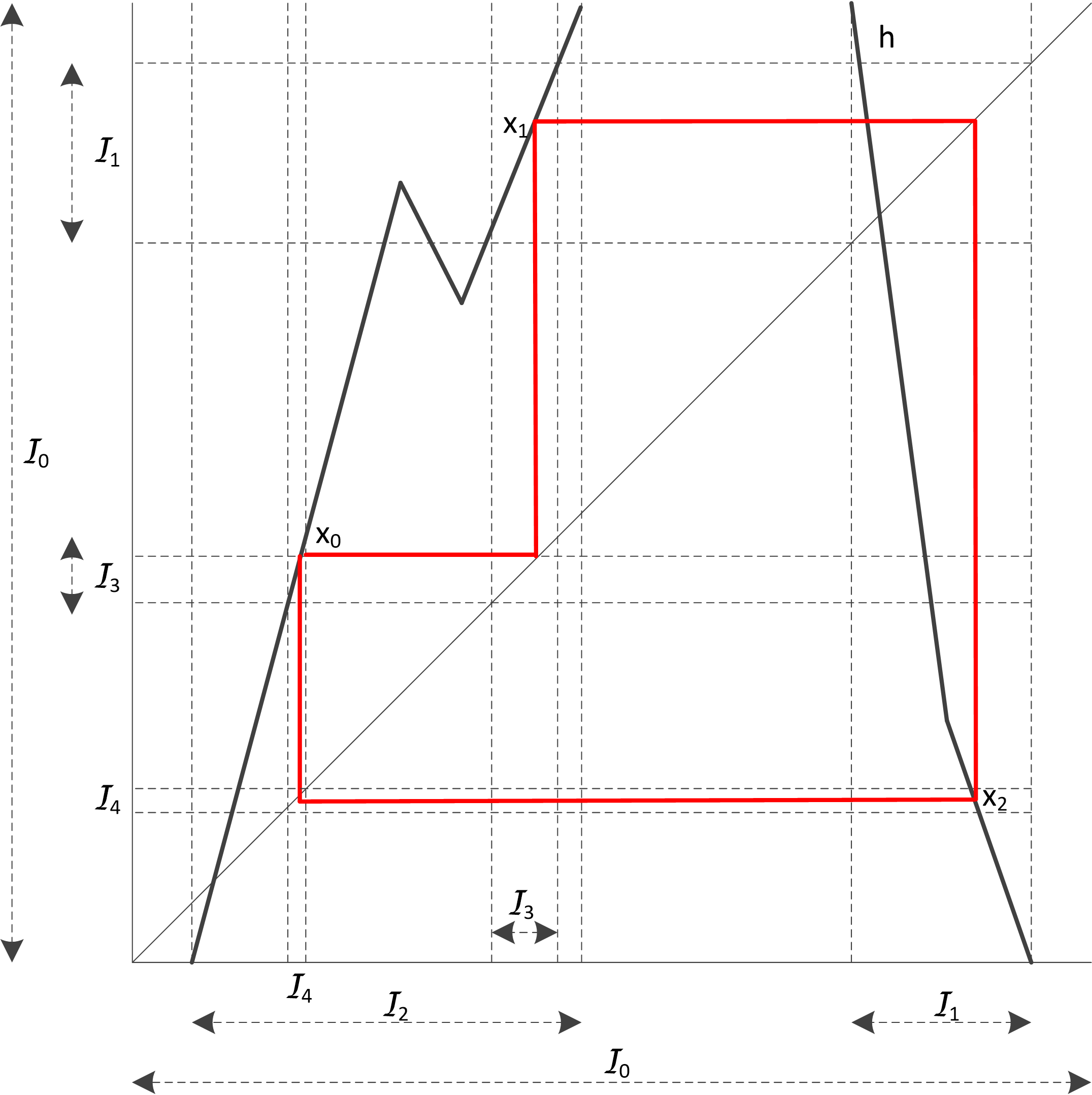}
  \caption{Proof of Theorem~\ref{theorem:period3}. The thick black lines represent $h$ and the thick red lines show the period-$3$ trajectory of $x_0 \xrightarrow{h} x_1\xrightarrow{h} x_2 \xrightarrow{h} x_0$.}
  \label{fig:WeeklyDiscussion_20200420Page-8}
\end{figure} 

\begin{corollary}
Consider a $\lambda$-preserving continuous map $h$ from $[0,1]$ onto itself. Suppose that an interval $\mathcal{J}$ exists such that $h(\mathcal{J})=\mathcal{J}$ and $h^{-1}(\mathcal{J})=\mathcal{J}$. Let $c$ 
be an endpoint of $\mathcal{J}$. If $h(c)=\mathcal{J}^0$ or $h(c)=\mathcal{J}^1$, and if $d_0\in\mathcal{J}$ with $d_0\neq c$ exists such that $h(d_0)=h(c)$, then a periodic point $x_0\in\mathcal{J}$ of period $3$ exists.
\label{corollary:period3}
\end{corollary}

\begin{proof}
Without loss of generality, suppose that $c=\mathcal{J}^0$. Suppose that $h(c)=\mathcal{J}^0$. The case of $h(c)=\mathcal{J}^1$ can be proved analogously. 

From the hypothesis, there exists $d_1\in\mathcal{J}$ such that $h(d_1)=\mathcal{J}^1$. If $d_0>d_1$, then let $\mathcal{I}_1=[c,d_1]$, $\mathcal{I}_2=[d_1,d_0]$ and $\mathcal{I}_0=\mathcal{J}$. The conclusion follows from Theorem~\ref{theorem:period3}. Otherwise, $d_0<d_1$. Note that $h([c,d_0])\supseteq [c,d_0]$, because otherwise $h^{-1}([d_0,\mathcal{J}^1]) \subset [d_0,\mathcal{J}^1]$ and $\lambda$ is not preserved. Thus, there exist $c'$ and $d'_0$ with $c<c'\le d'_0<d_0$ such that $h([c,c'])=h([d'_0,d_0])=[c,d_0]$. Let $\mathcal{I}_1=[c,c']$, $\mathcal{I}_2=[d'_0,d_0]$ and $\mathcal{I}_0=[c,d_0]$. The conclusion follows from Theorem~\ref{theorem:period3}. 
\end{proof}

Let $g\in\mathbb{G}$. Suppose that a collection of $\{\mathcal{J}_1, \mathcal{J}_2, \ldots\}$ of $g$ exist to have the properties stated in Lemma~\ref{lemma:Jcollection1}. 

\emph{Case $1$}. Suppose that either $[0,1]\setminus \bigcup_{i\ge1}\mathcal{J}_i^{\circ}\neq\emptyset$ and $g(x)=x$ on $[0,1]\setminus \bigcup_{i\ge1}\mathcal{J}_i^{\circ}$, or $|\{\mathcal{J}_1, \mathcal{J}_2, \ldots\}|>1$ and $g(\mathcal{J}_i)=\mathcal{J}_i$ for all $i$. Then an interval $\mathcal{J} \subseteq \mathcal{J}_i$ for any $i$ exists such that the hypothesis of Corollary~\ref{corollary:period3} holds with $h(c)=c$. Hence, a periodic point of period $3$ exists.

\emph{Case $2$}. Suppose that either $[0,1]\setminus \bigcup_{i\ge1}\mathcal{J}_i^{\circ}\neq\emptyset$ and $g(x)=1-x$ on $[0,1]\setminus \bigcup_{i\ge1}\mathcal{J}_i^{\circ}$,  or $|\{\mathcal{J}_1, \mathcal{J}_2, \ldots\}|>1$ and $g(\mathcal{J}_i)>g(\mathcal{J}_j), \forall \mathcal{J}_i<\mathcal{J}_j$. If $g(\mathcal{J}_i)\neq\mathcal{J}_i, \forall i$, then the period of any periodic point is even, because $g^2(\mathcal{J}_i)=\mathcal{J}_i$. Otherwise, $\exists i_0$ such that $g(\mathcal{J}_{i_0})=\mathcal{J}_{i_0}$. In this case, the period of any periodic point $x$ is even when $x\in\mathcal{J}_i$ with $i\neq i_0$, and an interval $\mathcal{J} \subseteq \mathcal{J}_{i_0}$ exists such that the hypothesis of Corollary~\ref{corollary:period3} holds with $h(c)=1-c$. Hence, a periodic point of period $3$ exists.

\emph{Case $3$}. Suppose that $\{\mathcal{J}_1, \mathcal{J}_2, \ldots\}=\{[0,1]\}$. This case is of particular importance because of Lemma~\ref{lemma:Jcollection2}. Two results have been shown for an interval map $h$ in the literature. 
\begin{itemize}
\item
First \cite[Proposition.\ 2.18]{ruette2015chaos}, if $h$ is transitive, then $h$ is TM if and only if it has a periodic point of odd period greater than $1$. 
\item
Second, define Sharkovsky's order of positive integers by 
\[
3 \lhd 5 \lhd 7 \lhd \cdots \lhd 2\cdot 3 \lhd 2\cdot 5 \lhd 2\cdot 7 \lhd \cdots \lhd 2^2\cdot 3 \lhd 2^2\cdot 5 \lhd 2^2 \cdot 7 \lhd \cdots \lhd 2^3 \lhd 2^2 \lhd 2 \lhd 1.
\]
Sharkovsky's theorem \cite{Sarkovskii1964} states that if $h$ has a periodic point of period $n$, then $h$ has periodic points of period $m$ for all integers $m\rhd n$. Because $3 \lhd n$ for any $n\neq 3$, Li-Yorke theorem is a specific case of Sharkovsky's theorem. 
\end{itemize}
From these two results, it follows that in the case where $\{\mathcal{J}_1, \mathcal{J}_2, \ldots\}=\{[0,1]\}$, there exists an odd number $n_0$ such that periodic points of period $n$ exist for any odd number $n\ge n_0$, no periodic points of period $n$ exist for any odd number $1<n< n_0$, and periodic points of period $n$ exist for $n=1$ and any even number $n$. 

\begin{example}
Let
\begin{equation}
g(x)=\left\{
\begin{array}{ll}
4x+\frac{1}{2}-\delta, & 0\le x <\frac{1}{2}\delta\\
2x+\frac{1}{2}, & \frac{1}{2}\delta \le x <\frac{1}{4}\\
-2x+\frac{3}{2}, & \frac{1}{4}\le x < \frac{3}{4}\\
2x-\frac{3}{2}, & \frac{3}{4}\le x < 1-\frac{1}{2}\delta \\
4x-\frac{7}{2}+\delta, & 1-\frac{1}{2}\delta \le x \le 1
\end{array}
\right.
\label{eq:example_delta_n0}
\end{equation}
for $0<\delta<\frac{1}{2}$. It can be shown that $\delta$ exists for any target $n_0$. The smaller target value of $n_0$, the larger value of $\delta$ is needed. Two examples are shown in Figure~\ref{fig:Report_finalfig1}.

\begin{figure}
\centering
\begin{subfigure}{.49\textwidth}
  \centering
  \includegraphics[width=1.\linewidth]{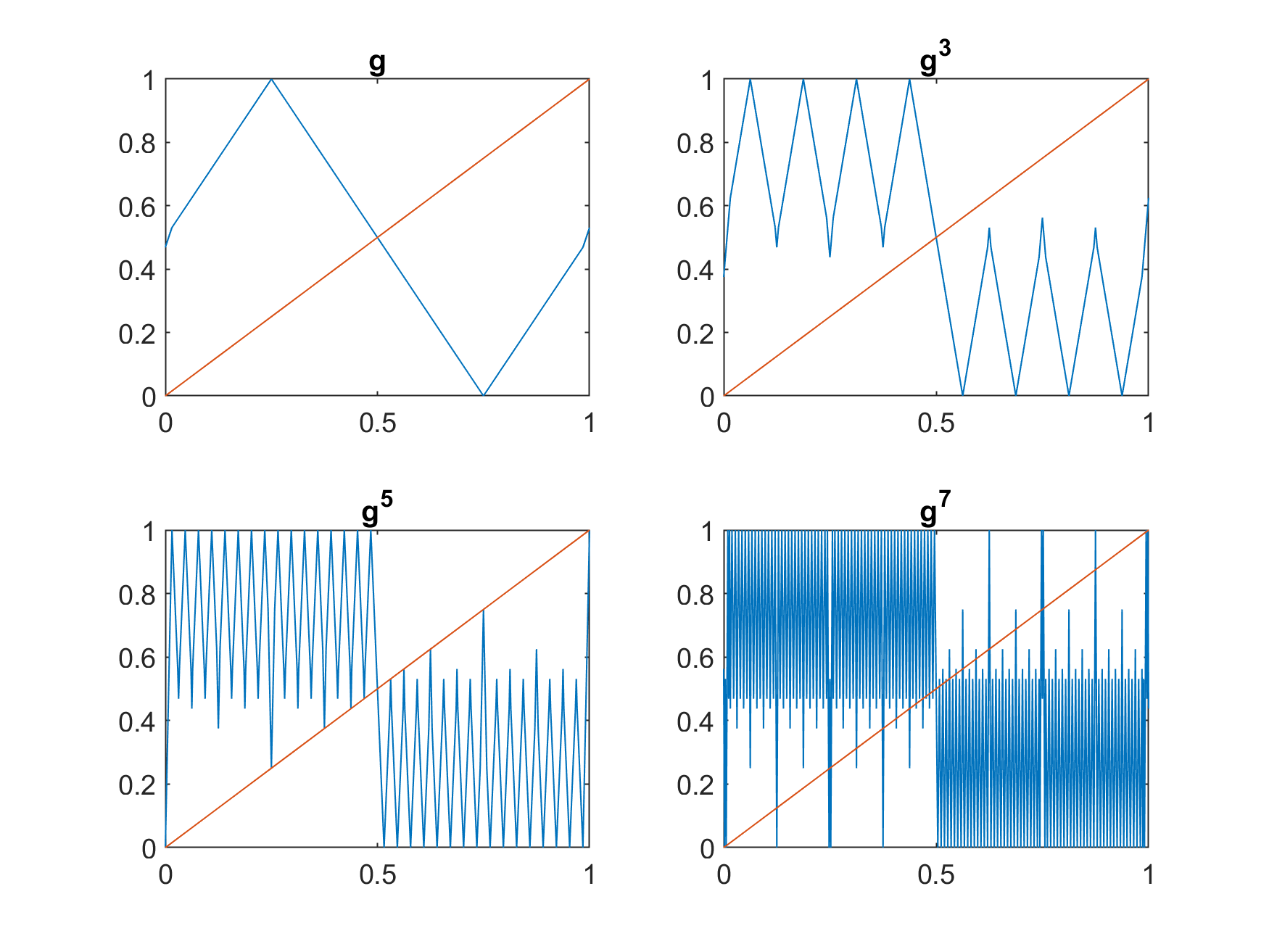}
    \caption{}
\end{subfigure}
\begin{subfigure}{.49\textwidth}
  \centering
  \includegraphics[width=1.\linewidth]{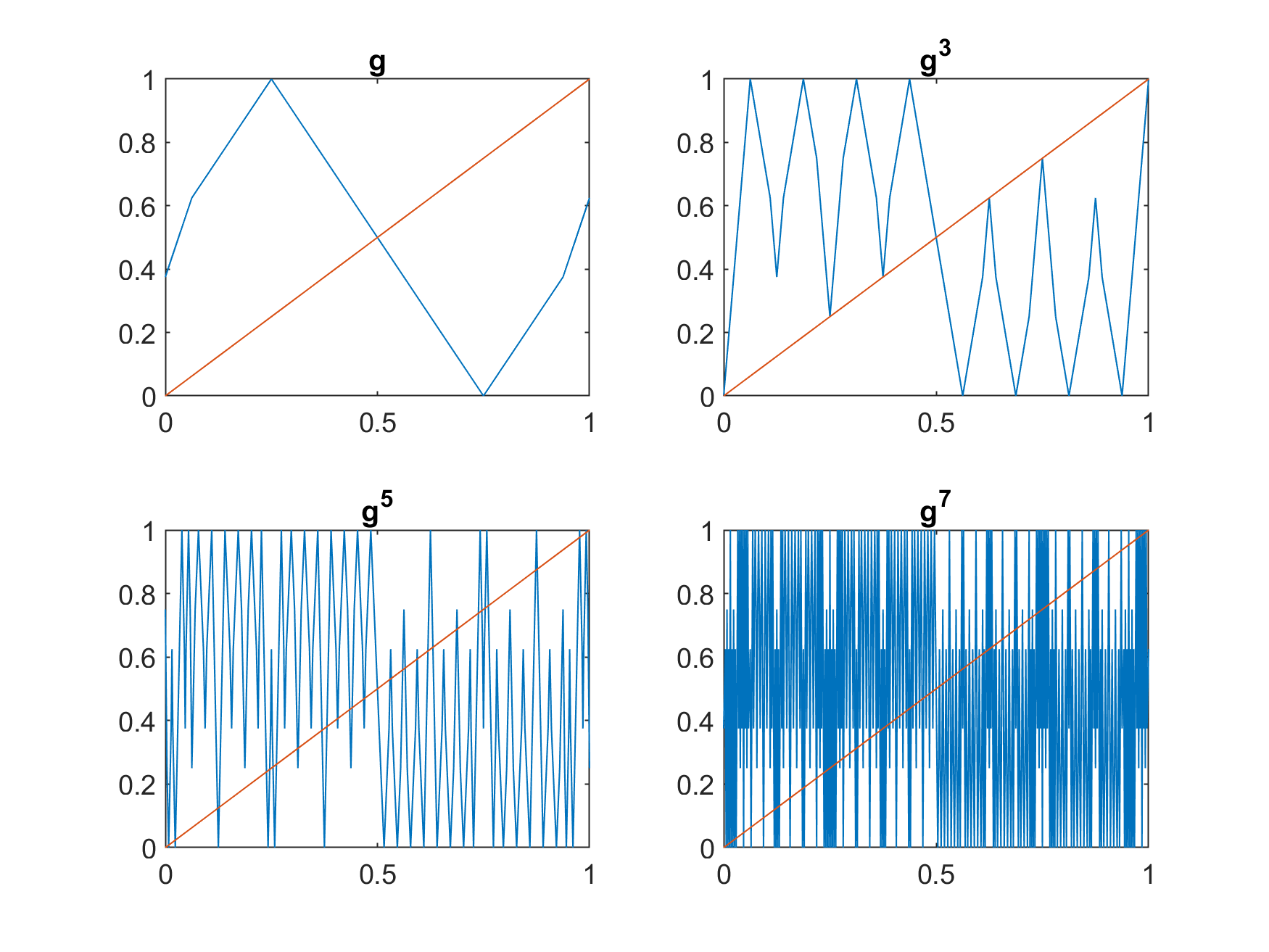}
    \caption{}
\end{subfigure}
\caption{Two examples of $g(x)$ defined in (\ref{eq:example_delta_n0}). Point $x_0$ is a periodic point of period $i$ if the graph of $g^i(x)$ intersects the red line representing $y=x$ at $x=x_0$, and if $x_0$ is not a periodic point of any period $j<i$. In (a), $\delta=2^{-5}$. Periodic points of period $3$ do not exist while periodic points of periods $5$ and $7$ exist. In (b), $\delta=2^{-3}$. Periodic points of period $3, 5, 7$ all exist. }
\label{fig:Report_finalfig1}
\end{figure}
\end{example}

\section{Entropy \label{sec:entropy}}

\begin{definition}[\textsl{Entropy}]
The entropy of a map $h$ on $\mathcal{I}$ where $\mathcal{I}\subset [0,1]$ is 
\[
c_{\lambda,\mathcal{I}}(h) = \int_{\mathcal{I}^0}^{\mathcal{I}^1} \log_2 |h'(x)| \,d \lambda(x).
\]
When $\mathcal{I}=[0,1]$, simplify notation $c_{\lambda, [0, 1]}(h)$ to $c_{\lambda}(h)$.
\end{definition}

Let $g\in\mathbb{G}$. Suppose that for interval $\mathcal{Y}$, $g^{-1}(\mathcal{Y})$ consists of $m$ affine legs on intervals $\mathcal{I}_i$ for $i=1, 2, \ldots, m$. Let $2^{k_i}$ be the absolute value of the slope of the affine segment of $g$ on interval $\mathcal{I}_i$. By definition, 
\[
\sum_{i=1}^m c_{\lambda, \mathcal{I}_i}(g)=\sum_{i=1}^m k_i \left|\mathcal{I}_i\right|=\sum_{i=1}^m k_i 2^{-k_i}\left|\mathcal{Y}\right|.
\]
The minimum value of entropy $\sum_{i=1}^m c_{\lambda, \mathcal{I}_i}(g)$ given $|\mathcal{Y}|$ is obtained when $\{k_i\}$ solves the following integer optimization
\begin{eqnarray}
&&\min_{k_1, \ldots, k_m} \sum_{i=1}^m k_i 2^{-k_i} \label{eq:minentropy}\\
&&\mbox{s.t. }\sum_{i=1}^m 2^{-k_i}=1. \label{eq:lambdaconservation}
\end{eqnarray}

\begin{lemma}
The solution to (\ref{eq:minentropy}) and (\ref{eq:lambdaconservation}) is 
\begin{equation}
k_i=\left\{
\begin{array}{cc}
i, i=1, 2, \ldots, m-1\\
m-1, i=m.
\end{array}
\right.
\label{eq:minimumentropysolution}
\end{equation}
\label{lemma:minimumentropysolution}
\end{lemma}
\begin{proof}
We first treat $k_i$ as continuous variables. Define the Lagrangian map 
\[
L(k_1, \ldots, k_m, \mu)=\sum_{i=1}^m k_i 2^{-k_i}+\mu \left(\sum_{i=1}^m 2^{-k_i}-1\right).
\]
Setting $\frac{\partial L}{\partial k_i}=0$ leads to $k_i$ be equal for all $i$, which is the interior solution and actually maximizes, rather than minimizing, the entropy. Next we check the boundary conditions. Without loss of generality, suppose that $k_1\le k_2\le\cdots\le k_m$. It is easy to see from (\ref{eq:lambdaconservation}) that $k_1\ge1$ and $k_m\le m-1$. If $k_1=1$, then the original problem is reduced to $\min \sum_{i=2}^m k_i 2^{-k_i}$ subject to $\sum_{i=2}^m 2^{-k_i}=\frac{1}{2}$ and we can continue to check the boundary conditions. If $k_m= m-1$, then $k_{m-1}=m-1, k_{m-2}=m-2, k_{m-3}=m-3, \ldots, k_1=1$. Either boundary condition leads to the same solution (\ref{eq:minimumentropysolution}), which minimizes the entropy.
\end{proof}

The set of $\{k_i\}$ (\ref{eq:minimumentropysolution}) not only results in the minimum entropy given $m$ but also will be shown to be the unique set that possesses the property of Lemma~\ref{lemma:klswitch}, which is a stepping stone to proving Theorem~\ref{theorem:denseG}. Before stating Lemma~\ref{lemma:klswitch}, we take a detour to solve a dynamic matching problem. 

\begin{problem}[Dynamic Matching]
Consider $m$ buckets and $m$ pumps. Pump $i$ has a fixed pumping rate $\alpha_i>0$. At any time, exactly one pump is pumping water into one bucket. The matching policy at time $t$ can be represented by a permutation $\Pi_t: \{1,\ldots,m\}\rightarrow\{1,\ldots,m\}$. Initially all buckets are empty at time $t=0$. The task is to find a dynamic matching policy $\Pi_t$ at any time $t\in[0,1]$ so that at $t=1$, bucket $i$ has $\beta_i>0$ amount of water. By the law of conservation, $\sum_{i=1}^m \beta_i = \sum_{i=1}^m \alpha_i$. 
\label{problem:dynamicmatching}
\end{problem}

Without loss of generality, suppose that $\alpha_1\ge\alpha_2\cdots\ge\alpha_m$ and $\beta_1\ge\beta_2\cdots\ge\beta_m$. Not any $\{\alpha_i\}$ and $\{\beta_i\}$ have a solution. For example, if $\beta_m <\alpha_m$ or $\beta_1 >\alpha_1$, then no solution exists  because the target of bucket $m$ or $1$, respectively, cannot be met. The following lemma states a necessary and sufficient condition of the existence of a solution. 

\begin{lemma}
A solution of Problem~\ref{problem:dynamicmatching} exists if and only if $\sum_{j=1}^i(\alpha_j-\beta_j)\ge0$ for all $i$.
\label{lemma:dynamicmatching}
\end{lemma}
\begin{proof}
If $\alpha_1-\beta_1<0$, then no solution exists to meet the target of bucket $1$. The following assumes $\alpha_1-\beta_1\ge0$. Let pump $1$ only serve bucket $1$. If $\alpha_1-\beta_1>0$, the arrangement of pump $1$ and bucket $1$ is tentative, as the fraction of time that pump $1$ serves bucket $1$ will be adjusted in subsequent steps. The quantity $\alpha_1-\beta_1$ represents the excess capacity due to the capacity of pump $1$ in excess of the target of bucket $1$.

Now consider pump $2$ and bucket $2$. If $\sum_{j=1}^2 (\alpha_j-\beta_j)<0$, the excess capacity $\alpha_1-\beta_1$ is insufficient to cover the shortage of $\alpha_2-\beta_2$. Moreover, from $\alpha_2-\beta_2<0$ it follows that $\beta_2>\alpha_2\ge\alpha_i$ for $i> 2$. No solution exists to meet the target of bucket $2$. The following assumes $\sum_{j=1}^2 (\alpha_j-\beta_j)\ge0$. There are two cases. In the first case, $\beta_2\ge\alpha_2$. Utilize the excess capacity $\alpha_1-\beta_1$ to cover the shortage of $\alpha_2-\beta_2$. Specifically, let bump $2$ serve bucket $2$ for time interval $1-z$ while bump $1$ serves bucket $1$, and then bump $1$ serve bucket $2$ for time interval $z$ while bump $2$ serves bucket $1$, where $z$ solves $z\alpha_1+(1-z)\alpha_2 = \beta_2$. Because $\alpha_2\le\beta_2\le\beta_1\le\alpha_1$, a unique solution $z$ exists with $0\le z\le1$.  All the arrangements regarding pump $2$ and bucket $2$ are final because pump $2$ has no excess capacity to offer and bucket $2$ has met its target. The arrangement of pump $1$ and bucket $1$ is still tentative if $\sum_{j=1}^2 (\alpha_j-\beta_j)>0$ because they have excess capacity to offer. In the second case, $\beta_2<\alpha_2$. Let bump $2$ only serve bucket $2$. The arrangements of pump $2$ and bucket $2$ is also tentative as they have excess capacity to offer just like pump $1$ and bucket $1$. In either case, $\sum_{j=1}^2 (\alpha_j-\beta_j)$ represent the accumulated excess capacity due to the total capacity of pumps $1$ and $2$ in excess of the total target of buckets $1$ and $2$. 

Continue the preceding process for $i=3, \ldots, m$. If $\sum_{j=1}^i (\alpha_j-\beta_j)<0$ for some $i$, then no solution exists to meet the target of bucket $i$, because the accumulated excess capacity from $1, 2, \ldots, i-1$, i.e., $\sum_{j=1}^{i-1} (\alpha_j-\beta_j)$, is insufficient to cover the shortage of $(\alpha_i-\beta_i)$. Otherwise, if $\beta_i\ge\alpha_i$, then utilize the accumulated excess capacity to cover the shortage of $(\alpha_i-\beta_i)$ by letting a subset of pumps $1, \ldots, i-1$, whose arrangements have so far been tentative, to serve bucket $i$ for some fraction of time interval to meet its target while letting pump $i$ to serve the corresponding buckets, thereby reducing the accumulated excess capacity by $(\alpha_i-\beta_i)$. If $\beta_i<\alpha_i$, then tentatively let bump $i$ serve bucket $i$ and note that they have excess capacity to offer, thereby increasing the accumulated excess capacity by $(\alpha_i-\beta_i)$. The process ends at $i=m$ when the accumulated excess capacity is used up to exactly cover the shortage of $(\alpha_m-\beta_m)$, because $\sum_{j=1}^m(\alpha_j-\beta_j)=0$ from the law of conservation.
\end{proof}

Now consider $\mathcal{Y}$ where $g^{-1}(\mathcal{Y})$ consists of $m$ affine legs. The absolute values of the slopes of the $i$-th affine leg is $2^{l_i}$. Let $\{k_i\}$ satisfies (\ref{eq:minimumentropysolution}) and be distinct from $\{l_i\}$. Is it possible to replace the $i$-th affine leg with \emph{piecewise} affine segments on the same interval such that the absolute value of any slope is in the set of $\{2^{k_i}\}$ while preserving $\lambda$? Lemma~\ref{lemma:klswitch} states that not only such a replacement exists, but also the new map $g_1$ after the replacement is an element of $\mathbb{G}$ and is within $\epsilon>0$ neighborhood of the original map $g$. The importance of such a replacement is that the entropy of the new $g_1$ reaches the minimum value given $m$. As will be clear in Theorem~\ref{theorem:denseG}, another map in $\mathbb{G}$ can be constructed from $g_1$ to have any target value of entropy that is greater than the minimum value.

\begin{lemma}
Let $g\in\mathbb{G}$ and $\mathcal{Y}$ be an interval with dyadic endpoints. Suppose that $g^{-1}(\mathcal{Y})$ consists of $m$ affine legs. The absolute value of the slope of the $i$-th leg is equal to $2^{l_i}$. Partition $[0,1]$ into $2m+1$ intervals $\mathcal{I}_j$ for $j=1, 2,\ldots, 2m+1$, such that $g^{-1}(\mathcal{Y})=\bigcup_{i=1}^m\mathcal{I}_{2i}$. Then $g_1\in\mathbb{G}$ exists and $[0,1]$ is partitioned into $2m+1$ intervals $\mathcal{J}_j$ for $j=1, 2,\ldots, 2m+1$ such that 
\begin{itemize}
\item
$|\mathcal{J}_{2i+1}|=|\mathcal{I}_{2i+1}|, g_1(\mathcal{J}_{2i+1})\simeq g(\mathcal{I}_{2i+1})$; 
\item
$g_1^{-1}(\mathcal{Y})$ consists of $m$ legs on $\{\mathcal{J}_{2i}\}$, i.e., $g_1^{-1}(\mathcal{Y}) = \bigcup_{i=1}^m \mathcal{J}_{2i}$; 
\item
$\rho(g,g_1)<\epsilon$;
\item
$\forall y\in \mathcal{Y}$, $g_1^{-1}(y)=\{x_1, \ldots, x_m\}$. If none of $x_i$ is a breakpoint, then the set of the absolute values of the slopes is $\{2^{k_i}\}$ where $\{k_i\}$ is given in (\ref{eq:minimumentropysolution}).
\end{itemize}
\label{lemma:klswitch}
\end{lemma}

\begin{proof}
Partition $\mathcal{Y}$ evenly into $2^M$ intervals $\mathcal{Y}_s$ for $s=1, 2, \ldots, 2^M$ where $2^{-M}<\frac{\epsilon}{2}$. Intervals $\mathcal{I}_{2i}$ for $i=1, 2, \ldots, m$ are correspondingly partitioned into $2^M$ intervals $\{\mathcal{I}_{2i,s}\}$ where $g^{-1}(\mathcal{Y}_s)=\bigcup_{i=1}^m \mathcal{I}_{2i,s}$. First construct as follows $g_2$ on $g^{-1}(\mathcal{Y}_1)=\bigcup_{i=1}^m \mathcal{I}_{2i,1}$ such that $g_2(\mathcal{I}^0_{2i,1})=g(\mathcal{I}^0_{2i,1})$ and $g_2(\mathcal{I}^1_{2i,1})=g(\mathcal{I}^1_{2i,1})$, and $\forall y\in\mathcal{Y}_1$, the set of the absolute values of the derivatives at $g_2^{-1}(y)$ is $\{2^{k_i}\}$ where $\{k_i\}$ is given in (\ref{eq:minimumentropysolution}). The same $g_2$ construction is then employed on $\mathcal{I}_{2i,s}$ for $s=2, 3, \ldots, 2^M$ while keeping the continuity in $g_2$. See Figure~\ref{fig:WeeklyDiscussion_20200504Page-1}. 

\begin{figure}
 \centering
  \includegraphics[width=11cm]{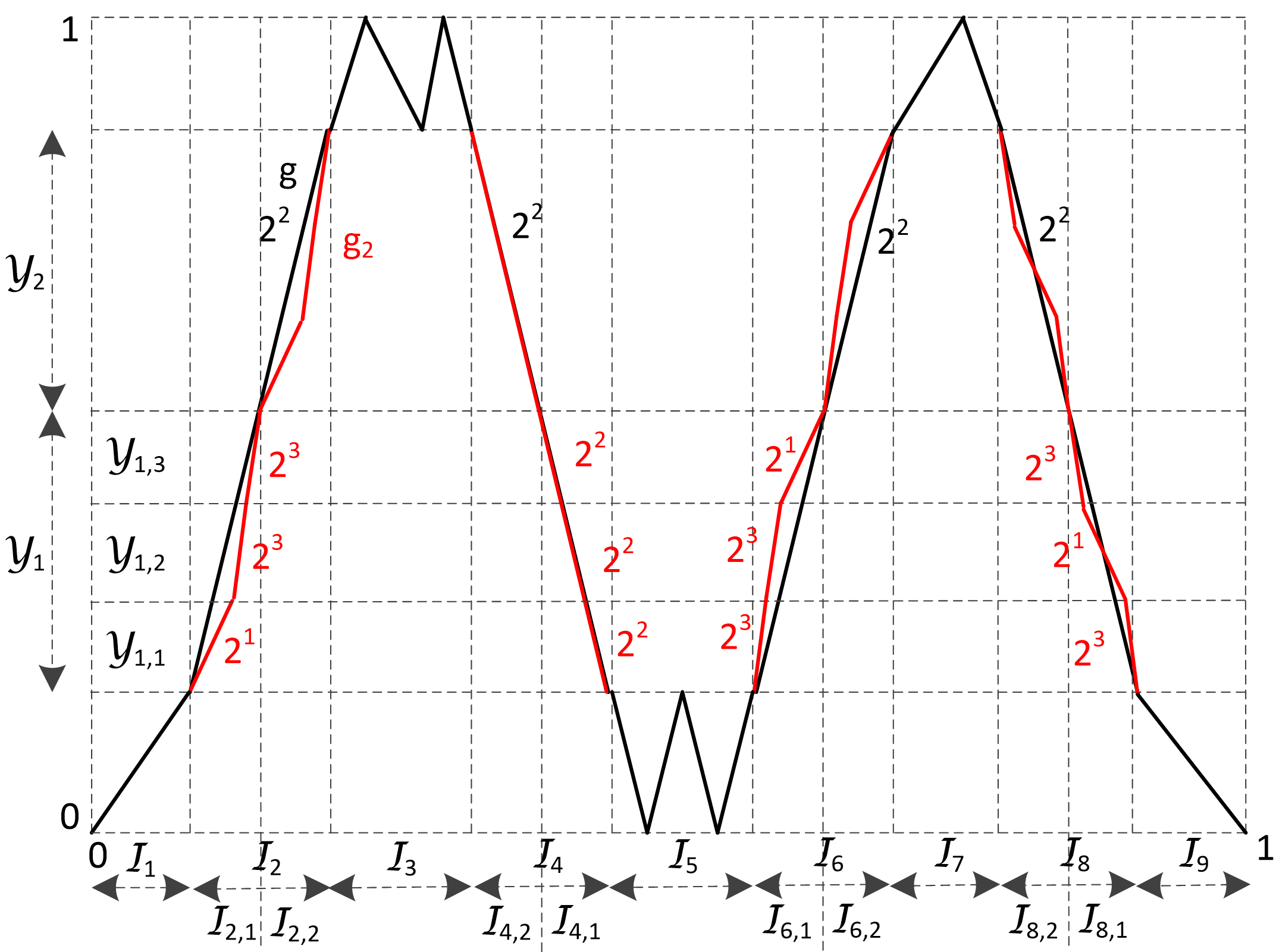}
  \caption{Construction of $g_2$. The original map $g$ is shown in black and the new map $g_2$ is in red.
A label next to an affine segment represents the absolute value of the slope of a replacing segment in red or a replaced one in black. $M=1, m=4$. $\{l_i\}=\{2, 2, 2, 2\}$. $\{k_i\}=\{1, 2, 3, 3\}$. $\mathcal{Y}_1$ is partitioned into $\{\mathcal{Y}_{1,1}, \mathcal{Y}_{1,2}, \mathcal{Y}_{1,3}\}$. To replace $\left(2^{2}, 2^{2}, 2^{2}, 2^{2}\right)$ of $g$, $g_2$ uses $\left(2^{1}, 2^{2}, 2^{3}, 2^{3}\right)$ for $\mathcal{Y}_{1,1}$, $\left(2^{3}, 2^{2}, 2^{3}, 2^{1}\right)$ for $\mathcal{Y}_{1,2}$ and $\left(2^{3}, 2^{2}, 2^{1}, 2^{3}\right)$ for $\mathcal{Y}_{1,3}$. The same construction of $g_2$ is employed for $\mathcal{Y}_2$ with horizontal shift for continuity.}
  \label{fig:WeeklyDiscussion_20200504Page-1}
\end{figure} 

Partition $\mathcal{Y}_1$ into a number of intervals $\{\mathcal{Y}_{1,t}\}$ and perturb $g$ as follows to obtain $g_2$. On $g^{-1}(\mathcal{Y}_{1,t})$, replace the $m$ legs of $g$ with affine segments whose slopes keep the same signs as the original $m$ segments of $g$ and take the absolute values equal to a permutation of $\{2^{k_i}\}$. That is, for $i=1, \ldots, m$, replace $2^{l_i}$ with $2^{k_{j_t(i)}}$ where $j_t(i)$ represents a permutation of $\Pi_t: \{1,\ldots,m\}\rightarrow\{1,\ldots,m\}$. The length of the $i$-th leg on $g^{-1}(\mathcal{Y}_{1,t})$ on the $x$-axis is thus $2^{-k_{j_t(i)}}|\mathcal{Y}_{1,t}|$. Permutation $\Pi_t$ used for different interval $\mathcal{Y}_{1,t}$ can be different. The task is to keep the total length of the $i$-th leg of $g^{-1}(\mathcal{Y}_1)$ on the $x$-axis unchanged, i.e., for all $i$,
\[
\sum_t 2^{-k_{j_t(i)}}|\mathcal{Y}_{1,t}| = 2^{-l_i}|\mathcal{Y}_1|=|\mathcal{I}_{2i,1}|,
\]
by employing appropriate $\{\mathcal{Y}_{1,t}\}$ partition and permutations $\{\Pi_t\}$. This problem is equivalent to Problem~\ref{problem:dynamicmatching} by viewing a new, replacing affine segment as a pump at rate $\alpha_i=2^{-k_i}$ and an original, replaced affine segment as a bucket with target $\beta_i=2^{-l_i}$ and viewing $\{\mathcal{Y}_{1,t}\}$ partition as the partition of service time intervals and $\{\Pi_t\}$ as the matching policy in the service intervals. 

In (\ref{eq:minimumentropysolution}), $k_1\le k_2\le \cdots \le k_m$. Without loss of generality, suppose that $l_1\le l_2 \le\cdots \le l_m$. Set $\{l_i\}$ can be partitioned into $\{l_{i_0+1}, \ldots, l_{i_1}\}, \{l_{i_1+1},\ldots, l_{i_2}\}, \ldots, \{l_{i_{n-1}+1},\ldots, l_{i_n}\}, \{l_{i_n+1}\}$ with $i_0=0, i_n=m-1$ such that 
\begin{eqnarray}
2^{-k_1}&=&2^{-l_{i_0+1}}+\cdots+2^{- l_{i_1}},\nonumber\\
2^{-k_2}&=&2^{-l_{i_1+1}}+\cdots+2^{- l_{i_2}}, \nonumber\\
&\cdots& \label{eq:kl_1}\\
2^{-k_n}&=&2^{-l_{i_{n-1}+1}}+\cdots+2^{- l_{i_n}},\nonumber\\
2^{-k_{n+1}}+\cdots+2^{-k_{i_n+1}} &=& 2^{-l_{i_n+1}}. \nonumber
\end{eqnarray}
As an example, suppose that $\{l_i\}=\{3, 3, 3, 3, 3, 3, 4, 4, 4, 4\}$ and $\{k_i\}=\{1, 2, 3, 4, 5, 6, 7, 8, 9, 9\}$. It follows that 
$2^{-1}=2^{-3}+2^{-3}+2^{-3}+2^{-3}, 2^{-2}=2^{-3}+2^{-3}, 2^{-3}=2^{-4}+2^{-4}, 2^{-4}=2^{-4}$ and $2^{-5}+2^{-6}+2^{-7}+2^{-8}+2^{-9}+2^{-9}=2^{-4}$. It can be shown from (\ref{eq:kl_1}) that $\sum_{j=1}^i (2^{-k_j}-2^{-l_j})\ge0$ for $i=1, \ldots, m$. From Lemma~\ref{lemma:dynamicmatching}, a solution of Problem~\ref{problem:dynamicmatching} exists and can be readily derived from the constructive steps in the proof of Lemma~\ref{lemma:dynamicmatching}. An example is illustrated in Figure~\ref{fig:WeeklyDiscussion_20200504Page-1}.

The $g_2$ constructed above may not be an element in $\mathbb{G}$, because the endpoints of $\{\mathcal{Y}_{1,t}\}$ from the solution of Lemma~\ref{lemma:dynamicmatching} are not necessarily all dyadic. If so, replace $\{\mathcal{Y}_{1,t}\}$ with $\{\mathcal{Y}'_{1,t}\}$, which are all dyadic, while keeping permutations $\Pi_t$ unchanged and $\sum_t |\mathcal{Y}'_{1,t}|=|\mathcal{Y}_1|$ to become $g_1$. 
Choose $\{\mathcal{Y}'_{1,t}\}$ sufficiently close to $\{\mathcal{Y}_{1,t}\}$ with 
\begin{equation}
\max_t \left||\mathcal{Y}_{1,t}|-|\mathcal{Y}'_{1,t}|\right|<\frac{\epsilon}{2}\cdot\frac{1}{2^{\max_i (k_i,l_i)}}\cdot \frac{1}{m}. 
\label{eq:maxYY'shift}
\end{equation}
Unlike $g_2$, $|g_1^{-1}(\mathcal{Y}_1)|$ is not necessarily equal to $|g^{-1}(\mathcal{Y}_1)|$ on every leg. Because $\sum_{i=1}^m 2^{-k_i}=\sum_{i=1}^m 2^{-l_i}$, the total of $|g_1^{-1}(\mathcal{Y}_1)|$ on all the $m$ legs is equal to that of $|g^{-1}(\mathcal{Y}_1)|$. That is, 
\[
\sum_t 2^{-k_{j_t(i)}}|\mathcal{Y}'_{1,t}| = 2^{-l_i}|\mathcal{Y}_1|
\]
may not hold for all $i$; however, 
\begin{equation}
\sum_{i=1}^m \sum_t 2^{-k_{j_t(i)}}|\mathcal{Y}'_{1,t}| = \sum_{i=1}^m 2^{-l_i}|\mathcal{Y}_1|.
\label{eq:I2iJ2iequal}
\end{equation}

To complete the proof, partition $[0,1]$ into $2m+1$ intervals $\{\mathcal{J}_{j}\}$ such that $g_1^{-1}(\mathcal{Y})=\bigcup_{i=1}^m \mathcal{J}_{2i}$ and $|\mathcal{J}_{2i+1}|=|\mathcal{I}_{2i+1}|$.  Specifically, let $|\mathcal{J}_{1}|=|\mathcal{I}_{1}|$. For $i=1, \ldots, m$, let 
\begin{equation}
|\mathcal{J}_{2i}|=\sum_{s=1}^{2^M}\sum_t 2^{-k_{j_t(i)}}|\mathcal{Y}'_{s,t}|
\label{eq:J2inew}
\end{equation}
and let $\mathcal{J}_{2i+1}$ be a horizontally shifted version of $\mathcal{I}_{2i+1}$ to accommodate small discrepancy between $|\mathcal{I}_{2i}|$ and $|\mathcal{J}_{2i}|$. From (\ref{eq:I2iJ2iequal}) and (\ref{eq:J2inew}), 
\[
\sum_{i=1}^m|\mathcal{J}_{2i}| 
= \sum_{i=1}^m\sum_{s=1}^{2^M}\sum_t 2^{-k_{j_t(i)}}|\mathcal{Y}'_{s,t}|
=\sum_{s=1}^{2^M}\sum_{i=1}^m\sum_t 2^{-k_{j_t(i)}}|\mathcal{Y}'_{s,t}|
=\sum_{s=1}^{2^M}\sum_{i=1}^m2^{-l_i}|\mathcal{Y}_s|
=\sum_{i=1}^m|\mathcal{I}_{2i}|.
\]
Thus, 
\[
\sum_{j=1}^{2m+1}|\mathcal{J}_{j}|=\sum_{i=1}^{m}|\mathcal{J}_{2i}|+\sum_{i=0}^{m}|\mathcal{J}_{2i+1}|=\sum_{i=1}^m|\mathcal{I}_{2i}|+\sum_{i=0}^{m}|\mathcal{I}_{2i+1}|=1.
\] 
Therefore, $\{\mathcal{J}_{j}\}$ is a valid partition of $[0,1]$. See Figure~\ref{fig:WeeklyDiscussion_20200504Page-2} for an illustration. 

\begin{figure}
 \centering
  \includegraphics[width=9.9cm]{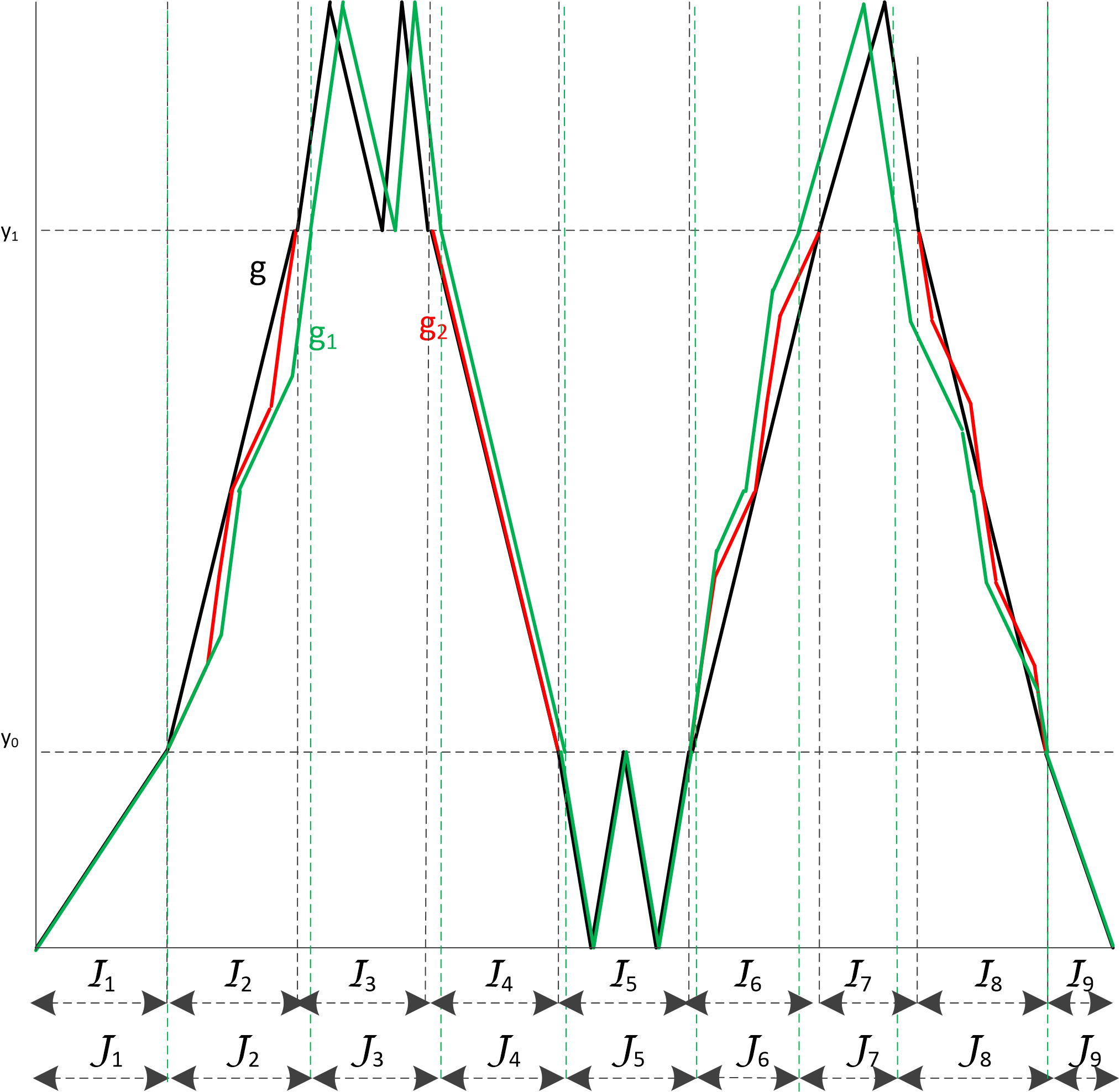}
  \caption{Construction of $g_1$. Map $g_1$ is obtained by slightly perturbing $g_2$ obtained in Figure~\ref{fig:WeeklyDiscussion_20200504Page-1} such that all the breakpoints of $g_1$ are dyadic.  Sets $\{\mathcal{I}_i\}$ and $\{\mathcal{J}_i\}$ are the partitions of $[0,1]$ by $g_2$ and $g_1$ respectively. $|\mathcal{J}_{2i+1}|=|\mathcal{I}_{2i+1}|$ for $i=1, \ldots, m$. It is not necessary that $|\mathcal{J}_{2i}|=|\mathcal{I}_{2i}|$ for all $i$ because all the breakpoints of $g_2$ are not necessarily dyadic; however, $\sum_{i=1}^m|\mathcal{I}_{2i}|=\sum_{i=1}^m|\mathcal{J}_{2i}|$. 
}
  \label{fig:WeeklyDiscussion_20200504Page-2}
\end{figure} 

Map $g_1$ on $\bigcup_{i=1}^m\mathcal{J}_{2i}$ has been constructed in the above. Next, let $g_1(\mathcal{J}_{2i+1})\simeq g(\mathcal{I}_{2i+1})$ for $i=0,\ldots,m$. From (\ref{eq:maxYY'shift}), the maximum discrepancy between $g_1$ and $g_2$ caused by the replacement of $\{\mathcal{Y}_{1,t}\}$ on $\bigcup_{i=1}^m\mathcal{I}_{2i}$ by $\{\mathcal{Y}'_{1,t}\}$ on $\bigcup_{i=1}^m\mathcal{J}_{2i}$ and due to the horizontal shifts from $\{\mathcal{I}_{2i+1}\}$ to $\{\mathcal{J}_{2i+1}\}$ is $\frac{\epsilon}{2}$. Hence, by construction, $g_1\in\mathbb{G}$ and $\rho(g,g_1)<\rho(g,g_2)+\rho(g_2,g_1)<\frac{\epsilon}{2}+\frac{\epsilon}{2}=\epsilon$.
\end{proof}

Lemma~\ref{lemma:klswitch} states that $m$ legs of affine segments with the absolute values of the slopes equal to any set $\{2^{l_i}\}$ can be approximated by $m$ legs of piecewise affine segments with the absolute values of the slopes $\{2^{k_i}\}$ with $\{k_i\}$ given by (\ref{eq:minimumentropysolution}). Set of integers $\{k_i\}$ given by (\ref{eq:minimumentropysolution}) is unique in the sense that by the analogy of Problem~\ref{problem:dynamicmatching}, any target amounts $\{2^{-l_i}\}$ are achievable by pumping rates $\{2^{-k_i}\}$ and the converse does not hold. Targets $\{2^{-k_i}\}$ are not achievable by any different pumping rates $\{2^{-l_i}\}$, because the difference between the maximum and minimum achievable amounts is upper-bounded by $\max_i 2^{-l_i} - \min_i 2^{-l_i}$, which is smaller than what is required by the targets $\max_i 2^{-k_i} - \min_i 2^{-k_i}=2^{-1}-2^{-m+1}$.

The slope of an affine segment of $g\in\mathbb{G}$ has to be in the form of $\pm 2^{k}$ for non-negative integer $k$. If $g$ is required to be LEO, then $k$ must be positive, and as a result, the entropy must be at least equal to $1$. If $g$ has $m$ legs on $[0,1]$, from Lemma~\ref{lemma:minimumentropysolution}, the minimum value of entropy $c_{\lambda}(g)$ is given by
\begin{equation}
c_{\min}(m)=\sum_{i=1}^{m-1} i 2^{-i}+(m-1)2^{-(m-1)}.
\end{equation}
It is easy to confirm that $c_{\min}(m)$ is an increasing sequence of $m$ and $\lim_{m\to\infty} c_{\min}(m)=2$. Therefore, for any $m$,
\begin{equation}
c_{\min}(m)<2.
\label{eq:cmin}
\end{equation}

\begin{theorem}
For any $c\in[2,\infty)$ and $\epsilon>0$, the subset of Markov LEO maps in $\mathbb{G}$ whose entropy is within $\epsilon$ of $c$ is dense in $C(\lambda)$.
\label{theorem:denseG}
\end{theorem}
\begin{proof}
Let $h\in C(\lambda)$. We will show that $g\in\mathbb{G}$ exists such that $g$ is Markov and LEO, $\rho(h,g)<\epsilon$ and $|c-c_{\lambda}(g)|<\epsilon$.

By Theorem~\ref{theorem:GLEOdense}, $g_0\in\mathbb{G}$ exists such that $g_0$ is Markov and LEO, and $\rho(h,g_0)<\frac{\epsilon}{3}$. Partition $[0,1]$ into $n$ intervals $\{\mathcal{Y}_i\}$ such that $g_0^{-1}(\mathcal{Y}_i)$ has $m_i$ legs for $i=1, \ldots, n$. Applying Lemma~\ref{lemma:klswitch} to all the $n$ intervals, $g_1\in \mathbb{G}$ exists such that $\rho(g_0,g_1)<\frac{\epsilon}{3}$ and the entropy of $g_1$ is given by 
\begin{equation}
c_{\lambda}(g_1)=\sum_{i=1}^{n}c_{\min}(m_i)\lambda(\mathcal{Y}_i)\le 2 \sum_{i=1}^{n}\lambda(\mathcal{Y}_i)=2. 
\label{eq:cminxm}
\end{equation}
The inequality is because of (\ref{eq:cmin}).

Consider any $c\ge2$. If $\left|c-c_{\lambda}(g_1)\right|\le\epsilon$, then let $g=g_1$ and the proof is complete. Otherwise, $\left|c-c_{\lambda}(g_1)\right|>\epsilon$. From (\ref{eq:cminxm}), $c-c_{\lambda}(g_1)>\epsilon$. Select any affine segment of $g_1$ of dyadic endpoints and dyadic length $\Delta x$ and replace it with a window perturbation to obtain $g$, as shown in Figure~\ref{fig:WeeklyDiscussion_20200504Page-3}. Specifically, suppose that the selected affine segment is of slope $2^{k}$. Let $\Delta x <2^{-k}\cdot\frac{\epsilon}{3}$. Thus, $\rho(g_1,g)<\frac{\epsilon}{3}$. Hence, $\rho(h,g)<\rho(h,g_0)+\rho(g_0,g_1)+\rho(g_1,g)<\epsilon$.

\begin{figure}
 \centering
  \includegraphics[width=6cm]{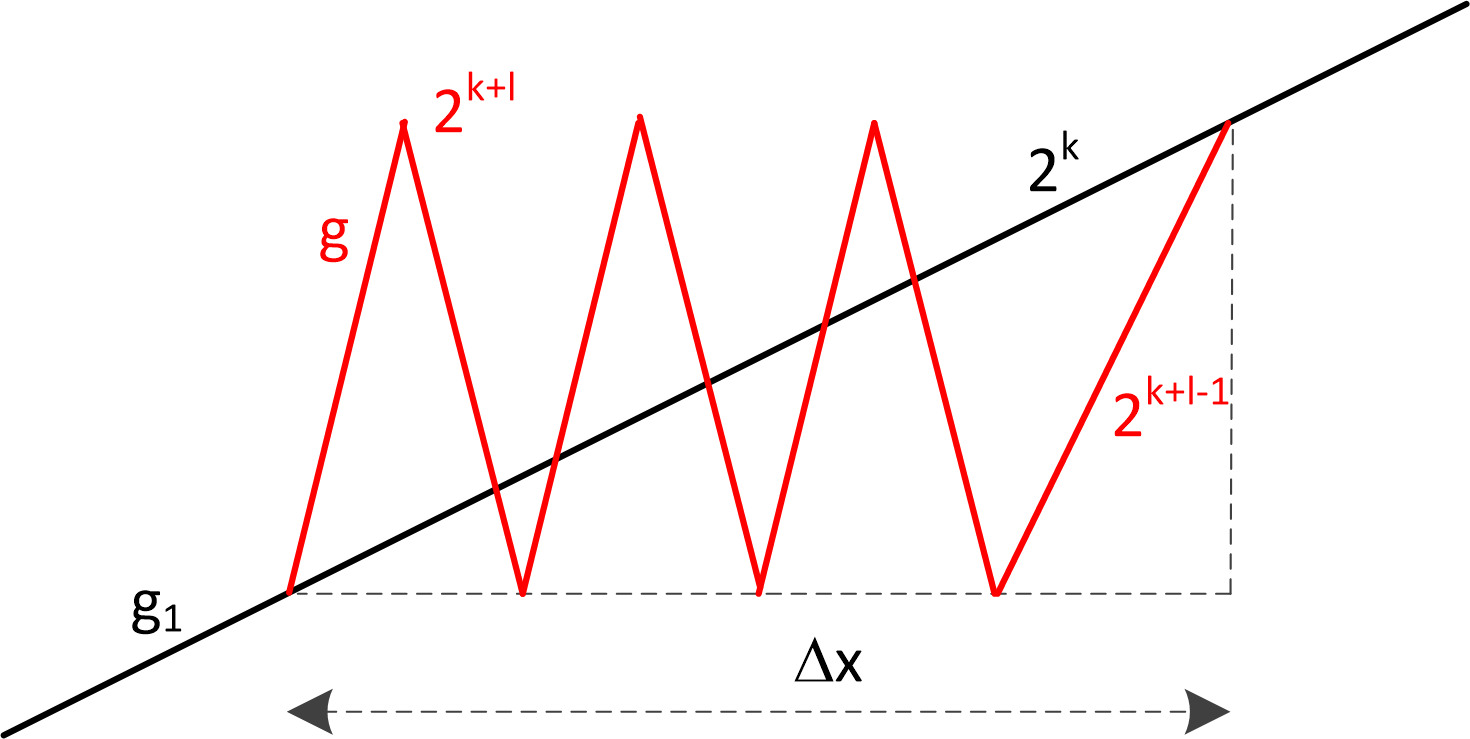}
  \caption{Increase entropy by replacing an affine segment (black) with a window perturbation (red). In this figure, $l=3$.}
  \label{fig:WeeklyDiscussion_20200504Page-3}
\end{figure} 

The window perturbation consists of $2^l-1$ legs, where each of the first $2^l-2$ legs is of slope of absolute value $2^{k+l}$ and the last leg is of slope of absolute value $2^{k+l-1}$, for integer $l$. The increase in entropy from $g_1$ to $g$ on $\Delta x$ is given by $\Delta x \cdot  (l-2^{-l+1})$. Thus, 
\[
c_{\lambda}(g)=c_{\lambda}(g_1)+\Delta x \cdot  (l-2^{-l+1}).
\]
Select $\Delta x$ sufficiently small and $l$ sufficiently large such that 
\[
\left|\Delta x \cdot  (l-2^{-l+1}) - \left(c-c_{\lambda}(g_1)\right)\right| <\epsilon.
\]
Hence, $|c_{\lambda}(g)-c|<\epsilon$, which completes the proof.
\end{proof}

The authors in \cite[Proposition.\ 21]{2019arXiv190607558B} prove that for every $c\in(0,\infty)$, the subset of Markov LEO $PA(\lambda)$ is dense in $C(\lambda)$ with $c_{\lambda}(h)=c$. Comparison with Theorem~\ref{theorem:denseG} shows that because the breakpoints in $\mathbb{G}$ are constrained to be dyadic, $c_{\lambda}(g)$ can only be within any $\epsilon$ neighborhood of an entropy target $c$ but is not guaranteed to be exactly equal to $c$. Moreover, to be LEO, the absolute value of the slope of any affine segment can only take a discrete value of $2^k$ for positive integer $k$ in $\mathbb{G}$, but can be any real number no smaller than $1$ in $PA(\lambda)$. Given a target map $h$, the smallest entropy value achievable by the dense subset of $\mathbb{G}$ is thus greater than what is achievable by that of $PA(\lambda)$.

\section{Decomposition \label{sec:generators}}

Recall that $\mathbb{F}$ is generated by two generator maps. Theorem~\ref{theorem:handfulgeneratormaps} in this section is to show that any map in $\mathbb{G}$ can be expressed as a composition of a finite number of basic maps in $\mathbb{G}$ and the generators in $\mathbb{F}$. To this end, first Theorem~\ref{theorem:1} shows that any map in $\mathbb{G}$ is a composition of maps in $\mathbb{F}$ and window perturbations, and then all window perturbations are shown to be generated by a few basic maps in $\mathbb{G}$. On the other hand, Theorem~\ref{theorem:Gnotfinitelygenerated} shows that unlike $\mathbb{F}$, $\mathbb{G}$ is not finitely generated.

First, consider type I breakpoints. 

\begin{lemma}
Let $\{\mathcal{I}_1<\cdots<\mathcal{I}_{2n+1}\}$ be a partition of $[0,1]$. Let $g\in \mathbb{G}$. Suppose that for $i=1, \ldots, n$, $g$ is an affine segment on $\mathcal{I}_{2i}$ whose slope is $(-1)^{p_i}2^{k_{i}}$, and $\mathcal{Y}=g(\mathcal{I}_{2i})$ is the same for all $i$. If  integers $\{l_{i}\}$ exist such that  
\begin{equation}
\sum_{i=1}^n 2^{-k_{i}} = \sum_{i=1}^n 2^{-l_{i}},
\label{eq:k=l}
\end{equation}
then $f_1\in \mathbb{F}$, $g_1\in \mathbb{G}$ and another partition of $[0,1]$ $\{\mathcal{J}_1<\cdots<\mathcal{J}_{2n+1}\}$ exist such that composition $g_1(f_1(x))=g(x)$ for any $x \in[0,1]$, $g_1(\mathcal{J}_j) \simeq g(\mathcal{I}_j)$ for any $j$, and 
\begin{itemize}
\item
for odd $j$, $|\mathcal{J}_j|=|\mathcal{I}_j|$;
\item
for even $j=2i$, $g_1$ on $\mathcal{J}_{2i}$ is an affine segment whose slope is equal to $(-1)^{p_i}2^{l_{i}}$.
\end{itemize}
\label{lemma:typeI_straighten_0}
\end{lemma}
\begin{proof}
The set of intervals $\{\mathcal{J}_j\}$ is completely defined once their lengths are defined. Specifically, let 
\begin{equation}
|\mathcal{J}_j|=\left\{
\begin{array}{ll}
|\mathcal{I}_j|, &\mbox{for odd $j$;}\\
|\mathcal{Y}|\cdot 2^{-l_{i}}, & \mbox{for even $j=2i$.}
\end{array}
\right.
\label{eq:k=l2}
\end{equation}
The endpoints of any interval $\mathcal{J}_j$ are dyadic by construction. To show the above partition is a valid one, note that for $i=1,2,\ldots,n$, 
\begin{eqnarray*}
|\mathcal{I}_{2i}| &=& |\mathcal{Y}|\cdot 2^{-k_{i}}, |\mathcal{J}_{2i}| = |\mathcal{Y}| \cdot 2^{-l_{i}} \\
\xRightarrow[\text{ }]{\text{By (\ref{eq:k=l})}}
\sum_{i=1}^n|\mathcal{I}_{2i}|&=&\sum_{i=1}^n|\mathcal{J}_{2i}|
\implies \sum_{j=1}^{2n+1} |\mathcal{J}_j| =\sum_{j=1}^{2n+1} |\mathcal{I}_j|=1.
\end{eqnarray*}

Construct $g_1$ as follows. With odd $j$, for $x\in \mathcal{J}_j$, let $g_1(x)=g(x-d_j)$, where $d_j=\mathcal{J}_{j}^0-\mathcal{I}_{j}^0$. For even $j=2i$, the graph of $g_1(x)$ on $\mathcal{J}_j$ is the affine segment that connects the right endpoint of $g_1$ on $\mathcal{J}_{j-1}$ and the left endpoint of $g_1$ on $\mathcal{J}_{j+1}$. Thus by construction (\ref{eq:k=l2}), the slope of the affine segment is $(-1)^{p_i}2^{l_i}$. Moreover, by (\ref{eq:k=l}), $g_1$ is $\lambda$-preserving. Hence, $g_1\in \mathbb{G}$.
 
Construct $f_1$ as follows. Set $f_1(0)=0$. For $j=1,2,\ldots, 2n+1$ and $x\in \mathcal{I}_{j}$, the slope of $f_1$ is set to $1$ for odd $j$ and to $2^{k_i-l_i}$ for even $j=2i$. By construction, the breakpoints of $f_1$ are dyadic and all the slopes are in the form of $2^m$ for integer $m$. To validate that $f_1\in \mathbb{F}$, note that 
\begin{eqnarray*}
|\mathcal{I}_{2i}|&=&|\mathcal{Y}|\cdot 2^{-k_i}
\implies |f_1(\mathcal{I}_{2i})|=|\mathcal{Y}|\cdot 2^{-k_i}\cdot 2^{k_i-l_i}=|\mathcal{Y}|\cdot 2^{-l_i}\\
\xRightarrow[\text{ }]{\text{By (\ref{eq:k=l})}} \sum_{i=1}^n|f_1(\mathcal{I}_{2i})|&=&\sum_{i=1}^n|\mathcal{I}_{2i}|
\implies \sum_{j=1}^{2n+1}|f_1(\mathcal{I}_{j})|=\sum_{j=1}^{2n+1}|\mathcal{I}_{j}|=1
\end{eqnarray*}

Finally, to show that $g_1(f_1)=g$, note that by construction of $f_1$ and $g_1$ that for $j=1, 2, \ldots, 2n+1$,
\[
f_1(\mathcal{I}_{j})\simeq \mathcal{J}_{j}, g_1(\mathcal{J}_{j})\simeq g(\mathcal{I}_{j})\implies g_1(f_1(\mathcal{I}_{j}))\simeq g(\mathcal{I}_{j})
\]
Hence,  $g_1(f_1)=g$. 
\end{proof}

\begin{corollary}
Let $\{\mathcal{I}'_1<\cdots<\mathcal{I}'_{2n+1}\}$ be a partition of $[0,1]$. Let $g\in \mathbb{G}$. Suppose that for $i=1, \ldots, n$, $g$ is a piecewise affine segment containing a single type I breakpoint $A_i$ on $\mathcal{I}'_{2i}$. Suppose that $\mathcal{Y}=g(\mathcal{I}'_{2i})$ is the same for all $i$ and that $A_{i,y}$ is the same for $i=1, \ldots, n$. Let $(-1)^{p_i}2^{l_{i}}$ be the slope of the affine segment on $g^{-1}([\mathcal{Y}^0,A_{i,y}])\cap \mathcal{I}'_{2i}$ and $(-1)^{p_i}2^{k_{i}}$ be the slope of the affine segment on $g^{-1}([A_{i,y},\mathcal{Y}^1])\cap \mathcal{I}'_{2i}$. If (\ref{eq:k=l}) holds, then $f_1\in \mathbb{F}$, $g_1\in \mathbb{G}$ and another partition of $[0,1]$ $\{\mathcal{J}'_1<\cdots<\mathcal{J}'_{2n+1}\}$ exist such that composition $g_1(f_1(x))=g(x)$ for any $x \in[0,1]$,
\begin{itemize}
\item
for odd $j$, $|\mathcal{I}'_j|=|\mathcal{J}'_j|$ and $g_1(\mathcal{J}'_j) \simeq g(\mathcal{I}'_j)$;
\item
for even $j$, $g_1$ on $\mathcal{J}'_j$ is an affine segment.
\end{itemize}
\label{lemma:typeI_straighten_1}
\end{corollary}
\begin{proof}
By Lemma~\ref{lemma:typeI_straighten_0}, one can replace the affine segment of $g$ on $g^{-1}([A_{i,y},\mathcal{Y}^1])\cap \mathcal{I}'_{2i}$ with another one whose slope is equal to $(-1)^{p_i}2^{l_i}$ for $i=1, 2, \ldots, n$, while keeping the slope of affine segment unchanged on $g^{-1}([\mathcal{Y}^0,A_{i,y}])\cap \mathcal{I}'_{2i}$. As a result, $g_1$ on $[{\mathcal{I}'_{2i}}^0, {\mathcal{I}'_{2i}}^1]$ is an affine segment with no type I breakpoint inside. 
\end{proof}

Figure~\ref{fig:WeeklyDiscussion_20200316Page-2} illustrates the use of Lemma~\ref{lemma:typeI_straighten_0} and Corollary~\ref{lemma:typeI_straighten_1}. 

\begin{figure}
 \centering
  \includegraphics[width=12cm]{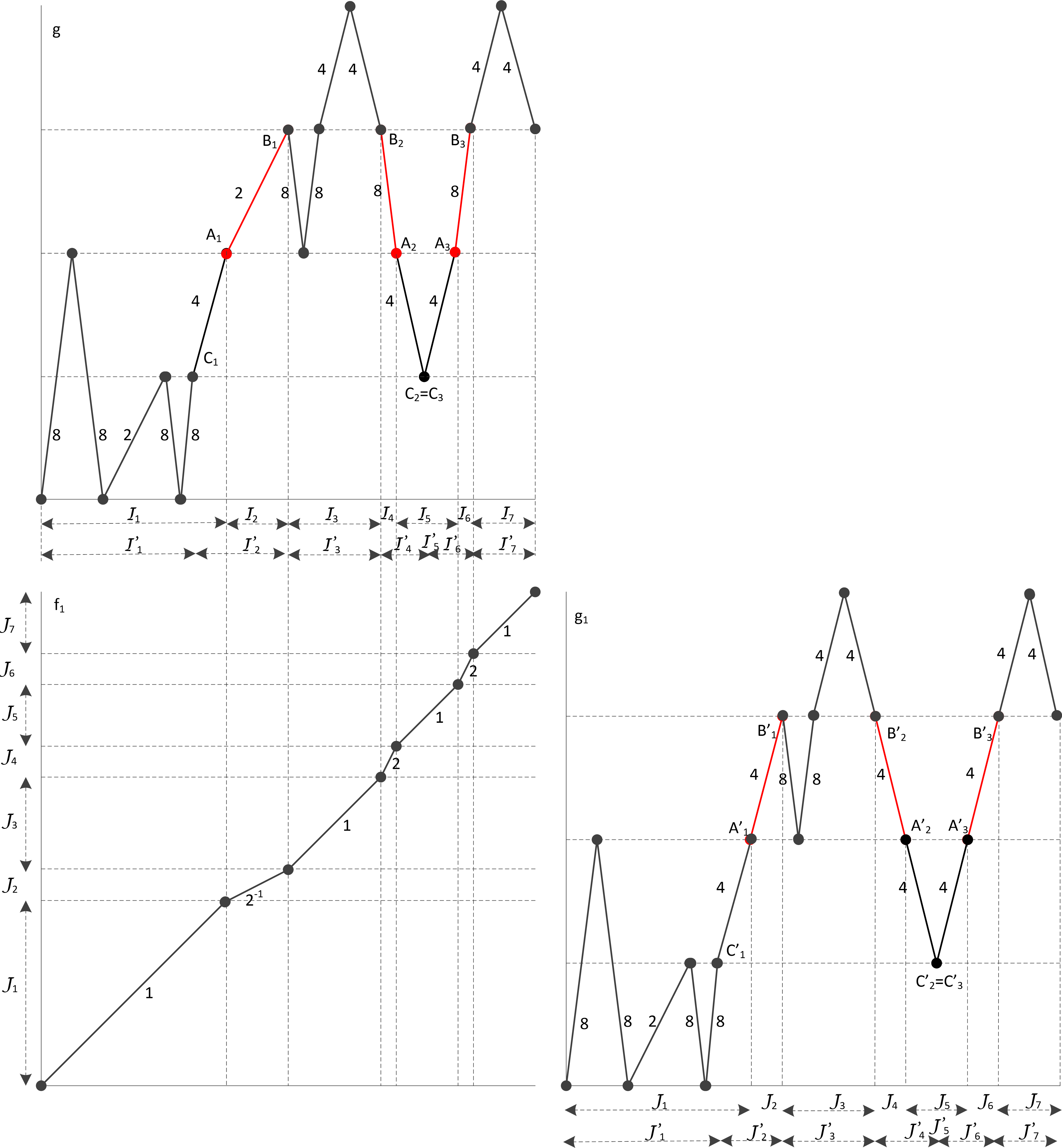}
  \caption{Use of Lemma~\ref{lemma:typeI_straighten_0} and Corollary~\ref{lemma:typeI_straighten_1}. The red segments in $g$ are replaced by those in $g_1$ where $g=g_1\circ f_1$. As a result, type I breakpoints $A_1, A_2, A_3$ of $g$ are eliminated in $g_1$ because the left and right derivatives are the same at $A'_1, A'_2, A'_3$. The number next to an affine segment represents the absolute value of the slope.}
  \label{fig:WeeklyDiscussion_20200316Page-2}
\end{figure} 

By Corollary~\ref{lemma:typeI_straighten_1}, $g$ can be generated by $g_1$, which eliminates certain type I breakpoints of $g$. In particular, the following corollary holds.
\begin{corollary}
Let $\{\mathcal{I}_i\}$ for $i=1, \ldots,m$ be a set of intervals with mutually disjoint interiors. Let $g\in \mathbb{G}$. If $g$ is monotone on $\mathcal{I}_i$ for all $i$, $\mathcal{Y}=g(\mathcal{I}_i)$ is the same for all $i$, and $g^{-1}(\mathcal{Y})=\bigcup_{i=1}^m \mathcal{I}_i$, then all the type I breakpoints in the interiors of $\{\mathcal{I}_i\}$ can be eliminated. 
\label{corollary:typeI_straighten_1}
\end{corollary}
\begin{proof}
$\forall y\in\mathcal{Y}$, $g^{-1}(\mathcal{Y})$ consists of one point on each $\mathcal{I}_i$. If the derivative exists at all points of $g^{-1}(y)$, then (\ref{eq:sum2^-k=1}) holds and so does (\ref{eq:k=l}) with a common set of $\{l_i\}$ for any $y$. The set of the affine segments in the neighborhood of $g^{-1}(y)$, i.e., $g^{-1}([y-\delta, y+\delta])$ for some $\delta>0$, whose slopes have absolute values equal to $2^{k_i(y)}$ on $\mathcal{I}_i$ can be replaced by affine segments with absolute values of slopes equal to $2^{l_i}$ with $\sum_{i=1}^m 2^{-l_i}=1$. Set $\{2^{l_i}\}$ remain the same for all $y\in\mathcal{Y}$. Thus, by Corollary~\ref{lemma:typeI_straighten_1}, the piecewise affine segment on each $\mathcal{I}_i$ is replaced by an affine segment and any breakpoint in the interior of $\mathcal{I}_i$ is eliminated.
\end{proof}

Corollary~\ref{corollary:typeI_straighten_1} covers the case of $g^{-1}(g(\bigcup_{i=1}^m \mathcal{I}_i)) = \bigcup_{i=1}^m \mathcal{I}_i$. The opposite case is where $\forall y\in g(\bigcup_{i=1}^m \mathcal{I}_i)$, $g^{-1}(y)\not\subset\bigcup_{i=1}^m \mathcal{I}_i$, which is addressed next.

\begin{lemma}
Let $\{\mathcal{I}_1<\cdots<\mathcal{I}_{2m+1}\}$ be a partition of $[0,1]$. Suppose that $g\in\mathbb{G}$ is an affine segment on $\mathcal{I}_{2i}$ and $\mathcal{Y}=g(\mathcal{I}_{2i})$ is the same for $i=1, \ldots, m$. If $\forall y\in \mathcal{Y}$, $g^{-1}(y)\not\subset\bigcup_{i=1}^m \mathcal{I}_i$, then $f_1\in\mathbb{F}, g_1\in\mathbb{G}$ and another partition of $[0,1]$ $\{\mathcal{J}_1<\cdots<\mathcal{J}_{2m+1}\}$ exist such that composition $g_1(f_1(x))=g(x)$ for any $x \in[0,1]$, $g_1(\mathcal{J}_{2i})\simeq g(\mathcal{I}_{2i})$ and $g_1$ is an affine segment on $\mathcal{J}_{2i}$ with the absolute value of the slope being $2^{k_i}$ where $\sum_{i=1}^m 2^{-k_i}=2^{-K}$ for some positive integer $K$.
\label{lemma:typeI_straighten_2_1}
\end{lemma}

\begin{proof}
Partition $\mathcal{Y}$ into intervals $\{\mathcal{Y}_j\}$, $j=1, 2, \ldots, n$, such that no breakpoint exists whose $y$-coordinate falls in the interior of any $\mathcal{Y}_j$, i.e., $\nexists$ breakpoint $B$ such that $B_y \in(\mathcal{Y}_j^0, \mathcal{Y}_j^1)$. Consider $g^{-1}(\mathcal{Y}_j)$. Let $g^{-1}(\mathcal{Y}_j)=\{\mathcal{I}_{j,1}, \mathcal{I}_{j,2}, \ldots, \mathcal{I}_{j,m}, \mathcal{I}_{j,m+1}, \ldots, \mathcal{I}_{j,m+n_j}\}$ with mutually disjoint interiors, where interval
$\mathcal{I}_{j,i}\subset\mathcal{I}_{2i}$ for $i=1, \ldots, m$. By the hypothesis of the lemma, $n_j\ge1$. Because $\mathcal{Y}=\bigcup_{j=1}^n \mathcal{Y}_j$, it follows that for $i=1, \ldots, m$,
\begin{equation}
\bigcup_{j=1}^n \mathcal{I}_{j,i}=\mathcal{I}_{2i}.
\label{eq:Iij=I} 
\end{equation}

The graph of $g$ is affine on every $\mathcal{I}_{j,i}$ because no breakpoint exists on $\mathcal{Y}_j^{\circ}$. Let $2^{k_{j,i}}$ be the absolute value of the slope of the affine segment on $\mathcal{I}_{j,i}$. Because $g$ is $\lambda$-preserving, 
\begin{equation}
\sum_{i=1}^{m+n_j} 2^{-k_{j,i}}=1.
\label{eq:LK0}
\end{equation}
Let 
\begin{equation}
\sum_{i=1}^{m} 2^{-k_{j,i}}=L\cdot 2^{-K},
\label{eq:LK}
\end{equation}
where $L$ is an odd integer. Because the graph of $g$ is an affine segment on $\mathcal{I}_{2i}$, $k_{j,i}$ does not depend on $j$ when $i=1, \ldots, m$, which is the reason that $L, K$ do not have subscript $j$ in (\ref{eq:LK}).

Assume that $L>1$. The following iterative procedure is employed to decrement $L$ by adjusting $k_{j,i}$ where $1\le i\le m+n_j$ while satisfying (\ref{eq:LK0}) so that eventually $L=1$.

From (\ref{eq:LK}), $\sum_{i=1}^{m} 2^{-k_{j,i}+K}=L$. Focus on $k_{j,i}$ that is greater than or equal to $K$ and arrange them in an increasing order. Add such $2^{-k_{j,i}+K}$ terms one-by-one until the sum reaches $1$. Therefore, a subset of $\{1, 2, \ldots, m\}$, denoted by $\Phi_1$, exists such that $\sum_{i\in\Phi_1} 2^{-k_{j,i}}=2^{-K}$. Because $L-1$ is even, let $\sum_{i=1, i\not\in\Phi_1}^{m} 2^{-k_{j,i}}=(L-1)\cdot 2^{-K}=L'\cdot 2^{-K'}$, where $0<K'<K$ and $L'$ is odd. $L'<L$. Similarly, a subset of $\{1, 2, \ldots, m\}\setminus\Phi_1$, denoted by $\Phi_2$, exists such that $\sum_{i\in\Phi_2} 2^{-k_{j,i}}=2^{-K'}$. Now let 
\begin{equation}
k'_{j,i}=\left\{
\begin{array}{ll}
k_{j,i}+K'+1-K, & i\in\Phi_1;\\
k_{j,i}+1, & i\in\Phi_2;\\
k_{j,i}, & i\in\{1, 2, \ldots, m\}\setminus(\Phi_1\cup\Phi_2).
\end{array}
\right.
\label{eq:kk'1}
\end{equation}
Therefore, 
\[
\sum_{i=1}^{m} 2^{-k'_{j,i}}=\left(\sum_{i\in\Phi_1} +\sum_{i\in\Phi_2} +\sum_{ i\in\{1, 2, \ldots, m\}\setminus(\Phi_1\cup\Phi_2)}\right)2^{-k'_{j,i}}=2^{-K}\cdot 2^{K-K'-1}+2^{-K'}\cdot 2^{-1}+(L'-1)\cdot 2^{-K'}=L'\cdot 2^{-K'}. 
\]
On the other hand, $\sum_{i=m+1}^{m+n_j} 2^{-k_{j,i}}=1-L\cdot 2^{-K}$. Similarly, a subset of $\{m+1, m+2, \ldots, m+n_j\}$, denoted by $\Psi$, exists such that $\sum_{i\in\Psi} 2^{-k_{j,i}}=2^{-K}$. Now let 
\begin{equation}
k'_{j,i}=\left\{
\begin{array}{ll}
k_{j,i}-1, & i\in\Psi;\\
k_{j,i}, & i\in\{m+1, m+2, \ldots, m+n_j\}\setminus\Psi.
\end{array}
\right.
\label{eq:kk'2}
\end{equation}
Therefore, 
\[
\sum_{i=m+1}^{m+n_j} 2^{-k'_{j,i}}=\left(\sum_{i\in\Psi} +\sum_{ i\in\{m+1, m+2, \ldots, m+n_j\}\setminus\Psi}\right)2^{-k'_{j,i}}=2^{-K}\cdot 2^1+
(1-L\cdot 2^{-K}-2^{-K})=1-L'\cdot 2^{-K'}.
\] 
From $\{k_{j,i}\}$ to $\{k'_{j,i}\}$, $L$ drops to $L'$. The above process continues until $L=1$. Because $k_{j,i}$ does not depend on $j$ when $i=1, 2, \ldots, m$, neither does $k'_{j,i}$ in (\ref{eq:kk'1}). This property remains in the process. 

When the process ends,  for any $j$
\begin{equation}
\sum_{i=1}^{m} 2^{-k''_{j,i}}=2^{-K''}.
\label{eq:LK''}
\end{equation} 
Because 
\begin{equation}
\sum_{i=1}^{m+n_j} 2^{-k''_{j,i}}=\sum_{i=1}^{m+n_j} 2^{-k_{j,i}}=1,
\label{eq:k''k=1}
\end{equation}
apply Lemma~\ref{lemma:typeI_straighten_0} for all $j$ one-by-one. Then $f_1\in\mathbb{F}, g_1\in\mathbb{G}$ and another partition of $[0,1]$ 
$\{\mathcal{J}_1<\cdots<\mathcal{J}_{2m+1}\}$ exist such that composition $g_1(f_1(x))=g(x)$ for any $x \in[0,1]$. Just like (\ref{eq:Iij=I}), for $i=1, \ldots, m$, $\mathcal{J}_{2i}$ can be decomposed to $\mathcal{J}_{2i}=\bigcup_{j=1}^n \mathcal{J}_{j,i}$ such that $g_1(\mathcal{J}_{j,i})\simeq g(\mathcal{I}_{j,i})$. Thus, $g_1(\mathcal{J}_{2i})\simeq g(\mathcal{I}_{2i})$. The graph of $g_1$ is an affine segment on $\mathcal{J}_{j,i}$ with the absolute value of the slope being $2^{k''_{j,i}}$, which does not depend on $j$ when $i=1, 2, \ldots, m$. Thus, $g_1$ is an affine segment in every $\mathcal{J}_{2i}$. The absolute values of the slopes of these affine segments for $i=1, \ldots, m$ satisfies (\ref{eq:LK''}).  This completes the proof.
\end{proof}

\begin{figure}
 \centering
  \includegraphics[width=14cm]{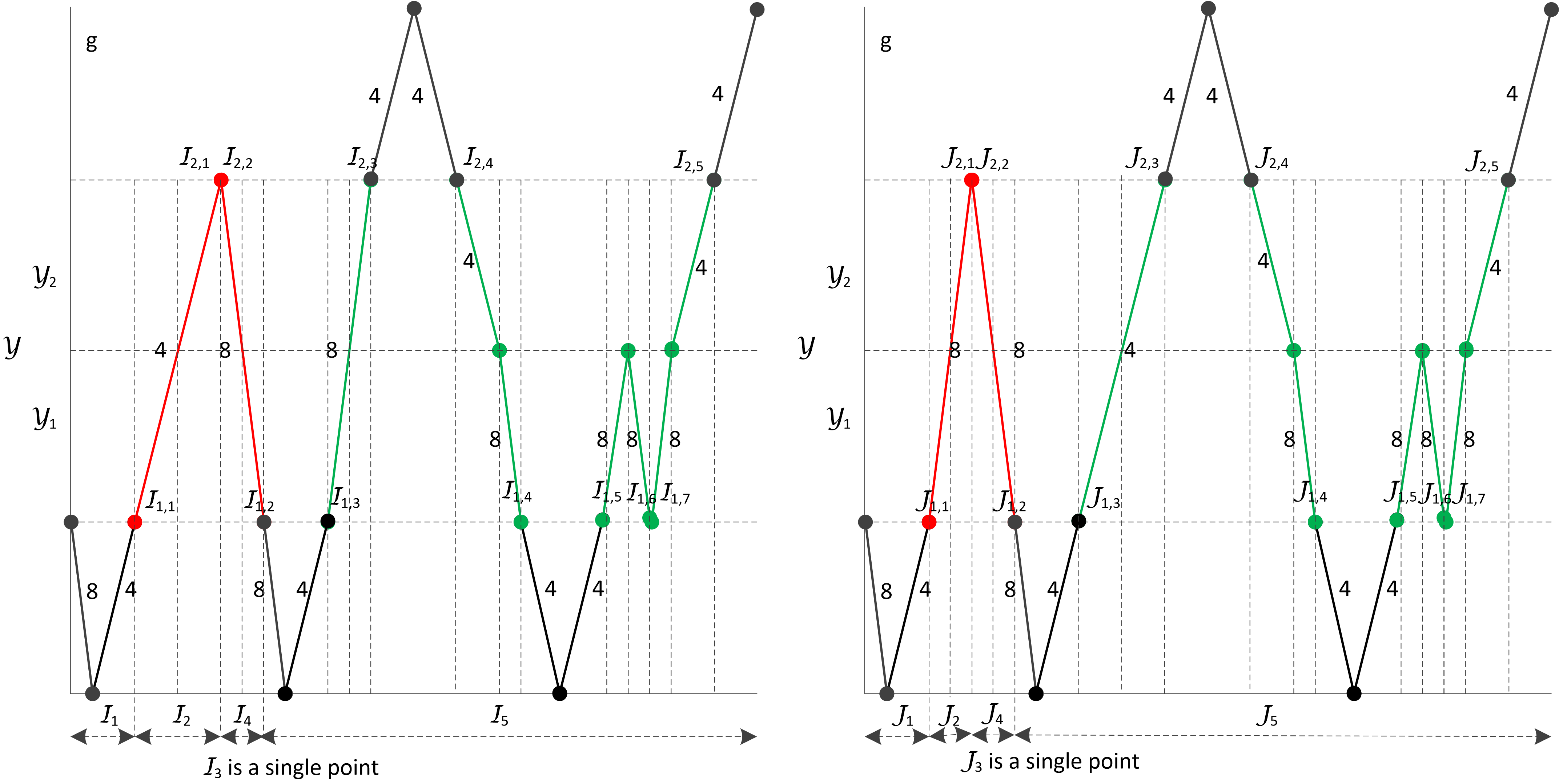}
  \caption{Illustration of the proof of Lemma~\ref{lemma:typeI_straighten_2_1}. The graphs of $g$ on $\{\mathcal{I}_{2i}\}$ and $g_1$ on $\{\mathcal{J}_{2i}\}$ are drawn in red. The graphs of $g$ on $g^{-1}(\mathcal{Y}) \cap ([0,1]\setminus \bigcup_{i=1}^m\mathcal{I}_{2i})$ and $g_1$ on $g_1^{-1}(\mathcal{Y}) \cap ([0,1]\setminus \bigcup_{i=1}^m\mathcal{J}_{2i})$ are drawn green. $m=2,n=2, n_1=7, n_2=5$.  $g$ on $\bigcup_{i=1}^m\mathcal{I}_{2i}$ satisfies $2^{-2}+2^{-3}=3\cdot2^{-3}$ in (\ref{eq:LK}). 
$g_1$ on $\bigcup_{i=1}^m\mathcal{J}_{2i}$ satisfies $2^{-3}+2^{-3}=1\cdot2^{-2}$ in (\ref{eq:LK''}). }
\label{fig:WeeklyDiscussion_20200316Page-7}
\end{figure} 

Figure~\ref{fig:WeeklyDiscussion_20200316Page-7} shows an example to illustrate the proof of Lemma~\ref{lemma:typeI_straighten_2_1}. 

Next, consider type II breakpoints. Let $w$ be an $m$-fold window perturbation on interval $\mathcal{I}$ and $w(x)=x$ or $w(x)=1-x$ if $x\in[0,1]\setminus\mathcal{I}$ to continuously connect the end points of the window perturbation. Abusing the terminology, we call $w$ a window perturbation too. Let $2^{k_i}$ be the absolute value of the slope of the $i$-th leg of the $m$-fold window perturbation $w$ for $i=1, 2, \ldots, m$. $\sum_{i=1}^m 2^{-k_i}=1$ to be $\lambda$-preserving.

Suppose that $g_1$ on $\mathcal{I}$ is an affine segment, whose slope has an absolute value of $2^{K}$. Then composition $g_1(w_1)$ is identical to $g_1$ on $[0,1]\setminus\mathcal{I}$. On $\mathcal{I}$ the affine segment of $g_1$ is replaced by an $m$-fold window perturbation, whose $i$-th leg has an absolute value of the slope $2^{k_i+K}$. The ratio of the slopes of any two legs of $g_1(w_1)$ is the same as that of the corresponding legs of $w_1$. Lemma~\ref{lemma:typeI_straighten_2a} follows.

\begin{lemma}
Suppose that $g$ on interval $\mathcal{I}$ is an $m$-fold window perturbation. Let $2^{k_i+K}$ be the absolute value of the slope of the $i$-th leg. If $\sum_{i=1}^m 2^{-k_i}=1$, then $g_1\in\mathbb{G}$ and an $m$-fold window perturbation map $w_1$ exist such that composition $g_1(w_1(x))=g(x)$ for any $x \in[0,1]$, $g_1([0,\mathcal{I}^0])\simeq g([0,\mathcal{I}^0])$, $g_1([\mathcal{I}^1,1])\simeq g([\mathcal{I}^1,1])$, and $g_1$ is an affine segment on $\mathcal{I}$ whose slope has the absolute value of $2^K$.
\label{lemma:typeI_straighten_2a}
\end{lemma}

Figure~\ref{fig:20200224fig_1} shows the use of Lemma~\ref{lemma:typeI_straighten_2a}. Interval $\mathcal{I}$ can be anywhere on $[0,1]$ for odd $m$. For even $m$, $\mathcal{I}$ must cover at least one endpoint; that is, $\mathcal{I}^0=0$ or $\mathcal{I}^1=1$.

\begin{figure}
\centering
\begin{subfigure}{.49\textwidth}
  \centering
  \includegraphics[width=1.\linewidth]{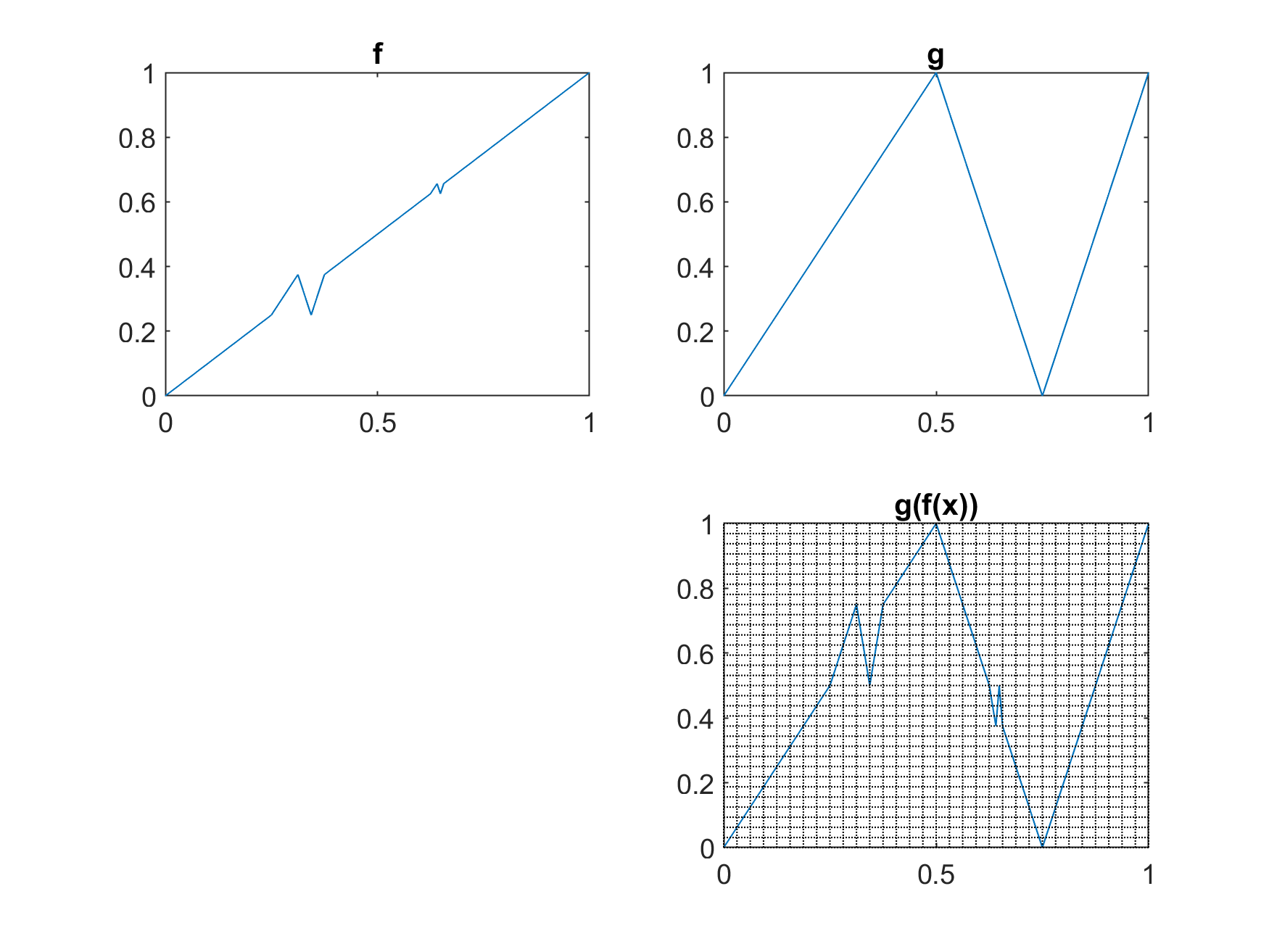}
    \caption{}
\end{subfigure}
\begin{subfigure}{.49\textwidth}
  \centering
  \includegraphics[width=1.\linewidth]{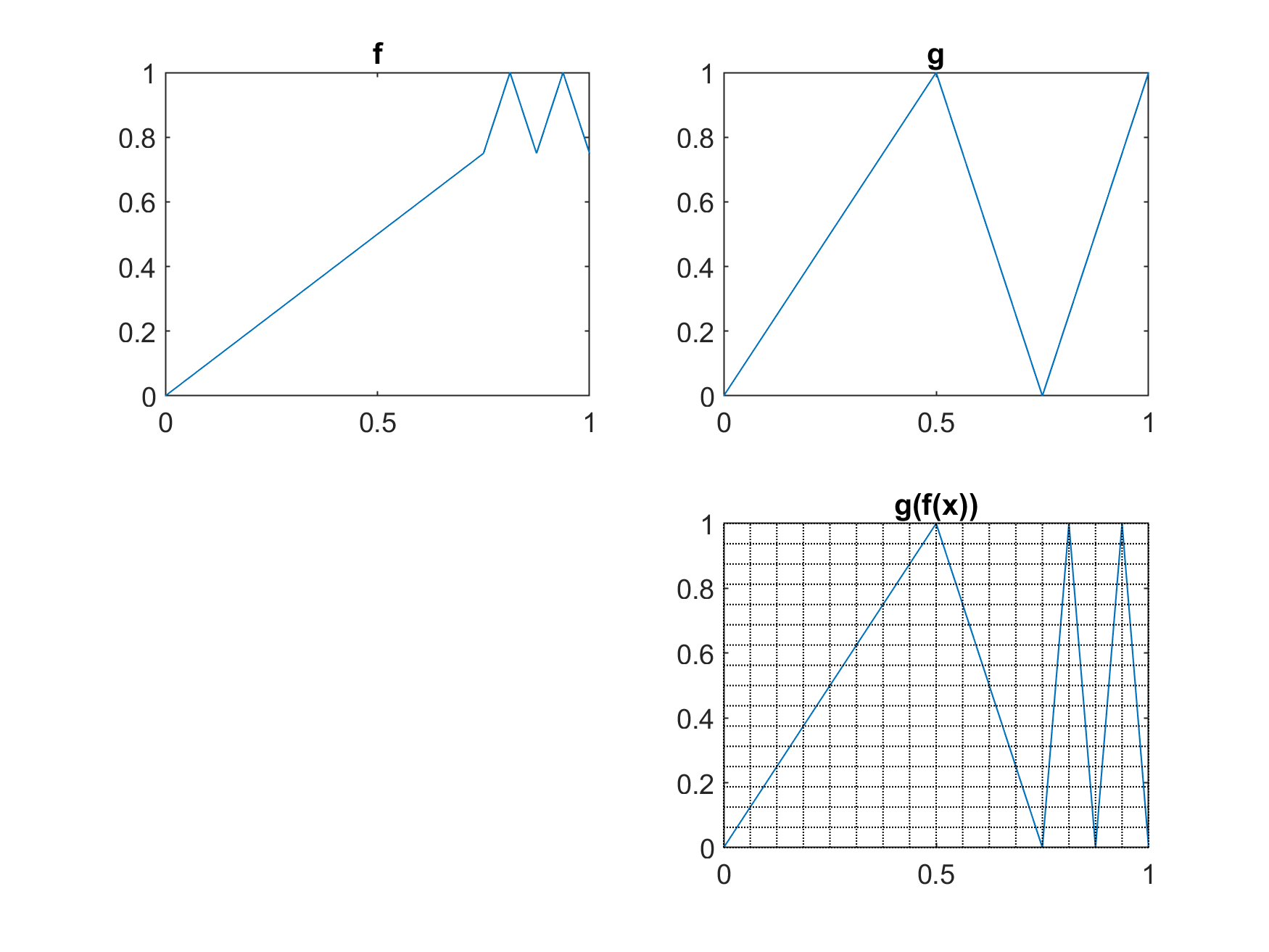}
    \caption{}
\end{subfigure}
\caption{Use of Lemma~\ref{lemma:typeI_straighten_2a}: (a) the case of odd $m$ and (b) the case of even $m$. A few type II breakpoints in $g(f(x))$ are eliminated in $g(x)$.}
\label{fig:20200224fig_1}
\end{figure}

A special case of Lemma~\ref{lemma:typeI_straighten_2a} is when $K=0$. Because $\sum_{i=1}^m 2^{-k_i}=1$, $\forall y\in g(\mathcal{I})$, $g^{-1}(y)\subset\mathcal{I}$. Thus, $g^{-1}(g(\mathcal{I})) = \mathcal{I}$. This case is similar to what Corollary~\ref{corollary:typeI_straighten_1} covers. On the other hand, just like Lemma~\ref{lemma:typeI_straighten_2_1}, Lemma~\ref{lemma:typeI_straighten_2} covers the case where $\forall y\in g(\bigcup_{i=1}^m \mathcal{I}_i)$, $g^{-1}(y)\not\subset\bigcup_{i=1}^m \mathcal{I}_i$. Specifically, in Lemma~\ref{lemma:typeI_straighten_2_1} when $\mathcal{I}_{2i}$, for $i=1, \ldots, m$, are adjacent, i.e., each of $\mathcal{I}_{3}, \mathcal{I}_{5}, \ldots, \mathcal{I}_{2m-1}$ is reduced to a single point, the $m$ affine segments form a $m$-fold window perturbation and can be further simplified as stated in Lemma~\ref{lemma:typeI_straighten_2}.

\begin{lemma}
Suppose that in Lemma~\ref{lemma:typeI_straighten_2_1}, $\bigcup_{i=1}^m \mathcal{I}_{2i}$ is an interval, denoted by $\mathcal{I}$. That is, $g$ is an $m$-fold window perturbation on $\mathcal{I}$. Then $f_1\in\mathbb{F}, g_1\in\mathbb{G}$, an $m$-fold window perturbation map $w_1$ and another partition of $[0,1]$ $\{\mathcal{J}_1<\cdots<\mathcal{J}_{2m+1}\}$ exist such that composition $g_1(w_1(f_1(x)))=g(x)$ for any $x \in[0,1]$, and $\bigcup_{i=1}^m \mathcal{J}_{2i}$ is an interval on which $g_1$ is an affine segment. 
\label{lemma:typeI_straighten_2}
\end{lemma}
\begin{proof}
Maps $f_1, g_2$ and interval partition $\{\mathcal{J}_1<\cdots<\mathcal{J}_{2m+1}\}$ are obtained by Lemma~\ref{lemma:typeI_straighten_2_1} such that $g_2(f_1)=g$.
Each of $\mathcal{J}_{3}, \mathcal{J}_{5}, \ldots, \mathcal{J}_{2m-1}$ is reduced to a single point just like $\mathcal{I}_{3}, \mathcal{I}_{5}, \ldots, \mathcal{I}_{2m-1}$. Thus, $\bigcup_{i=1}^m \mathcal{J}_{2i}$ is an interval, denoted by $\mathcal{J}$. In the proof of Lemma~\ref{lemma:typeI_straighten_2_1}, $g_2$ on $\mathcal{J}$ is an $m$-fold window perturbation, just like $g$ on $\mathcal{I}$. The difference between the two $m$-fold window perturbations is that the slopes of their legs are described in (\ref{eq:LK}) and (\ref{eq:LK''}) respectively. Now let $w_1$ be an $m$-fold window perturbation on $\mathcal{J}$ where the slope of the $i$-th leg is $(-1)^{i+1}2^{k''_{j,i}-K''}$ for $i=1, 2, \ldots, m$. Modify $g_2$ to become $g_1$ where the $m$-fold window perturbation on $\mathcal{J}$ of $g_2$ is replaced by an affine segment of $g_1$ whose slope has an absolute value of $2^{K''}$. The sign of the slope is such that $g_1$ is continuous at the endpoints of $\mathcal{J}$. By Lemma~\ref{lemma:typeI_straighten_2a}, $g_2=g_1(w_1)$. Therefore, $g_1(w_1(f_1))=g$. 
\end{proof}

\begin{figure}
 \centering
  \includegraphics[width=12cm]{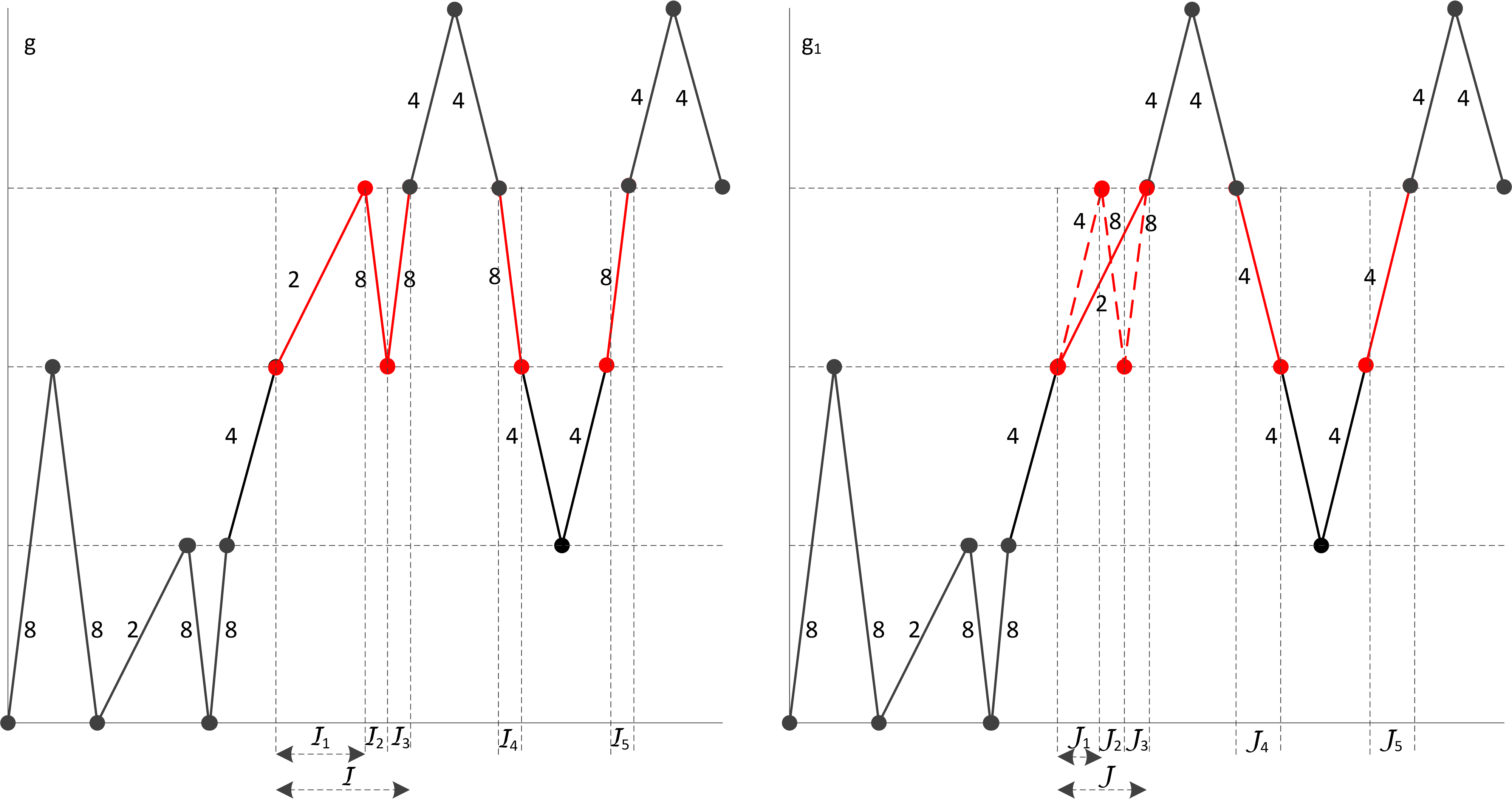}
  \caption{Use of Lemma~\ref{lemma:typeI_straighten_2}. The three solid red affine segments in the right figure represent $g_1$. The leftmost one is replaced by three dashed red segments due to a window perturbation $w_1$ with $m=3$. The resultant five red segments, dashed and solid, are horizontally adjusted by $f_1$ to generate $g$ in the left figure. Two type II breakpoints of $g$ are eliminated in $g_1$.}
  \label{fig:WeeklyDiscussion_20200316Page-3}
\end{figure} 
Figure~\ref{fig:WeeklyDiscussion_20200316Page-3} illustrates the use of Lemma~\ref{lemma:typeI_straighten_2}. By Lemma~\ref{lemma:typeI_straighten_2}, $g$ can be generated by $g_1$ by eliminating $m-1$ type II breakpoints.

In Lemma~\ref{lemma:typeI_straighten_2a} and Lemma~\ref{lemma:typeI_straighten_2}, $g$ has to be an $m$-fold window perturbation on $\mathcal{I}$. The $m$-fold window perturbation consists of $m$ legs each of which is an affine segment. One can generalize the notion of $m$-fold window perturbation such that it consists of $m$ legs each of which itself consists of piecewise affine segments. The more precise definition is given below.

\begin{definition}[\textsl{Generalized Window Perturbation}]
A generalized $m$-fold window perturbation $g$ on $\mathcal{I}$ is defined as follows. Suppose that $\mathcal{I}$ is partitioned into $\{\mathcal{I}_1<\mathcal{I}_2<\cdots<\mathcal{I}_m\}$. $g(\mathcal{I}_i)\simeq g(\mathcal{I}_j)$ if both $i$ and $j$ are either odd or even and $g(\mathcal{I}_i)\simeq g(\hat{\mathcal{I}}_j)$ if one of $i,j$ is even and the other is odd where $\hat{\mathcal{I}}_j$ represents $\mathcal{I}_j$ flipped horizontally. Each $\mathcal{I}_i$ is referred to as a component interval of $\mathcal{I}$. 
\end{definition}
Figure~\ref{fig:WeeklyDiscussion_20200316Page-5} provides two examples of the generalized $m$-fold window perturbations, one for an even $m$ and the other for an odd $m$.

\begin{figure}
 \centering
  \includegraphics[width=5cm]{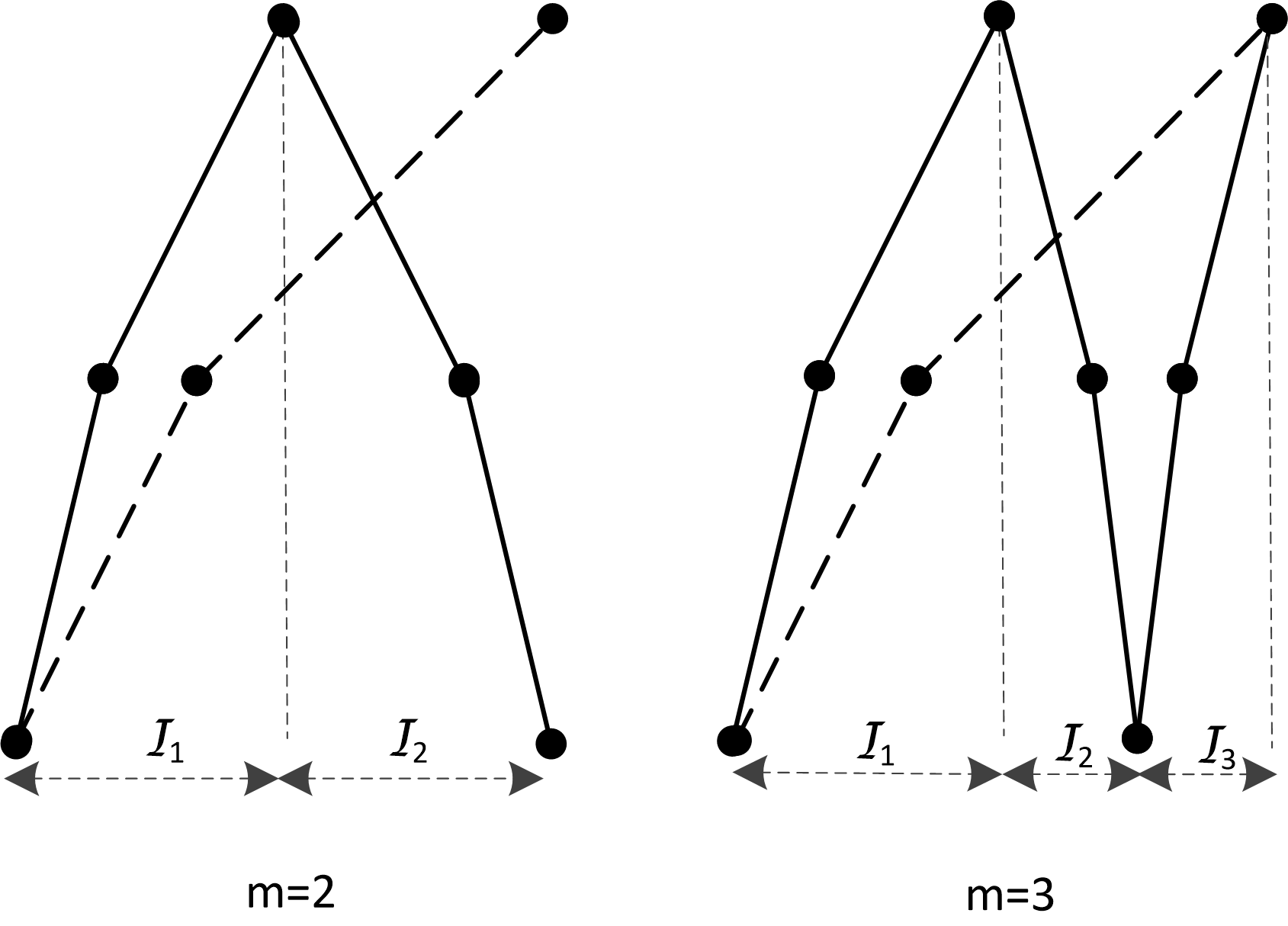}
  \caption{Illustration of generalized $m$-fold window perturbation}
  \label{fig:WeeklyDiscussion_20200316Page-5}
\end{figure} 

\begin{corollary}
Lemma~\ref{lemma:typeI_straighten_2a} and Lemma~\ref{lemma:typeI_straighten_2} hold
if $g$ is a generalized $m$-fold window perturbation instead of an $m$-fold window perturbation on interval $\mathcal{I}$, except that in the conclusion $g_1$ is a piecewise affine segment, instead of an affine segment, on $\mathcal{J}$ where $g_1(\mathcal{J})\simeq g(\mathcal{I}_0)$ or $g_1(\mathcal{J})\simeq g(\hat{\mathcal{I}_0})$ with $\mathcal{I}_0$ being one component interval of $\mathcal{I}$. Whether  $g_1(\mathcal{J})\simeq g(\mathcal{I}_0)$ or $g_1(\mathcal{J})\simeq g(\hat{\mathcal{I}_0})$ depends on the continuity of $g_1$.
\label{lemma:typeI_straighten_2b}
\end{corollary}

Recall that Corollary~\ref{corollary:typeI_straighten_1} and Lemma~\ref{lemma:typeI_straighten_2_1} consider two cases respectively: either $\forall y\in g(\mathcal{I})$, $g^{-1}(y)\subset\mathcal{I}$ or $\forall y\in g(\mathcal{I})$, $g^{-1}(y)\not\subset\mathcal{I}$. Corollary~\ref{lemma:typeI_straighten_4} addresses a mixed case using the notion of generalized window perturbations.

\begin{corollary}
Let $g\in\mathbb{G}$. Let interval $\mathcal{I}\subset [0,1]$ and $\mathcal{Y}=g(\mathcal{I})$. Suppose that $\mathcal{I}=\bigcup_{i=1}^m \mathcal{I}_i$ where $\{\mathcal{I}_i\}$ are $m$ intervals with mutually disjoint interiors and $g$ is monotone on every $\mathcal{I}_i$. Suppose that $c\in (\mathcal{Y}^0, \mathcal{Y}^1)$ exists such that $\forall y\in [\mathcal{Y}^0, c)$, $g^{-1}(y)\subset\mathcal{I}$, and $\forall y\in [c, \mathcal{Y}^1]$, $g^{-1}(y)\not\subset\mathcal{I}$. 
Then $f_1\in\mathbb{F}, g_1\in\mathbb{G}$, and an $m$-fold window perturbation map $w_1$ and interval $\mathcal{J}$ exists such that $g=g_1(w_1(f_1))$, the graph of $g_1$ on $\mathcal{J}$ consists of two affine segments connected by a type I breakpoint, and $g_1([0,\mathcal{J}^0])\simeq g([0,\mathcal{I}^0]), g_1([\mathcal{J}^1,1])\simeq g([\mathcal{I}^1,1])$.
\label{lemma:typeI_straighten_4}
\end{corollary}

\begin{proof}
Because $g$ is monotone on every $\mathcal{I}_i$, $\mathcal{I}_i$ can be partitioned into $\mathcal{I}_i=\mathcal{I}_{i,0}\cup\mathcal{I}_{i,1}$ where $g(\mathcal{I}_{i,0})=[\mathcal{Y}^0, c]$ and $g(\mathcal{I}_{i,1})=[c, \mathcal{Y}^1]$. Combining Corollary~\ref{corollary:typeI_straighten_1} and Lemma~\ref{lemma:typeI_straighten_2_1}, $f_1\in\mathbb{F}, g_2\in\mathbb{G}$ and interval $\mathcal{J}=\bigcup_{i=1}^m(\mathcal{J}_{i,0}\cup\mathcal{J}_{i,1})$, where $\{\mathcal{J}_{i,0},\mathcal{J}_{i,1}\}$ are intervals with mutually disjoint interiors, exist such that $g_2(f_1)=g$, and $g_2([0,\mathcal{J}^0])\simeq g([0,\mathcal{I}^0]), g_2([\mathcal{J}^1,1])\simeq g([\mathcal{I}^1,1])$. For $i=1, 2, \ldots, m$, the graph of $g_2$ on $\mathcal{J}_{i,1}$ is an affine segment whose slope has the absolute value $2^{k''_i}$ with $\sum_{i=1}^m 2^{-k''_i}=2^{-K''}$ for some positive integer $K''$, and the graph of $g_2$ on $\mathcal{J}_{i,0}$ is an affine segment with the absolute value of the slope equal to $2^{k''_i-K''}$. Recall that $\sum_{i=1}^m 2^{-k''_i+K''}=1$. Hence, $g_2$ is a generalized $m$-fold window perturbation on $\mathcal{J}$ where $\mathcal{J}_{i,0}\cup\mathcal{J}_{i,1}$ is a component interval. By Corollary~\ref{lemma:typeI_straighten_2b}, an $m$-fold window perturbation $w_1$ on $\mathcal{J}$ exists such that $g_2=g_1(w_1)$. Interval $\mathcal{J}$ can be partitioned to two intervals $\mathcal{J}=\mathcal{J}_0\cup\mathcal{J}_1$ with mutually disjoint interiors such that the graph of $g_1$  on $\mathcal{J}_0$ is affine with absolute value of the slope equal to $1$ and the graph of $g_1$ on $\mathcal{J}_1$ is affine with absolute value of the slope equal to $2^{K''}$.
\end{proof}

Corollary~\ref{lemma:typeI_straighten_4} holds under a slightly modified hypothesis: $\forall y\in [\mathcal{Y}^0, c]$, $g^{-1}(y)\not\subset\mathcal{I}$, and $\forall y\in (c, \mathcal{Y}^1]$, $g^{-1}(y)\subset\mathcal{I}$. One can further extend the result to a scenario where $c_1, c_2\in (\mathcal{Y}^0, \mathcal{Y}^1)$ exist with $c_1<c_2$ such that $\forall y\in [\mathcal{Y}^0, c_1)\cup(c_2, \mathcal{Y}^1]$, $g^{-1}(y)\subset\mathcal{I}$, and $\forall y\in [c_1, c_2]$, $g^{-1}(y)\not\subset\mathcal{I}$. The same conclusion as in Corollary~\ref{lemma:typeI_straighten_4} holds except that $g_1$ on $\mathcal{I}$ consists of three affine segments connected by two type I breakpoints.

In Lemma~\ref{lemma:typeI_straighten_2_1} and Lemma~\ref{lemma:typeI_straighten_2} $g$ is required to be an $m$-fold window perturbation on $\mathcal{I}$. This requirement is relaxed in Lemma~\ref{lemma:typeI_straighten_3}.

\begin{lemma}
Replacing ``$g\in\mathbb{G}$ is an affine segment on interval $\mathcal{I}_{2i}$'' in Lemma~\ref{lemma:typeI_straighten_2_1} and replacing ``$g$ is an $m$-fold window perturbation on interval $\mathcal{I}$'' in Lemma~\ref{lemma:typeI_straighten_2} by ``let $g$ be monotone on interval $\mathcal{I}_{2i}$ for all $i$'', the conclusions in Lemma~\ref{lemma:typeI_straighten_2_1} and in Lemma~\ref{lemma:typeI_straighten_2} still hold. 
\label{lemma:typeI_straighten_3}
\end{lemma}
\begin{proof}
The difference from Lemma~\ref{lemma:typeI_straighten_2_1} and Lemma~\ref{lemma:typeI_straighten_2} is that $g$ is not necessarily an affine segment on $\mathcal{I}_{2i}$. As in the proof of Lemma~\ref{lemma:typeI_straighten_2_1}, partition interval $\mathcal{Y}$ into intervals $\{\mathcal{Y}_j\}$, $j=1, 2, \ldots, n$ such that no breakpoint exists whose $y$-coordinate falls in the interior of any $\mathcal{Y}_j$. Consider $g^{-1}(\mathcal{Y}_j)$. Let $g^{-1}(\mathcal{Y}_j)=\{\mathcal{I}_{j,1}, \mathcal{I}_{j,2}, \ldots, \mathcal{I}_{j,m}, \mathcal{I}_{j,m+1}, \ldots, \mathcal{I}_{j,m+n_j}\}$. By the hypothesis of the lemma, interval $\mathcal{I}_{j,i}\subset\mathcal{I}$ for $i=1, \ldots, m$, and interval $\mathcal{I}_{j,i}\cap\mathcal{I}=\emptyset$ for $i=m+1, \ldots, m+n_j$ and $n_j\ge1$. Let $2^{k_{j,i}}$ be the absolute value of the slope of the affine segment on $\mathcal{I}_{j,i}$.

Let $j^*=\arg\min_{j=1, 2, \ldots, n} n_j$. If $\arg\min_{j=1, 2, \ldots, n} n_j$ is not unique, then pick any one of them as $j^*$. For $j=1, \ldots, n$, let 
\begin{equation}
k''_{j,i}=\left\{
\begin{array}{ll}
k_{j^*,i}, & i=1, 2, \ldots, m+n_{j^*}-1;\\
k_{j^*,m+n_{j^*}}+i-m-n_{j^*}+1, & i=m+n_{j^*}, m+n_{j^*}+1, \ldots, m+n_j-1;\\
k_{j^*,m+n_{j^*}}+n_j-n_{j^*}, & i=m+n_j.
\end{array}
\right. 
\label{eq:k''lemma7}
\end{equation}
It is easy to verify that (\ref{eq:k''k=1}) holds. As in the proof of Lemma~\ref{lemma:typeI_straighten_2_1}, one can find $f_1\in\mathbb{F}, g_1\in\mathbb{G}$ and a partition of $[0,1]$ $\{\mathcal{J}_1<\cdots<\mathcal{J}_{2m+1}\}$, where $\mathcal{J}_{2i}$ is further partitioned to $\mathcal{J}_{2i}=\bigcup_{j=1}^n\mathcal{J}_{j,i}$, such that $g_1(f_1(x))=g(x)$, $g_1(\mathcal{J}_{j,i})\simeq g(\mathcal{I}_{j,i})$, and the absolute value of the slope changes from $2^{k_{j,i}}$ of $g$ on $\mathcal{I}_{j,i}$ to $2^{k''_{j,i}}$ of $g_1$ on $\mathcal{J}_{j,i}$ for $i=1, \ldots, n_j$ and $j=1,\ldots,n$, and remains unchanged from $g$ on $[0,1]\setminus \bigcup_{i=1}^m\mathcal{I}_{2i}$ to $g_1$ on $[0,1]\setminus \bigcup_{i=1}^m\mathcal{J}_{2i}$. 

Note from the preceding construction (\ref{eq:k''lemma7}) that $k''_{j,i}$ does not depend on $j$ when $i=1, 2, \ldots, m$, because $n_{j^*}\ge1$. Thus, the graph of $g_1$ on $\mathcal{J}_{2i}$ is one affine segment and Lemma~\ref{lemma:typeI_straighten_2_1} is applicable to $g_1$ and the conclusion in Lemma~\ref{lemma:typeI_straighten_2_1} still holds. 

Moreover, if $\bigcup_{i=1}^m\mathcal{I}_{2i}$ is an interval, then as in the proof of Lemma~\ref{lemma:typeI_straighten_2}, $\bigcup_{i=1}^m\mathcal{J}_{2i}$ is an interval too, denoted by $\mathcal{J}$. The graph of $g_1$ on $\mathcal{J}$ is an $m$-fold window perturbation. Hence, Lemma~\ref{lemma:typeI_straighten_2} is applicable to $g_1$ and the conclusion in Lemma~\ref{lemma:typeI_straighten_2} still holds. 
\end{proof}

In summary, Lemma~\ref{lemma:typeI_straighten_0} to Lemma~\ref{lemma:typeI_straighten_3} can be used to eliminate type I and type II breakpoints. The following theorem shows that for any $g\in\mathbb{G}$, all interior breakpoints can be eliminated by repetitively applying these lemmas and corollaries. The only $\mathbb{G}$ maps that have no interior breakpoints are the trivial maps $g_{0,+}$ and $g_{0,-}$. 

\begin{theorem} 
Let $g\in\mathbb{G}$. Then $g$ is equal to the composition of a trivial map followed by a combination of $\mathbb{F}$ maps and window perturbations.
\label{theorem:1}
\end{theorem}

\begin{proof}
Suppose that Lemma~\ref{lemma:typeI_straighten_0} to Lemma~\ref{lemma:typeI_straighten_3} have been applied to eliminate the breakpoints of $g$ such that $g=g_1\circ f_1\circ w_1 \circ f_2\circ w_2\circ \cdots$ where $g_1\in\mathbb{G}$ and $f_1\circ w_1 \circ f_2\circ w_2\circ \cdots$ represent a combination of $\mathbb{F}$ maps and window perturbations. Assume that no more interior breakpoints in $g_1$ can be eliminated using the preceding lemmas. Next we show that $g_1$ has no interior breakpoints.

Denote by $A_0$ the point of $g_1$ at $A_{0,x}=0$. Without loss of generality, suppose that the derivative at $A_0$ is positive. As $x$ increases from $0$, $g_1(x)$ increases until it reaches another point $A_1$ where $g_1$ stops increasing. If $A_{1,x}=1$, then $g_1=g_{0,+}$ and the proof is done. Otherwise, $A_1$ must be a type II breakpoint and the right derivative is negative. As $x$ increases from $A_{1,x}$, $g_1(x)$ decreases until it reaches another type II breakpoint $A_2$. 

First, suppose that $A_{2,y}\le A_{0,y}$. Then a unique point $B_2$ exists where $A_{1,x}<B_{2,x}<A_{2,x}$ and $B_{2,y}=A_{0,y}$, as shown in Figure~\ref{fig:WeeklyDiscussion_20200316Page-6}(a). Consider the following three cases illustrated respectively by three dash lines coming out of point $A_2$ in Figure~\ref{fig:WeeklyDiscussion_20200316Page-6}(a).

\begin{figure}
 \centering
  \includegraphics[width=12cm]{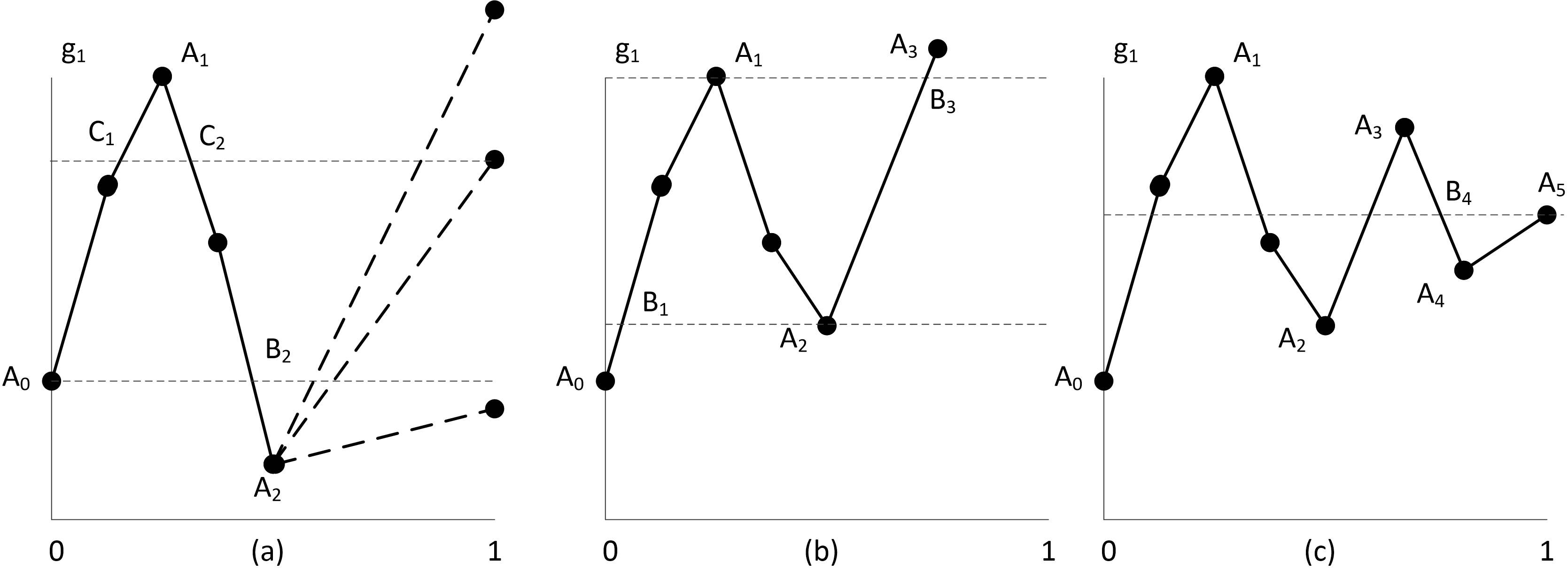}
  \caption{Proof of Theorem~\ref{theorem:1}. In (c), $n=5$.}
  \label{fig:WeeklyDiscussion_20200316Page-6}
\end{figure} 
\begin{itemize}
\item
Assume $\max(g_1([A_{2,x},1]))< A_{0,y}$. Then $\forall y\in[A_{0,y},A_{1,y}]$, $\nexists x\in(B_{2,x},1]$ such that $g_1(x)=y$. That is, $g_1^{-1}(y)$ consists of two points $x_1, x_2$ where $x_1\in[A_{0,x}, A_{1,x}], x_2\in[A_{1,x}, B_{2,x}]$. Thus, $A_0 A_1$ and $A_1 B_2$ must be affine segments with slopes $2, -2$ respectively. By Lemma~\ref{lemma:typeI_straighten_2a}, breakpoint $A_1$ can be eliminated with a $2$-fold window perturbation. Contradiction.
\item
Assume $A_{0,y}\le\max(g_1([A_{2,x},1]))\le A_{1,y}$. Let $C_1, C_2$ be points between $A_0, A_1$ and between $A_1, B_2$ respectively such that $C_{1_y}=C_{2,y}=\max(g_1([A_{2,x},1]))$. Then $\forall y\in(\max(g_1([A_{2,x},1])),A_{1,y}]$, $\nexists x\in(B_{2,x},1]$ such that $g_1(x)=y$. Thus, $C_1 A_1$ and $A_1C_2$ must be affine segments with slopes $2, -2$ respectively. On the other hand, $\forall y\in[A_{0,y}, \max(g_1([A_{2,x},1]))]$, $\exists x\in(B_{2,x},1]$ such that $g_1(x)=y$. By Lemma~\ref{lemma:typeI_straighten_2_1}, one can eliminate all type I breakpoints, if any, between $A_0, C_1$ and between $C_2, B_2$. By Lemma~\ref{lemma:typeI_straighten_0}, make $A_0C_1$ and $C_2B_2$ affine segments with the same slope except for the sign. Therefore, the graph of $g_1$ on $A_0C_1A_1C_2B_2$ is a generalized $2$-fold window perturbation. By Corollary~\ref{lemma:typeI_straighten_2b}, breakpoint $A_1$ can be eliminated with a $2$-fold window perturbation. Contradiction.
\item
Assume $\max(g_1([A_{2,x},1]))> A_{1,y}$. By Lemma~\ref{lemma:typeI_straighten_2_1}, one can eliminate all type I breakpoints, if any, between $A_0, A_1$ and between $A_1, B_2$. By Lemma~\ref{lemma:typeI_straighten_0}, make $A_0A_1$ and $A_1B_2$ affine segments with the same slope except for the sign. Therefore, the graph of $g_1$ on $A_0A_1B_2$ is a $2$-fold window perturbation. By Lemma~\ref{lemma:typeI_straighten_2a}, breakpoint $A_1$ can be eliminated with a $2$-fold window perturbation. Contradiction.
\end{itemize}

In the following, suppose that $A_{2,y}> A_{0,y}$. Then a unique point $B_1$ exists where $A_{0,x}<B_{1,x}<A_{1,x}$ and $B_{1,y}=A_{2,y}$. As $x$ increases from $A_{2,x}$, $g_1(x)$ increases until it reaches another type II breakpoint $A_3$. If $A_{3,y}\ge A_{1,y}$, then a unique point $B_3$ exists where $A_{2,x}<B_{3,x}<A_{3,x}$ and $B_{3,y}=A_{1,y}$, as shown in Figure~\ref{fig:WeeklyDiscussion_20200316Page-6}(b). By Corollary~\ref{lemma:typeI_straighten_4} and Lemma~\ref{lemma:typeI_straighten_3}, type II breakpoints $A_1$ and $A_2$ can both be eliminated. Contradiction.

Therefore, $A_{3,y}< A_{1,y}$. The process continues as shown in Figure~\ref{fig:WeeklyDiscussion_20200316Page-6}(c). For odd $i$, type II breakpoint $A_i$ is facing down and $A_{i,y}>A_{i+2,y}$. For even $i$, type II breakpoint $A_i$ is facing up and $A_{i,y}<A_{i+2,y}$. $A_{2i+1,y}>A_{2j,y}$ for any $i,j$. Suppose that $A_n$ is the endpoint where $A_{n,x}=1$. $\min(A_{n-1,y},A_{n-2,y})<A_{n,y}<\max(A_{n-1,y},A_{n-2,y})$. Therefore, a unique point $B_{n-1}$ exists where $A_{n-2,x}<B_{n-1,x}<A_{n-1,x}$ and $B_{n-1,y}=A_{n,y}$. By Lemma~\ref{lemma:typeI_straighten_2_1}, one can eliminate all type I breakpoints, if any, between $B_{n-1}, A_{n-1}$ and between $A_{n-1}, A_n$. By Lemma~\ref{lemma:typeI_straighten_0}, make $B_{n-1}A_{n-1}$ and $A_{n-1}, A_n$ affine segments with the same slope except for the sign. Therefore, the graph of $g_1$ on $B_{n-1}A_{n-1}A_n$ is a $2$-fold window perturbation. By Lemma~\ref{lemma:typeI_straighten_2a}, breakpoint $A_{n-1}$ can be eliminated with a $2$-fold window perturbation. Contradiction.

Hence, we conclude that $g_1$ has no interior breakpoints.
\end{proof}

Because any $\mathbb{F}$ map can be generated by the two generators defined in (\ref{eq:fAfB}), it suffices to study the basic maps in $\mathbb{G}$ to generate the window perturbations thanks to Theorem~\ref{theorem:1}.

Denote by $w_{m,\mathcal{J}}$ an $m$-fold window perturbation map where $w_{m,\mathcal{J}}(x)=x$ for $x\in[0,1]\setminus \mathcal{J}$ and $w_{m,\mathcal{J}}(x)$ is an $m$-fold window perturbation on $\mathcal{J}$. Specifically, let $\{\mathcal{J}_1<\mathcal{J}_2<\cdots<\mathcal{J}_m\}$ be a partition of $\mathcal{J}$. The graph of $w_{m,\mathcal{J}}$ is an affine segment on each $\mathcal{J}_i$ with slope $(-1)^{i-1} 2^{k_i}$ where $\sum_{i=1}^m 2^{-k_i}=1$. The graph of $w_{m,\mathcal{J}}$ on $\mathcal{J}_i$ is referred to as the $i$-th leg of $w_{m,\mathcal{J}}$. Map $w_{m,\mathcal{J}}$ defined here is from the lower left corner to the upper right corner. Map $1-w_{m,\mathcal{J}}$, which is from the upper left corner to the lower right corner, can be generated by $g_{0,-} (w_{m,\mathcal{J}})$.

\begin{lemma}
Any $(m+2)$-fold window perturbation map $w_{m+2,\mathcal{J}}$ on interval $\mathcal{J}$ is equal to $w_{m,\mathcal{J}}(w(f))$ where $w$ is a $3$-fold window perturbation $w$ and $f\in\mathbb{F}$. 
\label{lemma:mtom+2}
\end{lemma}
\begin{proof}
Let $w_{m,\mathcal{J}}$ be a $m$-fold window perturbation. Let $\mathcal{J}_m$ be the interval corresponding to the $m$-th leg of $w_{m,\mathcal{J}}$. Let $w_{3,\mathcal{J}_m}$ be a $3$-fold window perturbation on $\mathcal{J}_m$ with the absolute values of the slopes being $2^{q_j}$ for $j=1,2,3$ on the three legs respectively. By definition, $\sum_{j=1}^{3}2^{-q_j}=1$. 

By construction, $w_{m,\mathcal{J}}(w_{3,\mathcal{J}_m})$ is an $(m+2)$-fold window perturbation map on interval $\mathcal{J}$ where $w_{m,\mathcal{J}}(w_{3,\mathcal{J}_m})(x)=w_{m,\mathcal{J}}(x)$ for $x\in[0,1]\setminus \mathcal{J}_m$ and $w_{m,\mathcal{J}}(w_{3,\mathcal{J}_m})(x)$ consists of three legs on $\mathcal{J}_m$ whose slopes are $(-1)^{m-1+j} 2^{k_m+q_j}$ for $j=1,2,3$. Interval $\mathcal{J}_m$ is thus partitioned to three intervals $\mathcal{J}_{m_j}$ corresponding to the three legs.

Let $\mathcal{J}'_1<\mathcal{J}'_2<\cdots<\mathcal{J}'_{m+2}$ be the partition of $\mathcal{J}$ of any desired $(m+2)$-fold window perturbation $w_{m+2,\mathcal{J}}$, which is an affine segment on each $\mathcal{J}'_i$ with slope $(-1)^{i-1} 2^{l_i}$. Because 
\[
\sum_{i=1}^{m+2} 2^{-l_i}=\sum_{i=1}^{m-1} 2^{-k_i}+\sum_{j=1}^{3}2^{-k_m-q_j}=1, 
\]
by Lemma~\ref{lemma:typeI_straighten_0}, $f\in\mathbb{F}$ exists to map $\mathcal{J}'_i$ to $\mathcal{J}_i$ for $i=1, \ldots, m-1$ and $\mathcal{J}'_{m-1+j}$ to $\mathcal{J}_{m_j}$ for $j=1, 2, 3$ without altering anything on $[0,1]\setminus\mathcal{J}$. Hence, $w_{m+2,\mathcal{J}}=w_{m,\mathcal{J}}(w_{3,\mathcal{J}_m}(f))$.
\end{proof}

From Lemma~\ref{lemma:mtom+2}, any $m$-fold window perturbation $w_{m,\mathcal{J}}$ can be generated by repetitively applying $3$-fold window perturbations on appropriate intervals of $\mathcal{J}$ to $1$-fold $w_{1,\mathcal{J}}$ for odd $m$ or $2$-fold window perturbation $w_{2,\mathcal{J}}$ for even $m$. The $1$-fold window perturbation $w_{1,\mathcal{J}}$ is simply $g_{0,+}$ or $g_{0,-}$. Next we show that all $3$-fold or $2$-fold window perturbations can be generated by a finite number of basic window perturbations.

Define the basic $3$-fold window perturbation
$\bar{w}_{3,[\frac{1}{4},\frac{1}{2}]}$ as the special case of $w_{3,[\frac{1}{4},\frac{1}{2}]}$ with the absolute values of the slopes being $2, 4, 4$ on the three legs respectively. Lemma~\ref{lemma:3foldgenerator} states that almost any $3$-fold window perturbation can be generated by the the basic $3$-fold window perturbation. The remaining cases of $3$-fold window perturbations are addressed in Lemma~\ref{lemma:3foldgenerator0}.

\begin{lemma}
Any $3$-fold window perturbation $w_{3,\mathcal{J}}$ is equal to $f_1(\bar{w}_{3,[\frac{1}{4},\frac{1}{2}]}(f_2))$ for $f_1, f_2\in\mathbb{F}$ if $0<\mathcal{J}^0<\mathcal{J}^1<1$.
\label{lemma:3foldgenerator}
\end{lemma}
\begin{proof}
We prove the lemma by construction as illustrated in Figure~\ref{fig:WeeklyDiscussion_20200406Page-6}. 
\begin{figure}
 \centering
  \includegraphics[width=14cm]{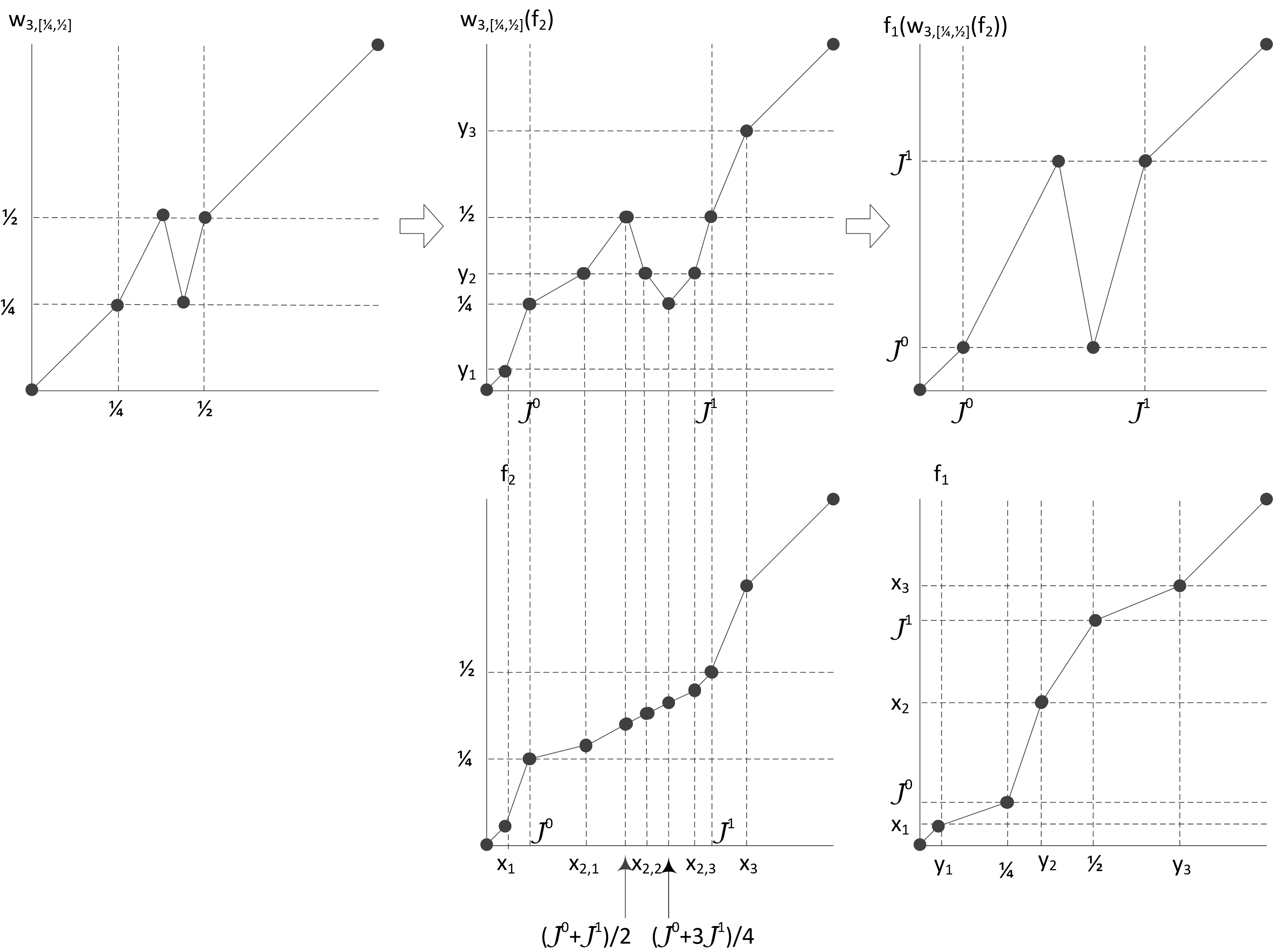}
  \caption{Construction of $w_{3,\mathcal{J}}$ with $\bar{w}_{3,[\frac{1}{4},\frac{1}{2}]}$.}
  \label{fig:WeeklyDiscussion_20200406Page-6}
\end{figure}

Map $f_2$, which scales $\bar{w}_{3,[\frac{1}{4},\frac{1}{2}]}$ horizontally to $\bar{w}_{3,[\frac{1}{4},\frac{1}{2}]}(f_2)$, does the following.
\begin{enumerate}[label=(\alph*)]
\item
Map $[0, \mathcal{J}^0]$ to $[0,\frac{1}{4}]$. If $\frac{\frac{1}{4}}{\mathcal{J}^0}$ is in the form of $2^k$, then $f_2$ on $[0, \mathcal{J}^0]$ is an affine segment; otherwise, $f_2$ on $[0, \mathcal{J}^0]$ consists of two affine segments separated by point $(x_1,y_1)$, a partition point between points $(0,0)$ and $(\mathcal{J}^0, \frac{1}{4})$ by Lemma~\ref{lemma:x1x2x3}.
\item
Map $[\mathcal{J}^0, \mathcal{J}^1]$ to $[\frac{1}{4}, \frac{1}{2}]$. If $\frac{\frac{1}{2}-\frac{1}{4}}{\mathcal{J}^1-\mathcal{J}^0}$ is in the form of $2^k$, then $f_2$ on $[\mathcal{J}^0, \mathcal{J}^1]$ is an affine segment; otherwise, a point $(x_2,y_2)$ exists in Lemma~\ref{lemma:x1x2x3} such that $\frac{\frac{1}{2}-y_2}{\mathcal{J}^1-x_2}$ and $\frac{y_2-\frac{1}{4}}{x_2-\mathcal{J}^0}$ are both in the form of $2^k$. The graph of $f_2$ on $[\mathcal{J}^0, \mathcal{J}^1]$ consists of six affine segments separated by partition points $(x_{2,1}, \frac{1}{4}+\frac{y_2-\frac{1}{4}}{2})$, $(\frac{\mathcal{J}^0+\mathcal{J}^1}{2}, \frac{3}{8})$, $(x_{2,2}, \frac{3}{8}+\frac{\frac{1}{2}-y_2}{4})$, $(\frac{\mathcal{J}^0+3\mathcal{J}^1}{4}, \frac{7}{16})$ and $(x_{2,3}, \frac{7}{16}+\frac{y_2-\frac{1}{4}}{4})$ between points $(\mathcal{J}^0, \frac{1}{4})$ and $(\mathcal{J}^1, \frac{1}{2})$, where $x_{2,1}=\mathcal{J}^0+\frac{x_2-\mathcal{J}^0}{2}$, $x_{2,2}=\frac{\mathcal{J}^0+\mathcal{J}^1}{2}+\frac{\mathcal{J}^1-x_2}{4}$ and $x_{2,3}=\frac{\mathcal{J}^0+3\mathcal{J}^1}{4}+\frac{x_2-\mathcal{J}^0}{4}$. Note that the derivative of $f_2$ is $\frac{y_2-\frac{1}{4}}{x_2-\mathcal{J}^0}$ on $[\mathcal{J}^0, x_{2,1}]$, $[x_{2,2},\frac{\mathcal{J}^0+3\mathcal{J}^1}{4}]$ and $[\frac{\mathcal{J}^0+3\mathcal{J}^1}{4}, x_{2,3}]$, and is $\frac{\frac{1}{2}-y_2}{\mathcal{J}^1-x_2}$ on $[x_{2,1},\frac{\mathcal{J}^0+\mathcal{J}^1}{2}]$, $[\frac{\mathcal{J}^0+\mathcal{J}^1}{2}, x_{2,2}]$ and $[x_{2,3}, \mathcal{J}^1]$.
\item
Map $[\mathcal{J}^1,1]$ to $[\frac{1}{2},1]$. If $\frac{1-\frac{1}{2}}{1-\mathcal{J}^1}$ is in the form of $2^k$, then $f_2$ on $[\mathcal{J}^1,1]$ is an affine segment; otherwise, $f_2$ on $[\mathcal{J}^1,1]$ consists of two affine segments separated by point $(x_3,y_3)$, a partition point between points $(\mathcal{J}^1, \frac{1}{2})$ and $(1,1)$ by Lemma~\ref{lemma:x1x2x3}.
\end{enumerate}

Map $f_1$, which scales $\bar{w}_{3,[\frac{1}{4},\frac{1}{2}]}(f_2)$ vertically to $f_1(\bar{w}_{3,[\frac{1}{4},\frac{1}{2}]}(f_2))$, does the following.
\begin{enumerate}[label=(\alph*)]
\item
Map $[0,\frac{1}{4}]$ to $[0, \mathcal{J}^0]$. If $\frac{\mathcal{J}^0}{\frac{1}{4}}$ is in the form of $2^k$, then $f_1$ on $[0, \frac{1}{4}]$ is an affine segment; otherwise, by symmetry, $f_1$ on $[0,\frac{1}{4}]$ consists of two affine segments separated by point $(y_1,x_1)$, where $x_1$ and $y_1$ are given in the preceding step (a).
\item
Map $[\frac{1}{4}, \frac{1}{2}]$ to $[\mathcal{J}^0, \mathcal{J}^1]$. If $\frac{\mathcal{J}^1-\mathcal{J}^0}{\frac{1}{2}-\frac{1}{4}}$ is in the form of $2^k$, then $f_1$ on $[\frac{1}{4}, \frac{1}{2}]$ is an affine segment; otherwise, by symmetry, $f_1$ on $[\frac{1}{4}, \frac{1}{2}]$ consists of two affine segments separated by point $(y_2, x_2)$, where $x_2$ and $y_2$ are given in the preceding step (b).
\item
Map $[\frac{1}{2},1]$ to $[\mathcal{J}^1,1]$. If $\frac{1-\frac{1}{2}}{1-\mathcal{J}^1}$ is in the form of $2^k$, then $f_1$ on $[\frac{1}{2},1]$ is an affine segment; otherwise, by symmetry, $f_1$ on $[\frac{1}{2},1]$ consists of two affine segments separated by $(y_3,x_3)$, where $x_3$ and $y_3$ are given in the preceding step (c).
\end{enumerate}

Hence, $f_1$ and $f_2$ scale and translate the $3$-fold window perturbation on $[\frac{1}{4}, \frac{1}{2}]$ to $[\mathcal{J}^0, \mathcal{J}^1]$. Finally, the absolute values of the slopes are $2, 2, 4$ on the three legs of $f_1(\bar{w}_{3,[\frac{1}{4},\frac{1}{2}]}(f_2))$. One can apply $f_3\in\mathbb{F}$ such that $f_1(\bar{w}_{3,[\frac{1}{4},\frac{1}{2}]}(f_2(f_3)))$ achieves any $2^{q_1}, 2^{q_2}, 2^{q_3}$ by Lemma~\ref{lemma:typeI_straighten_0}.
\end{proof}
Next consider $2$-fold window perturbations. 
\begin{lemma}
Any $2$-fold window perturbation $w_{2,\mathcal{J}}$ can be generated by $w_{2,[\frac{3}{4},1]}$ if $\mathcal{J}\subset[0,1]$.
\label{lemma:2foldgenerator}
\end{lemma}
\begin{proof}
As in Lemma~\ref{lemma:3foldgenerator}, it can be shown with scaling construction that any $2$-fold window perturbation $w_{2,\mathcal{J}}$ can be generated with $f_1(w_{2,[\frac{3}{4},1]}(f_2))$ for $f_1, f_2\in\mathbb{F}$ if $0<\mathcal{J}^0<\mathcal{J}^1=1$. On the other hand, $w_{2,\mathcal{J}_1}$ with $0=\mathcal{J}_1^0<\mathcal{J}_1^1<1$ can be generated with $g_{0,-}(w_{2,\mathcal{J}}(g_{0,-}))$ where $\mathcal{J}=[1-\mathcal{J}_1^1,1]$. Hence, the conclusion follows.
\end{proof}
Finally consider the special cases of $3$-fold window perturbations that are not addressed in Lemma~\ref{lemma:3foldgenerator}.
\begin{lemma}
Any $3$-fold window perturbation $w_{3,\mathcal{J}}$ with $\mathcal{J}^0=0$ and/or $\mathcal{J}^1=1$ is equal to the composition of $2$-fold window perturbations and $f\in\mathbb{F}$.
\label{lemma:3foldgenerator0}
\end{lemma}
\begin{proof}
Suppose $\mathcal{J}^1=1$. The case of $\mathcal{J}^0=0$ can be addressed similarly.

Consider the $3$-fold window perturbation $\bar{w}_{3,\mathcal{J}}$ as the special case of $w_{3,\mathcal{J}}$ with the absolute values of the slopes being $2, 4, 4$ on the three legs respectively.
\[
\bar{w}_{3,\mathcal{J}}=w_{2,\mathcal{J}}\left(w_{2,\mathcal{J}_{\frac{1}{2}}}\right),
\]
where interval $\mathcal{J}_{\frac{1}{2}}=[\frac{\mathcal{J}^0+1}{2},1]$. Then any ${w}_{3,\mathcal{J}}=\bar{w}_{3,\mathcal{J}}(f)$ for some $f\in\mathbb{F}$ by Lemma~\ref{lemma:typeI_straighten_0}.
\end{proof}

The following theorem follows from Theorem~\ref{theorem:1} and Lemmas~\ref{lemma:mtom+2}, \ref{lemma:3foldgenerator}, \ref{lemma:2foldgenerator} and \ref{lemma:3foldgenerator0}.
\begin{theorem}
Let $g\in\mathbb{G}$. Then $g$ is equal to the composition of a combination of $g_{0,+}$, $g_{0,-}$, $\bar{w}_{3,[\frac{1}{4},\frac{1}{2}]}$, $w_{2,[\frac{3}{4},1]}$, $w_{2,[0,1]}$ and the two generator maps of $\mathbb{F}$ defined in (\ref{eq:fAfB}). 
\label{theorem:handfulgeneratormaps}
\end{theorem}

Figure~\ref{fig:WeeklyDiscussion_20200406Page-7} plots $g_{0,-}$, $\bar{w}_{3,[\frac{1}{4},\frac{1}{2}]}$, $w_{2,[\frac{3}{4},1]}$, and $w_{2,[0,1]}$.
\begin{figure}
 \centering
  \includegraphics[width=16cm]{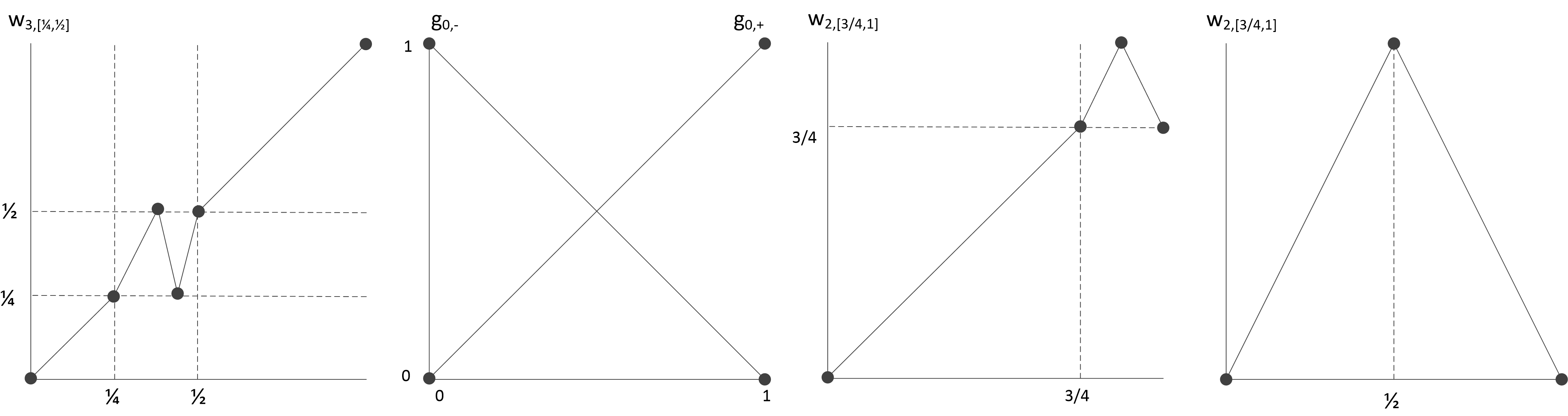}
  \caption{Plots of basic maps in $\mathbb{G}$: $\bar{w}_{3,[\frac{1}{4},\frac{1}{2}]}$, $g_{0,+}$, $g_{0,-}$, $w_{2,[\frac{3}{4},1]}$, and $w_{2,[0,1]}$}
  \label{fig:WeeklyDiscussion_20200406Page-7}
\end{figure} 

Finally, Theorem~\ref{theorem:Gnotfinitelygenerated} shows that unlike $\mathbb{F}$, $\mathbb{G}$ is not finitely generated. To this end, Lemma~\ref{lemma:typeIIbreakpoints} studies the number of type II breakpoints of a composition map in $\mathbb{G}$. Denote by $\#(g)$ the number of type II breakpoints of map $g$.
\begin{lemma}
Let $g_1, g_2 \in \mathbb{G}$. Then $\#(g_1\circ g_2)\ge \#(g_1)+\#(g_2)$.
\label{lemma:typeIIbreakpoints}
\end{lemma}
\begin{proof}
Consider two cases.

\emph{Case $1$}. Suppose that point $B$ on the graph of $g_2$ is a type II breakpoint. There exists $\delta>0$ such that the graph of $g_2$ is an affine segment on $[B_x-\delta, B_x]$ and a different affine segment on $[B_x, B_x+\delta]$ and the slopes of the two affine segments are of different signs. Either $g_2(B_x-\delta)>g_2(B_x)$ and $g_2(B_x+\delta)>g_2(B_x)$, or $g_2(B_x-\delta)<g_2(B_x)$ and $g_2(B_x+\delta)<g_2(B_x)$. Thus a sufficiently small $\delta$ exists such that the graph of $g_1$ is an affine segment on both $\langle g_2(B_x-\delta), g_2(B_x)\rangle$ and $\langle g_2(B_x+\delta), g_2(B_x)\rangle$. Therefore, $g_1(g_2(B_x))$ is a type II breakpoint on the graph of $g_1\circ g_2$. That is, every type II breakpoint of $g_2$ corresponds to at least one type II breakpoint of $g_1\circ g_2$.

\emph{Case $2$}. Suppose that point $A$ on the graph of $g_1$ is a type II breakpoint. Because $g_2$ is a continuous map onto $[0,1]$, point $C$ exists on the graph of $g_2$ such that $C_x \in g_2^{-1}(A_x)$ and $\delta>0$ exists such that the graph of $g_2$ on $[C_x-\delta, C_x+\delta]$ is monotone. Following the preceding argument in case $1$,  $g_1(g_2(C_x))$ is a type II breakpoint on the graph of $g_1\circ g_2$. That is, every type II breakpoint of $g_1$ corresponds to at least one type II breakpoint of $g_1\circ g_2$. Because point $C$ is not a type II breakpoint of $g_2$, this type II breakpoint on the graph of $g_1\circ g_2$ is not included in the preceding case $1$ and there is no double counting between cases $1$ and $2$. Hence, $\#(g_1\circ g_2)\ge \#(g_1)+\#(g_2)$.
\end{proof}

\begin{theorem}
$\mathbb{G}$ is not finitely generated.
\label{theorem:Gnotfinitelygenerated}
\end{theorem}
\begin{proof}
For any dyadic number $\delta \in (0,1), \#(w_{2,[\delta,1]})=1$. If $g_1, g_2\in\mathbb{G}$ exist such that $w_{2,[\delta,1]}=g_1\circ g_2$, then by Lemma~\ref{lemma:typeIIbreakpoints}, one of $g_1$ and $g_2$ is a trivial map. Thus, $w_{2,[\delta,1]}$ cannot be generated by $w_{2,[\delta',1]}$ with another dyadic number $\delta'\neq\delta$ or other maps in $\mathbb{G}$ with two or more type II breakpoints. As there are infinitely many $\delta$, the set of $\{w_{2,[\delta,1]}\}$ cannot be generated by a finite number of generators.
\end{proof}

\section{Equivalence Classes and Generators \label{sec:equivalence}}

Comparison of Theorem~\ref{theorem:handfulgeneratormaps} and Theorem~\ref{theorem:Gnotfinitelygenerated} indicates that the maps in $\mathbb{F}$ play an important role in allowing a finite number of basic maps to generate any map in $\mathbb{G}$. To study this idea formally, define the notion of equivalence relation and equivalence classes.

\begin{definition}[\textsl{Equivalence Relation}]
A binary relation $\sim$ on $\mathbb{G}$ is defined as follows. Suppose that $g_1, g_2 \in \mathbb{G}$. $g_1\sim g_2$ if and only if $f_1, f_2\in\mathbb{F}$ exist such that $g_2=f_1 \circ g_1 \circ f_2$. Binary relation $\sim$ is an equivalence relation because it is reflexive, symmetric and transitive. 
\label{definition:er}
\end{definition}

\begin{definition}[\textsl{Equivalence Class}]
The equivalence class of $g\in\mathbb{G}$, denoted by $[g]$, is the set $\{\hat{g}\in\mathbb{G}| \hat{g}\sim g\}$.
\label{definition:ec}
\end{definition}

To understand the effect of $f_1$ and $f_2$ on $g$ in Definition~\ref{definition:ec}, let $g\in\mathbb{G}$ and $f\in\mathbb{F}$. Suppose that the graph of $f$ is an affine segment on interval $\mathcal{I}_0=f^{-1}(\mathcal{I}_1)$. The slope of the affine segment is $s=\frac{|\mathcal{I}_1|}{|\mathcal{I}_0|}$. From the properties of $\mathbb{F}$, a portion of the graph of $g$ is scaled at a ratio of $s$ to become part of $g\circ f$ or $f\circ g$. Specifically, to obtain $g\circ f$, the graph of $g$ on $\mathcal{I}_1$ is scaled \emph{horizontally} to an interval $\mathcal{I}'_0$ with $|\mathcal{I}'_0|=|\mathcal{I}_0|$. The exact location of $\mathcal{I}'_0$ on $[0,1]$ is such that the continuity is maintained in $g\circ f$ and thus depends on the scaling of other portions. To obtain $f\circ g$, the graph of $g$ on $g^{-1}(\mathcal{I}_0)$ is scaled \emph{vertically} to an interval $\mathcal{I}'_1$ with $|\mathcal{I}'_1|=|\mathcal{I}_1|$. The exact location of $\mathcal{I}'_1$ on $[0,1]$ is such that the continuity is maintained in $f\circ g$. Figure~\ref{fig:WeeklyDiscussion_20200406Page-1} illustrates the scaling operation of an affine segment of $f$.

\begin{figure}
 \centering
  \includegraphics[width=10cm]{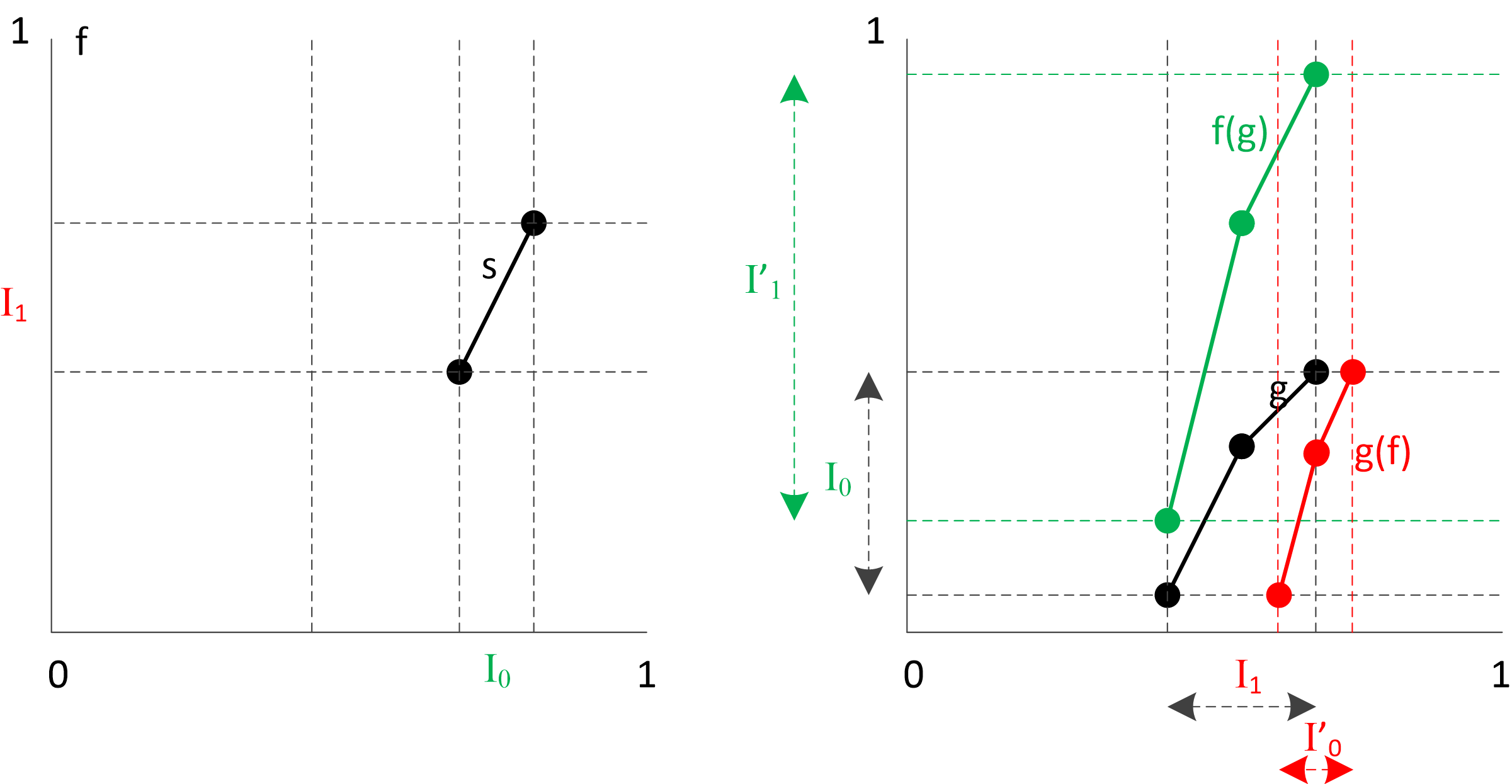}
  \caption{The scaling operation of $g\circ f$ and $f\circ g$ on the graph of $g$ by an affine segment of $f$.}
  \label{fig:WeeklyDiscussion_20200406Page-1}
\end{figure} 

\begin{lemma}
Let $f_1 \in \mathbb{F}$ and $g\in\mathbb{G}$. Then there exists $f_2 \in \mathbb{F}$ such that $f_1 \circ g \circ f_2 \in \mathbb{G}$.
\label{lemma:f1gf2_f2exists}
\end{lemma}
\begin{proof}
Partition $[0,1]$ into a set of intervals $\{\mathcal{Y}_i\}$ such that the interiors of $f_1^{-1}(\mathcal{Y}_i)$ contains no breakpoints of $f_1$ and the endpoints of $\{\mathcal{Y}_i\}$ are all dyadic for all $i$. Suppose that the derivative of $f_1$ on $f_1^{-1}(\mathcal{Y}_i)$ is $s_i$. From $g$ to $f_1\circ g$, $f_1$ vertically scales the graph of $g$ on $g^{-1}(f_1^{-1}(\mathcal{Y}_i))$ by a factor of $s_i$. 

In the rectangle diagram representation of a map $f\in \mathbb{F}$, $[0,1]$ is partitioned into $\{\mathcal{I}_i\}$ and $f\in\mathbb{F}$ is completely defined by specifying a scaling factor from $\mathcal{I}_i$ to $\mathcal{J}_i$, for $i=1, 2, \ldots, $, where $\{\mathcal{J}_i\}$ is another interval partition of $[0,1]$. To construct $f_2$, let the scaling factor be $s_i$ on $g^{-1}(f_1^{-1}(\mathcal{Y}_i))$ for all $i$. Let $g^{-1}(f_1^{-1}(\mathcal{Y}_i))=\bigcup_j \mathcal{I}_{i,j}$ where $\mathcal{I}_{i,1}, \mathcal{I}_{i,2}, \ldots$ are intervals of mutually disjoint interiors. Map $f_2$ scales the graph of $f_1 \circ g$ on $\mathcal{I}_{i,j}$ horizontally to the graph of  $f_1 \circ g\circ f_2$ on $\mathcal{J}_{i,j}$ by $s_i$ and thus $|\mathcal{J}_{i,j}|=s_i |\mathcal{I}_{i,j}|$. Because $g$ is $\lambda$-preserving, 
\[
\sum_j|\mathcal{I}_{i,j}|=|f^{-1}_1(\mathcal{Y}_i)|=\frac{|\mathcal{Y}_i|}{s_i} \Rightarrow \sum_i\sum_j|\mathcal{J}_{i,j}|=\sum_i |\mathcal{Y}_i|=1.
\]
Therefore, $\mathcal{J}_{i,j}$ is a valid partition of $[0,1]$. Moreover, $s_i$ is in the form of $2^k$ for integer $k$ and the endpoints of $g^{-1}(f_1^{-1}(\mathcal{Y}_i))$ are dyadic because $f_1\in \mathbb{F}$ and $g\in \mathbb{G}$. Therefore, $f_2\in \mathbb{F}$. 

From $f_1\circ g$ to $f_1 \circ g \circ f_2$, $f_2$ horizontally scales the graph of $f_1 \circ g$ on $g^{-1}(f_1^{-1}(\mathcal{Y}_i))$ by a factor of $s_i$. Combining the two steps, from $g$ to $f_1 \circ g \circ f_2$, the graph of $g$ on $g^{-1}(f_1^{-1}(\mathcal{Y}_i))$ is scaled horizontally and vertically by the same factor for all $i$. Because $g\in\mathbb{G}$, it follows that $f_1 \circ g\circ f_2 \in \mathbb{G}$.
\end{proof}

\begin{lemma}
The equivalence classes defined in Definition~\ref{definition:ec} form a partition of set $\mathbb{G}$.
\end{lemma}
\begin{proof}
Because any map in $\mathbb{F}$ is invertible, it follows that if $\hat{g}\in[g]$, then $g\in[\hat{g}]$ and $[g]=[\hat{g}]$, and if $\hat{g}_1,\hat{g}_2\in[g]$, then $[\hat{g}_1]=[\hat{g}_2]$. Therefore, any map in $\mathbb{G}$ is in exactly one equivalence class.
\end{proof}

However, the equivalence classes do not form a monoid. Consider the following definition of a binary operation $\odot$ on $[g]$. Let $[g]= [g_1]\odot[g_2]$ where $\hat{g}\in [g]$ if and only if $\hat{g}_1 \in [g_1]$ and $\hat{g}_2 \in [g_2]$ exist such that $\hat{g}\in[\hat{g}_1 \circ \hat{g}_2]$. Example~\ref{example:notec} shows that $[g_1]\odot[g_2]$ is not necessarily a single equivalence class. 

\begin{example}
Let $g_1=w_{2,[\frac{3}{4},1]}$ and $g_2=w_{2,[0,1]}$. Let $\hat{g}_2=w_{2,[0,1]}\in [g_2]$. Consider two elements in equivalence class $[g_1]$: $\hat{g}_{1,1}=w_{2,[\frac{1}{2},1]}\in [g_1]$ and $\hat{g}_{1,2}=w_{2,[\frac{1}{4},1]}\in [g_1]$. Figure~\ref{fig:WeeklyDiscussion_20200406Page-8} compares $\hat{g}_{1,1} \circ \hat{g}_2$ and $\hat{g}_{1,2} \circ \hat{g}_2$ and shows that they are not in the same equivalence class.

\begin{figure}
 \centering
  \includegraphics[width=7cm]{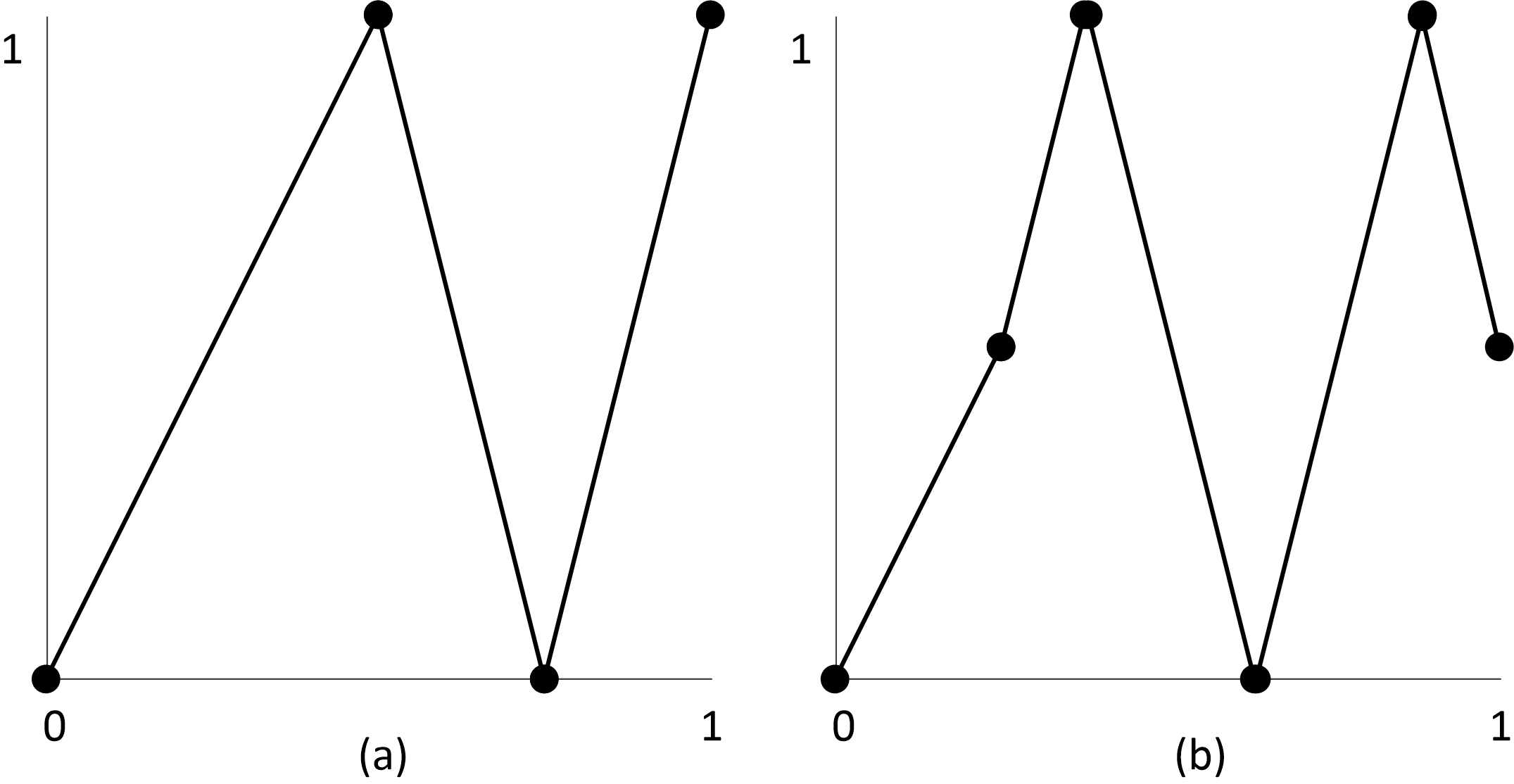}
  \caption{A counterexample to show that (a) $\hat{g}_{1,1} \circ \hat{g}_2$ and (b) $\hat{g}_{1,2} \circ \hat{g}_2$ are not in the same equivalence class. $\hat{g}_{1,1}, \hat{g}_{1,2} \in [g_1]$ where $g_1=w_{2,[\frac{3}{4},1]}$ and $\hat{g}_2=w_{2,[0,1]}$.}
  \label{fig:WeeklyDiscussion_20200406Page-8}
\end{figure}
\label{example:notec}
\end{example}

To avoid the technical difficulty of working with equivalence classes directly, consider the notion of sets of equivalence classes instead.
\begin{definition}[\textsl{Set of Equivalence Classes}]
Let $\Phi\subset \mathbb{G}$. Let $\{[g]_{g\in\Phi}\}$ be the set of equivalence classes $[g]$, $\forall g\in\Phi$. Define a binary operation $\odot$ on $\{[g]_{g\in\Phi}\}$ as follows: $\{[g]_{g\in\Phi_1}\} \odot \{[g]_{g\in\Phi_2}\}$ is the set of equivalence classes $[\hat{g}_1 \circ \hat{g}_2]$ where $g_1\in\Phi_1$ and $g_2\in\Phi_2$ exist such that $\hat{g}_1 \in [g_{1}]$ and $\hat{g}_2 \in [g_{2}]$.
\label{definition:ecset}
\end{definition}

\begin{lemma}
Let $\Phi_1, \Phi_2, \Phi_3 \subset \mathbb{G}$. Then 
\[
\left(\{[g]_{g\in\Phi_1}\} \odot \{[g]_{g\in\Phi_2}\}\right)\odot \{[g]_{g\in\Phi_3} \}= \{[g]_{g\in\Phi_1}\} \odot \left(\{[g]_{g\in\Phi_2}\}\odot \{[g]_{g\in\Phi_3}\}\right).
\]
\label{lemma:setofec_associativity}
\end{lemma}
\begin{proof}
By Definition~\ref{definition:ecset}, if $g \in (\{[g]_{g\in\Phi_1}\} \odot \{[g]_{g\in\Phi_2}\})\odot \{[g]_{g\in\Phi_3} \}$, then $f_1, f_2, \ldots, f_{10} \in \mathbb{F}$ exist such that for some $g_1\in\Phi_1, g_2\in\Phi_2, g_3\in\Phi_3$, 
\[
g=f_1 \circ \left(\left(f_2 \circ \left( \left(f_3 \circ g_1 \circ f_4 \right) \circ  \left(f_5 \circ g_2 \circ f_6 \right)\right) f_7 \right)\circ \left(f_8 \circ g_3 \circ f_9 \right)\right) \circ f_{10}.
\]
Let 
\[
g=f'_1 \circ \left( \left(f'_2 \circ g_1 \circ f'_3 \right) \circ  \left(f'_4 \circ \left(\left( f'_5 \circ g_2 \circ f'_6 \right) \circ \left(f'_7 \circ g_3 \circ f'_8\right) \right) \circ f'_9 \right)\right) \circ f'_{10},
\]
where $f'_1, \ldots, f'_{10}\in\mathbb{F}$ are determined such that 
\[
\left\{
\begin{array}{l}
f_1 \circ f_2 \circ f_3= f'_1 \circ f'_2,\\
f_4 \circ f_5= f'_3 \circ f'_4 \circ f'_5,\\
f_6 \circ f_7 \circ f_8= f'_6 \circ f'_7,\\
f_9 \circ f_{10}= f'_8 \circ f'_9 \circ f'_{10},
\end{array}
\right.
\]
and $f'_2 \circ g_1 \circ f'_3$, $f'_5 \circ g_2 \circ f'_6$, $f'_7 \circ g_3 \circ f'_8$, $f'_4 \circ \left(\left( f'_5 \circ g_2 \circ f'_6 \right) \circ \left(f'_7 \circ g_3 \circ f'_8\right) \right) \circ f'_9$ are all in $\mathbb{G}$. Let $f'_1=g_{0,+}$ and $f'_2=f_1 \circ f_2 \circ f_3$. From Lemma~\ref{lemma:f1gf2_f2exists}, $f'_3$ exists to make $f'_2 \circ g_1 \circ f'_3\in\mathbb{G}$. Next, let $f'_4=g_{0,+}$ and $f'_5 = (f'_3)^{-1}\circ f_4 \circ f_5$ and $f'_6$ exists to make $f'_5 \circ g_2 \circ f'_6\in\mathbb{G}$. Let $f'_7 = (f'_6)^{-1}\circ f_6 \circ f_7 \circ f_8$ and $f'_8$ exists to make $f'_7 \circ g_3 \circ f'_8\in\mathbb{G}$. Finally, $f'_9$ exists such that $f'_4 \circ \left(\left( f'_5 \circ g_2 \circ f'_6 \right) \circ \left(f'_7 \circ g_3 \circ f'_8\right) \right) \circ f'_9 \in\mathbb{G}$. Let $f'_{10}=(f'_9)^{-1}\circ (f'_8)^{-1}\circ f_9 \circ f_{10}$. Hence, by Definition~\ref{definition:ecset}, $g\in \{[g]_{g\in\Phi_1}\} \odot \left(\{[g]_{g\in\Phi_2}\}\odot \{[g]_{g\in\Phi_3}\}\right)$. This completes the proof.
\end{proof}

\begin{theorem}
Let $\Phi_a=\{g_{0,+}\}, \Phi_b=\{g_{0,-}\}, \Phi_c = \{\bar{w}_{3,[\frac{1}{4},\frac{1}{2}]}\}, \Phi_d=\{w_{2,[\frac{3}{4},1]}\}$ and $\Phi_e=\{w_{2,[0,1]}\}$. Construct a collection of sets of equivalence classes, each of which is equal to $\{[g]_{g\in\Phi_1}\} \odot \{[g]_{g\in\Phi_2}\}\odot \cdots $ where $\Phi_i$ for any $i$ is one of $\Phi_a, \Phi_b, \Phi_c, \Phi_d, \Phi_e$. Then, the collection is a monoid and finitely generated. The union of all the elements of the collection is a set of equivalence classes, the union of which is $\mathbb{G}$.
\end{theorem}
\begin{proof}
By Lemma~\ref{lemma:setofec_associativity}, associativity holds for the elements in the collection. Equivalence class set $\{[g]_{g\in\Phi_e}\}$ is the identity element. The collection is thus a monoid. By construction, $\{[g]_{g\in\Phi_a}\}$, $\{[g]_{g\in\Phi_b}\}$, $\{[g]_{g\in\Phi_c}\}$, $\{[g]_{g\in\Phi_d}\}$ and $\{[g]_{g\in\Phi_e}\}$ are the generators of the monoid. By Theorem~\ref{theorem:handfulgeneratormaps}, any map $g\in\mathbb{G}$ is equal to the composition of a combination of maps, each of which is in one of equivalence classes $[g]_{g\in\Phi_a}$, $[g]_{g\in\Phi_b}$, $[g]_{g\in\Phi_c}$, $[g]_{g\in\Phi_d}$ and $[g]_{g\in\Phi_e}$. Therefore, the last part of the theorem holds.
\end{proof}

Next we characterize $[g]$. Partition $[0,1]$ into intervals $\{\mathcal{Y}_l\}$, $l=1, 2, \ldots, m$ with $|\mathcal{Y}_l|>0$ for all $l$ and $\mathcal{Y}_1<\mathcal{Y}_2<\cdots<\mathcal{Y}_m$ such that no breakpoint exists whose $y$-coordinate falls in the interior of any $\mathcal{Y}_l$, i.e., no breakpoint $B$ exists such that $B_y \in(\mathcal{Y}_l^0, \mathcal{Y}_l^1)$ for any $l$. As $x$ increases from $0$ to $1$, $g(x)$ moves from one interval to another or stays in one interval but changes the sign of the derivative. We characterize $g$ by the sequence of the indices, referred to as \emph{evolution sequence}, representing the intervals on which $g(x)$ resides. 

More precisely, let $\mathcal{I}_1<\cdots<\mathcal{I}_n$ be a partition of $[0,1]$ such that for any $i=1, 2, \ldots, n$, a unique $l_i \in \{1, 2, \ldots, m\}$ exists where $g(\mathcal{I}_i)= \mathcal{Y}_{l_i}$ and the graph of $g$ is affine on every $\mathcal{I}_i$. Because no breakpoint exists inside any of $\{\mathcal{Y}_l\}$, the construction of $\{\mathcal{I}_i\}$ exists and is unique. The evolution sequence is defined as $\pm l_1 l_2\cdots l_n$, where the sign $+$ or $-$ represents the sign of the derivative of $g$ in interval $\mathcal{I}_1$. In an evolution sequence $\pm l_1 l_2\cdots l_n$, adjacent indices $l_i$ and $l_{i+1}$ differ at most by $1$. In $\mathcal{I}_i\cup\mathcal{I}_{i+1}$, $g$ is increasing if $l_{i+1}-l_{i}=1$ or decreasing if $l_{i+1}-l_{i}=-1$. If $l_{i+1}-l_{i}=0$, $g$ alternates between increasing and decreasing in $\mathcal{I}_i$ and $\mathcal{I}_{i+1}$.

Suppose that $\hat{g}\in[g]$. From the scaling operation illustrated in Figure~\ref{fig:WeeklyDiscussion_20200406Page-1}, partitions $\{I_i\}$ and $\{I'_i\}$ of $[0,1]$ for $i=1, \ldots, n$ exist such that for all $i$, $\hat{g}(I'_i)\simeq c_ig(I_i)+d_i$ for some numbers $c_i,d_i$. The following lemma follows immediately.
\begin{lemma}
If $g_1, g_2\in\mathbb{G}$ have the same evolution sequence, then $g_1$ and $g_2$ are in the same equivalence class.
\label{lemma:sameEC}
\end{lemma}
The converse of Lemma~\ref{lemma:sameEC} is not true. 

Define the size of an evolution sequence $\pm l_1 l_2\cdots l_n$ as $|\pm l_1 l_2\cdots l_n|=mn$, where $m$ represents the number of intervals $\{\mathcal{Y}_l\}$ and $n$ the number of intervals $\{\mathcal{I}_i\}$. It is easy to show that if partition $\{\mathcal{Y}'_l\}$ is a strict refinement of partition $\{\mathcal{Y}_l\}$, then $m'>m$ and $n'>m$. Therefore, for given $g$, one can minimize $|\pm l_1 l_2\cdots l_n|$ by using only the partition $\{\mathcal{Y}_l\}$ where at least one breakpoint exists at the boundary of any two adjacent intervals. In this case, the evolution sequence is referred to as the \emph{characteristic sequence of $g$}, denoted by $C(g)$, as shown in Figure~\ref{fig:WeeklyDiscussion_20200406Page-3}. Characteristic sequence $C(g)$ is unique for any $g$ given the breakpoints of $g$.

\begin{figure}
 \centering
  \includegraphics[width=8cm]{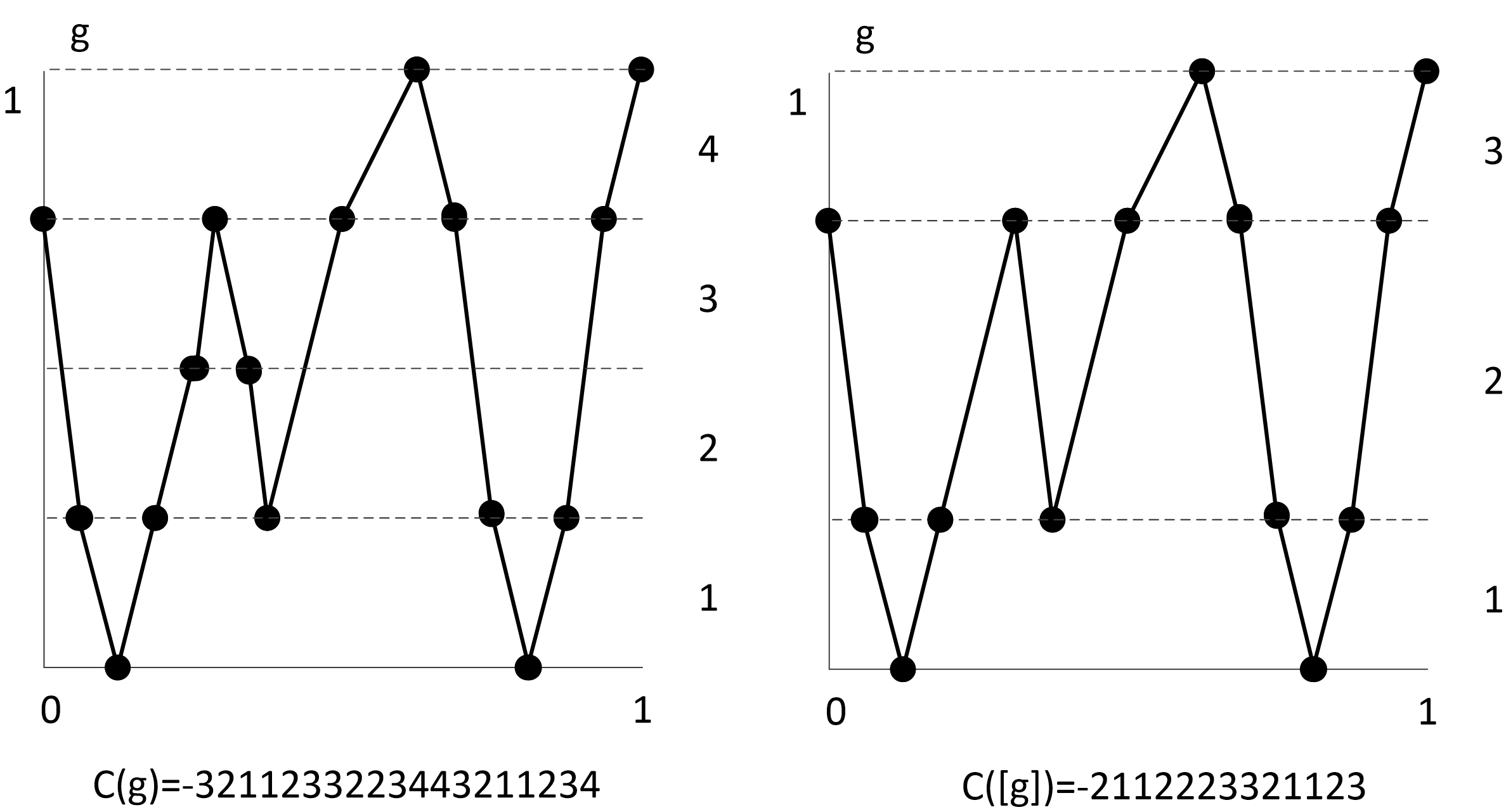}
  \caption{Example of $C(g)$ and $C([g])$. }
  \label{fig:WeeklyDiscussion_20200406Page-3}
\end{figure}

Furthermore, not all elements in $[g]$ have the same size of the characteristic sequence, because $f_1(g(f_2))$ adds or removes type I breakpoints as seen in Section~\ref{sec:generators}. The one of the minimum size is referred to as the \emph{characteristic sequence of $[g]$}, denoted by $C([g])$. That is, suppose that 
\begin{equation}
\hat{g}^*=\arg\min_{\hat{g}\in[g]} |C(\hat{g})|.
\label{eq:g^*}
\end{equation} 
Then
\begin{equation}
C([g])=C(\hat{g}^*). 
\label{eq:C[g]}
\end{equation}
In (\ref{eq:g^*}), $|C(\hat{g})|$ represents the size of $C(\hat{g})$. It may appear that if $\hat{g}^*_1$ and $\hat{g}^*_2$ both satisfy (\ref{eq:g^*}), $C(\hat{g}^*_1)$ and $C(\hat{g}^*_2)$ are not necessarily equal, even though $|C(\hat{g}^*_1)|=|C(\hat{g}^*_2)|$. If so, then $C[g]$ in (\ref{eq:C[g]}) would not be well defined. However, Theorem~\ref{theorem:C[g]} shows that $C(\hat{g}^*_1)=C(\hat{g}^*_2)$. Therefore, $C([g])$ defined in (\ref{eq:C[g]}) is unique given $g$.

\begin{theorem}
If $g^*_1$ and $g^*_2$ both satisfy (\ref{eq:g^*}), then $C(\hat{g}^*_1)=C(\hat{g}^*_2)$.
\label{theorem:C[g]}
\end{theorem}
\begin{proof}
Let $\mathcal{A}=\{0=y_0<y_1<\cdots<y_m=1\}$ be a set where if $x$ is a type II breakpoint of $g$ or $x=0,1$ then $g(x)\in \mathcal{A}$. Suppose that $x_1$ is a type I breakpoint and $g(x_1)\not\in\mathcal{A}$. Let $\mathcal{Y}$ be an interval such that $g(x_1)\in\mathcal{Y}$ and $\mathcal{Y}\cap\mathcal{A}=\emptyset$. Breakpoint $x_1$ can be eliminated by applying Corollary~\ref{lemma:typeI_straighten_1} on $\mathcal{Y}$. Applying Corollary~\ref{lemma:typeI_straighten_1} to eliminate all such type I breakpoints by some $f\in\mathbb{F}$, one obtains $\hat{g}=g\circ f$. When $\hat{g}$ cannot be simplified with Corollary~\ref{lemma:typeI_straighten_1}, for any interval $\mathcal{Y}$, either $\mathcal{Y}\cap\mathcal{A}\neq\emptyset$ or $\hat{g}^{-1}(\mathcal{Y})$ contains no breakpoints.

For $\hat{g}$, let partition $\{\mathcal{Y}_l\}$ where $\mathcal{Y}_l=[y_{l-1},y_l]$ for $l=1,\ldots,m$. Characteristic sequence $C(\hat{g})$ can be obtained with partition $\{\mathcal{Y}_l\}$, because if $x$ is a breakpoint of $\hat{g}$ then $\hat{g}(x)\in\mathcal{A}$. On the other hand, recall that $\hat{g}_1=f_1(g(f_2))$ with any $f_1, f_2\in\mathbb{F}$ scales $g$ horizontally and vertically, and thus does not add or eliminate any type II breakpoint or endpint. The size of any partition of $[0,1]$ to obtain $C(\hat{g}_1)$ thus cannot be smaller than $|\{\mathcal{Y}_l\}|$. Therefore, $|C(\hat{g})|\le |C(\hat{g}_1)|$ and $\hat{g}$ satisfies (\ref{eq:g^*}). 

Given $g$, $C(\hat{g}^*)$ is unique in the sense that for $\hat{g}^*_1, \hat{g}^*_2\in[g]$, if neither $\hat{g}^*_1$ nor $\hat{g}^*_2$ can be simplified with Corollary~\ref{lemma:typeI_straighten_1}, then $C(\hat{g}^*_1)=C(\hat{g}^*_2)$. The reason is that the partition set $\{\mathcal{Y}_l\}$ defined for $\hat{g}^*_1$ can be scaled by some $f\in\mathbb{F}$ to become that defined for $\hat{g}^*_2$. As a result, $\hat{g}^*_1$ and $\hat{g}^*_2$ have the same evolution sequence.
\end{proof}
 
The following corollary from Lemma~\ref{lemma:sameEC} and Theorem~\ref{theorem:C[g]} provides a simple way to check whether $g_1$ and $g_2$ are in the same equivalence class.
\begin{corollary}
$g_1, g_2\in\mathbb{G}$ are in the same equivalence class if and only if $C([g_1])=C([g_2])$.
\label{corollary:sameEC}
\end{corollary}

\section{Topological Conjugacy \label{sec:conjugacy}}

\begin{definition}[\textsl{Topological Conjugacy}]
Continuous maps $s_1$ and $s_2$ from $[0,1]$ to $[0,1]$ are topologically conjugate if there exists a homeomorphism $h$ such that $s_2 = h \circ s_1 \circ h^{-1}$.
\label{definition:topologicalconjugacy}
\end{definition}
\begin{remark}
Homeomorphism $h$ represents a change of coordinates between $s_1$ and $s_2$. From $s_2 = h \circ s_1 \circ h^{-1}$, it follows that $s^n_2 = h \circ s^n_1 \circ h^{-1}$ for $n\ge0$. Therefore, topologically conjugate $s_1$ and $s_2$ share the same dynamics from the topological viewpoint. 
\end{remark}

It is not always an easy task to determine the topological conjugacy of $s_1$ and $s_2$ directly by the definition. However, for linear or expanding Markov maps, this task is reduced to comparing the index maps $s^*_1$ and $s^*_2$ associated with $s_1$ and $s_2$. A continuous map $s$ is \emph{linear Markov} if it is piecewise affine and the set $P$ of all $s^k(x)$, where $k\ge0$ and $x$ is an endpoint of an affine piece, is finite. A continuous map $s$ is \emph{expanding Markov} if it is piecewise monotone, the set $P$ of all $s^k(x)$, where $k\ge0$ and $x$ is an endpoint of a monotone piece, is finite, and there is a constant $c>1$ such that $|s(x)-s(y)|\ge c|x-y|$ whenever $x$ and $y$ lie in the same monotone piece. Set $P$ is the orbit of all endpoints. The orbit of point $x$ is defined as the set $\{s^k(x) |k\ge0\}$ of any map $s$.

Let $P=\{0=x_0<x_1<\cdots<x_N=1\}$ and $P^*=\{0, 1, \ldots, N\}$. Define index map $s^*: P^* \rightarrow P^*$ by 
\begin{equation}
s^*(i)=j, \mbox{ if } s(x_i)=x_j
\label{eq:s^*definition} 
\end{equation}
for $i=0, 1, \ldots, N$. 

\begin{theorem} [\textsl{Block and Coven, 1987, \cite[Theorem.\ 2.7]{10.2307/2000600}}]
Linear or expanding Markov maps $s_1$ and $s_2$ are topologically conjugate if and only if $s^*_1=s^*_2$ or $s^*_1={}^*s_2$.
\label{theorem:s*1s*2}
\end{theorem}
Here ${}^*s$ is the reverse of $s^*$, defined by ${}^*s(i)=N-s^*(N-i)$. From Theorem~\ref{theorem:s*1s*2}, a linear or expanding Markov map $s$ is characterized by $s^*$ as far as topological conjugacy is concerned.
The following proposition connects the notions of topological conjugacy and equivalence classes.
\begin{proposition}
Suppose that $g_1, g_2\in\mathbb{G}$. If $g_1^*=g_2^*$, then $g_1$ and $g_2$ are in the same equivalence class. 
\label{proposition:tcec}
\end{proposition}
\begin{proof}
Because $g_1^*=g_2^*$, $|g_1^*|=|g_2^*|$. Let $N=|g_1^*|=|g_2^*|$. For $k=1,2$, let $P_k=\{0=x_{k,0}<\cdots<x_{k,N}=1\}$ be the set $P$ of $g_k$. The graph of $g_k$ is monotone on any interval $[x_{k,i-1}, x_{k,i}]$ for $i=1, \ldots, N$. Use $P_1$ and $P_2$ to derive the partition set $\{\mathcal{Y}_l\}$ defined in Section~\ref{sec:equivalence} to determine the evolution sequences of $g_1$ and $g_2$. Maps $g_1$ and $g_2$ have the same evolution sequence because $g_1^*=g_2^*$, and are thus in the same equivalence class by Lemma~\ref{lemma:sameEC}.
\end{proof}
The converse of Proposition~\ref{proposition:tcec} is not necessarily true. Figure~\ref{fig:WeeklyDiscussion_20200504Page-5} shows an example of $g_1$ and $g_2$ that are in the same equivalence class but are not topologically conjugate. In this sense, topological conjudacy is a stronger relationship between maps than equivalence classes.
\begin{figure}
 \centering
  \includegraphics[width=14cm]{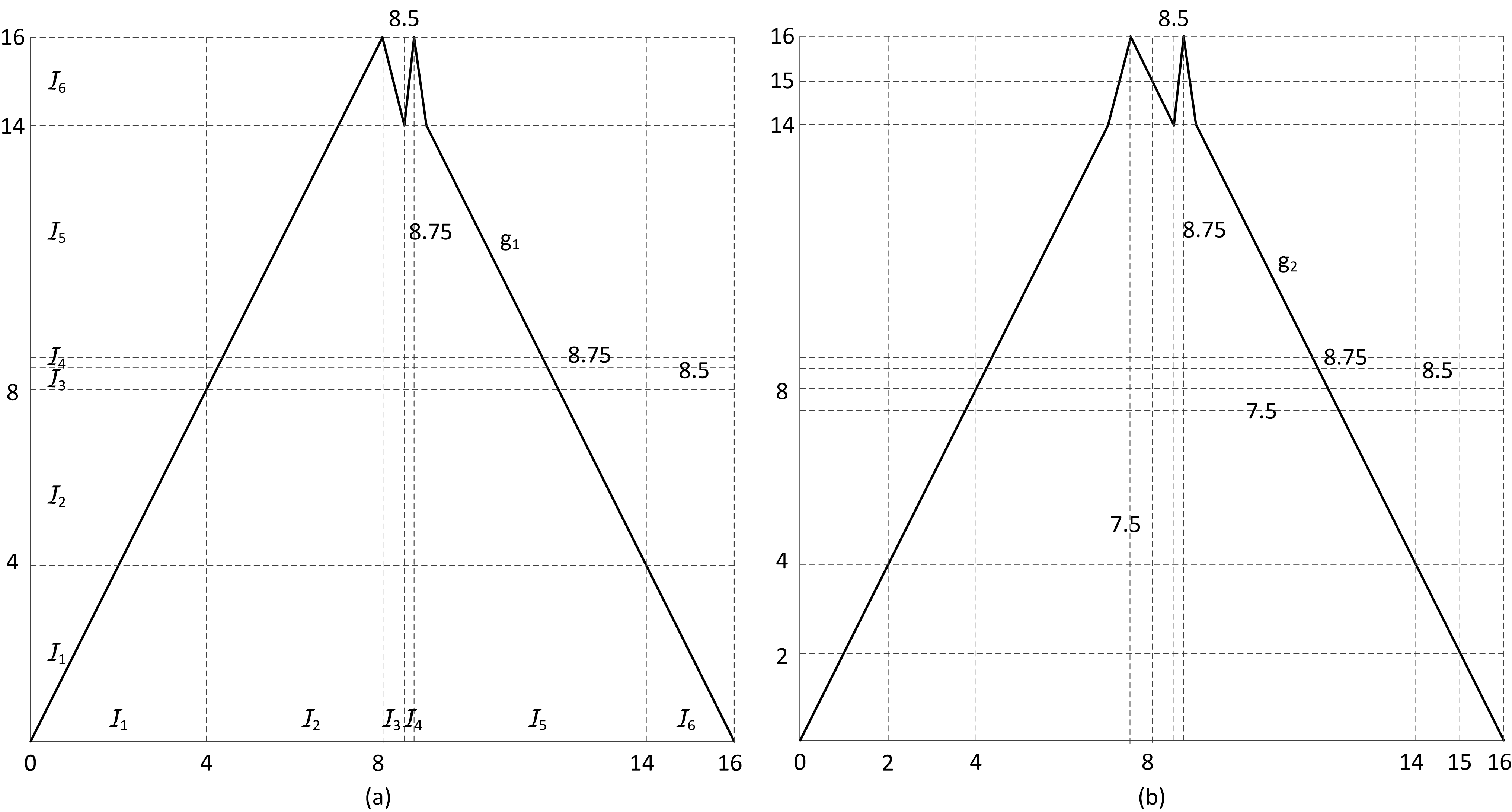}
  \caption{An example of $g_1, g_2$ in the same equivalence class but of distinct partition sets of $P$. In (a), for $g_1$, $P=2^{-4}\cdot\{0, 4, 8, 8.5, 8.75, 14, 16\}$.  In (b), for $g_2$, $P=2^{-4}\cdot\{0, 2, 4, 7.5, 8, 8.5, 8.75, 14, 15, 16\}$. $g_1^*\neq g_2^*$.}
  \label{fig:WeeklyDiscussion_20200504Page-5}
\end{figure} 

Now consider topological conjugacy and $\lambda$-preservation together. Observe the following.
\begin{itemize}
\item
Suppose that $s_1$ and $s_2$ are topologically conjugate. If $s_1\in PA(\lambda)$, it is possible that $s_2\not\in PA(\lambda)$. An example is shown in Figure~\ref{fig:WeeklyDiscussion_20200525fig-1+Page-1}(a). 
\item
It is possible that no $\lambda$-preserving $s_2$ exists to be topologically conjugate to a given $s_1$. Figure~\ref{fig:WeeklyDiscussion_20200525fig-1+Page-1}(b) shows one example of such $s_1$. In this example, for any topologically conjugate $s_2$, $s_2(s_2(1))=1$, $s_2^{-1}(s_2(1))=\{c,1\}$ where $0<c<s_2(1)$, and $s_2(x)<s_2(1)$ for $0\le x<c$ and $s_2(x)>s_2(1)$ for $c<x<1$. Therefore, $\lambda(s_2^{-1}([s_2(1),1]))=\lambda([c,1])>\lambda([s_2(1),1])$. $s_2\not\in C(\lambda)$.
\end{itemize}

\begin{figure}
\centering
\begin{subfigure}{.55\textwidth}
  \centering
  \includegraphics[width=1.\linewidth]{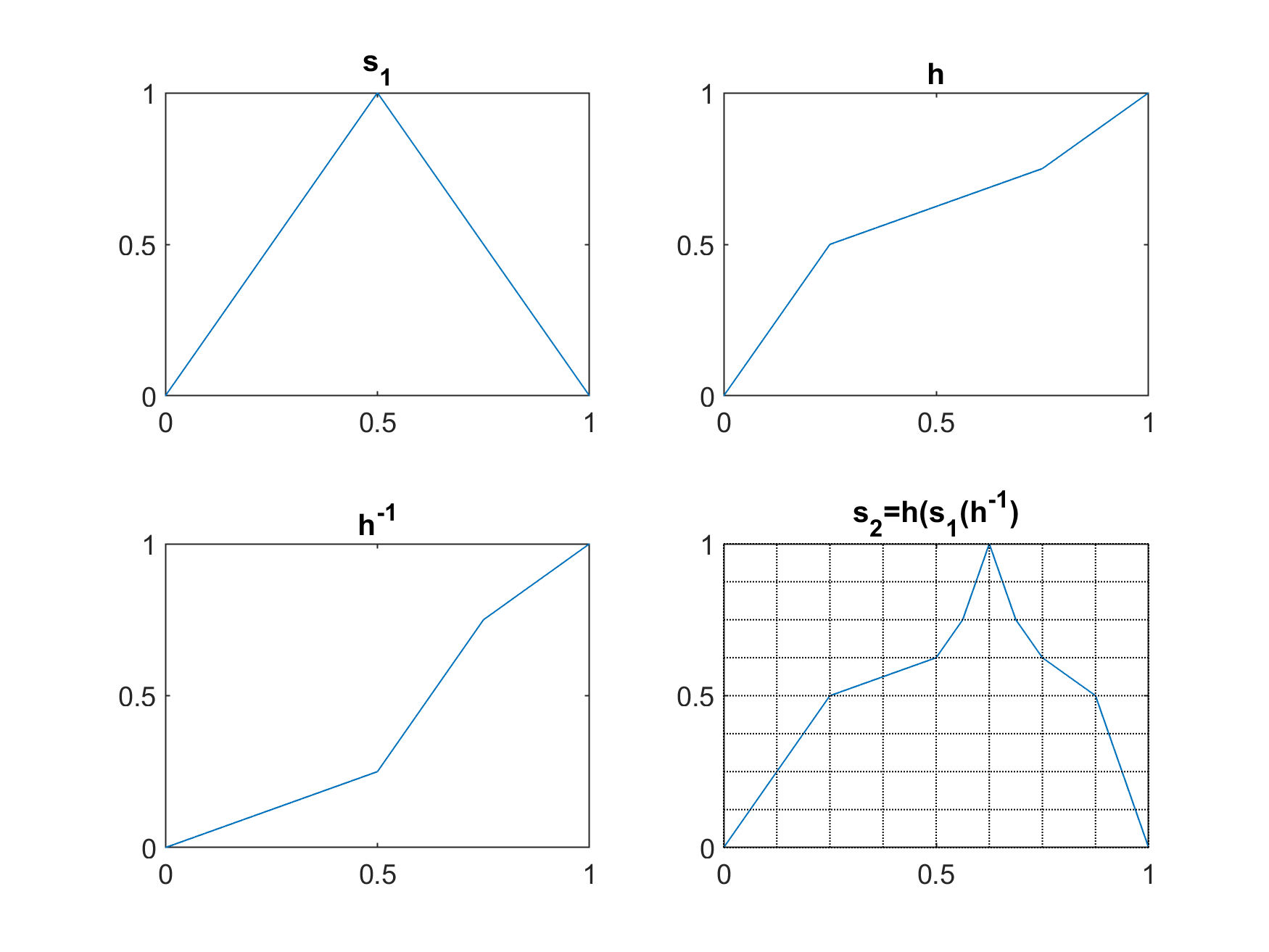}
    \caption{}
\end{subfigure}
\begin{subfigure}{.40\textwidth}
  \centering
  \includegraphics[width=0.5\linewidth]{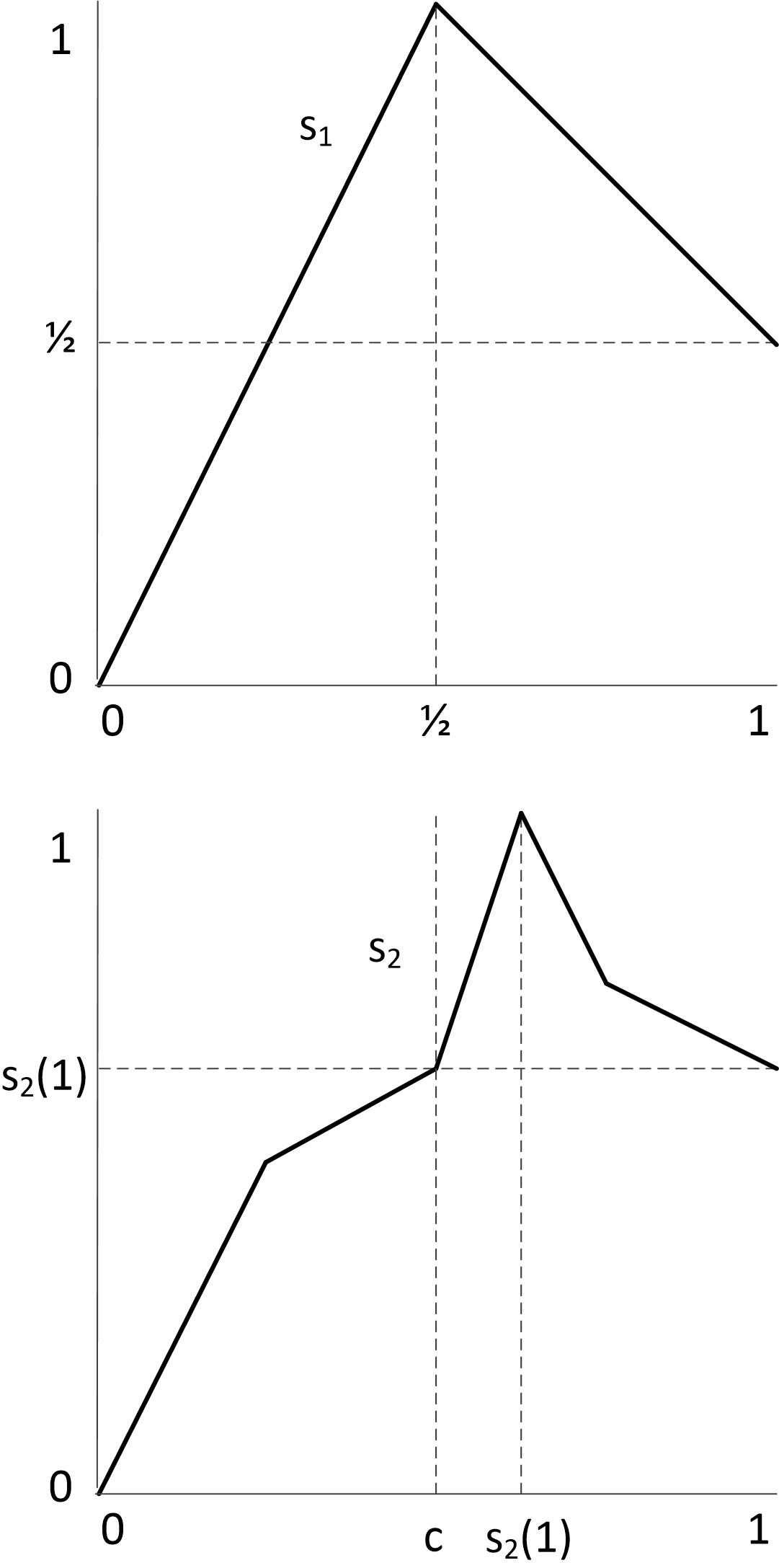}
   \caption{}
\end{subfigure}
\caption{Two examples to illustrate the observations of topological conjugacy and $\lambda$-preservation. (a) An example shows that while $s_1\in PA(\lambda)$, $s_2\not\in PA(\lambda)$. (b) An example of $s_1$ to which no $\lambda$-preserving $s_2$ exists that is topologically conjugate. The same homeomorphism $h$ as in (a) is used to produce $s_2$ in (b) as an illustrative example.}
\label{fig:WeeklyDiscussion_20200525fig-1+Page-1}
\end{figure}

This section characterizes continuous maps $s$ to which a $\lambda$-preserving $t$ exists to be topologically conjugate. To this end, make the following assumption of $s$.
\begin{assumption} 
First, set $\{0=\hat{x}_0<\cdots<\hat{x}_N=1\}$ exists such that for any $i=0, \ldots, N$, $s(\hat{x}_i)=\hat{x}_j$ for some $j$ with $0\le j \le N$. Second, the graph of $s$ is monotone on $[\hat{x}_{i-1}, \hat{x}_i]$ for any $i=1, \ldots, N$.
\label{assumption:familyofcontinuousmaps}
\end{assumption}
Map $s$ is not necessarily affine or expanding on $[\hat{x}_{i-1}, \hat{x}_i]$. Linear or expanding Markov maps are a strict subset of continuous maps that satisfy Assumption~\ref{assumption:familyofcontinuousmaps}. The remainder of this section is to characterize continuous maps $s$ under Assumption~\ref{assumption:familyofcontinuousmaps} for which $t\in PA(\lambda)$ or $t\in \mathbb{G}$ exists such that $t$ and $s$ are topologically conjugate. First consider linear or expanding Markov maps and then extend the results to a mixed case where the maps are linear but not expanding on some intervals and expanding but not linear on other intervals. This mixed case is important for the study of $\lambda$-preserving maps.

For $s$ under Assumption~\ref{assumption:familyofcontinuousmaps}, define index map $s^*$ as in (\ref{eq:s^*definition}). The basic idea is to construct $t$ by ``continuously connecting the dots''. Specifically, let $P=\{0=x_0<\cdots<x_N=1\}$ be a partition of $[0,1]$ such that $t(x_i)=x_j$ and the graph of $t$ be a monotone and piecewise affine segment on $[x_{i-1}, x_i]$. Index map $t^*$ is well defined in (\ref{eq:s^*definition}). The partition $P$ and piecewise affine segments on $\{[x_{i-1}, x_i]\}$ are to be constructed so that $t^*=s^*$ and $t$ preserves $\lambda$.

Define an $N\times N$ matrix $A^*$ from $s^*$ 
\begin{equation}
A^*_{i,j}=\left\{
\begin{array}{cc}
1, &\mbox{if }\min(s^*(i-1),s^*(i))<j\le \max(s^*(i-1),s^*(i)),\\
0, &\mbox{otherwise}.
\end{array}
\right.
\label{eq:A^*froms^*}
\end{equation}
While $A^*$ is defined by index map $s^*$, $s^*$ is uniquely determined by $A^*$. 

Now construct piecewise affine $t$. Let interval $\mathcal{I}_i=[x_{i-1},x_i]$ for $i=1, \ldots, N$. Let $t$ be monotone on each $\mathcal{I}_i$ and be affine on $\mathcal{I}_i \cap t^{-1}(\mathcal{I}_j)$ whenever $A^*_{i,j}=1$. Denote by $a_{i,j}$ the slope of the affine segment. $a_{i,j}\neq0$ is required for $\lambda$-preservation. Because $t$ is monotone on $\mathcal{I}_i$, $a_{i,j}>0$ if $s^*_{i-1}<s^*_i$ and $a_{i,j}<0$ otherwise. See Figure~\ref{fig:WeeklyDiscussion_20200525Page-2} for an illustration. Given $i$, if $a_{i,j}$ is the same for all $j$ whenever $A^*_{i,j}=1$, then $t$ is an affine segment on $\mathcal{I}_i$; if $|a_{i,j}|>1$ for all $j$ whenever $A^*_{i,j}=1$, then $t$ is an expanding monotone piece on $\mathcal{I}_i$. 

\begin{figure}
 \centering
  \includegraphics[width=5cm]{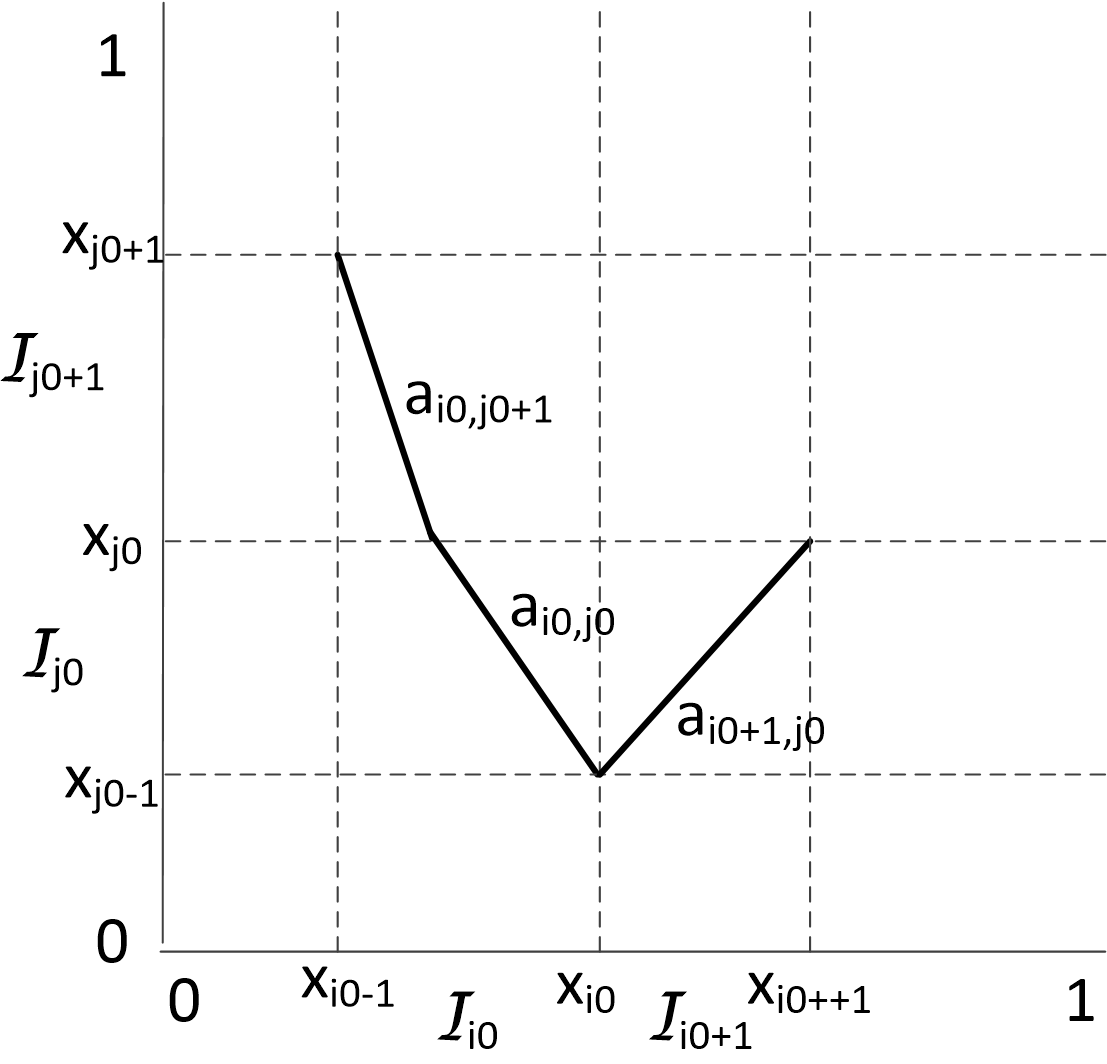}
  \caption{Illustration of piecewise expanding monotone $t$ on $\mathcal{I}_i$. In this example, $s^*_{i_0-1}=j_0+1, s^*_{i_0}=j_0-1, s^*_{i_0+1}=j_0$. $A^*_{i_0,j_0}=A^*_{i_0,j_0+1}=1$ and $A^*_{i_0,j}=0$ for $j\neq j_0, j_0+1$. The graph of $t$ is monotone on $\mathcal{I}_{i_0}$ consisting of two affine segments with slopes $a_{i_0,j_0}$ and $a_{i_0,j_0+1}$, respectively. $A^*_{i_0+1,j_0}=1$ and $A^*_{i_0+1,j}=0$ for $j\neq j_0$. The graph of $t$ is affine on $\mathcal{I}_{i_0+1}$ with slope $a_{i_0+1,j_0}$.}
  \label{fig:WeeklyDiscussion_20200525Page-2}
\end{figure} 

For $t^*=s^*$, $\{\mathcal{I}_i\}$ and $\{a_{i,j}\}$ are determined such that $t$ satisfies
\begin{equation}
t(\mathcal{I}_i)=\bigcup_{j=\min(s^*(i-1),s^*(i))+1}^{\max(s^*(i-1),s^*(i))} \mathcal{I}_j.
\label{eq:mappingstructure_s}
\end{equation}
That is to say that $\mathcal{I}_j\subseteq t(\mathcal{I}_i)$ if and only if  $A^*_{i,j}=1$.

Define non-negative matrix $A$ 
\begin{equation}
A_{i,j}=\left\{
\begin{array}{cc}
|a_{i,j}|^{-1}, &\mbox{if }A^*_{i,j}=1\\
0, &\mbox{otherwise}.
\end{array}
\right.
\label{eq:Aji}
\end{equation}
For $t$ to be Markov and continuous and satisfy (\ref{eq:mappingstructure_s}), $\{|\mathcal{I}_i|\}$ solves the following system of linear equations,
\begin{equation}
\left[
\begin{array}{c}
|\mathcal{I}_1|\\
\vdots\\
|\mathcal{I}_N|
\end{array}
\right]
= A
\left[
\begin{array}{c}
|\mathcal{I}_1|\\
\vdots\\
|\mathcal{I}_N|
\end{array}
\right].
\label{eq:IAI}
\end{equation}
A vector is said \emph{positive} if each element is positive. A solution to (\ref{eq:IAI}) must be a positive vector that sums to $1$ for $P$ to be a valid partition, 
\begin{equation}
\sum_{i=1}^N |\mathcal{I}_i|=1, |\mathcal{I}_i|>0, \mbox{ for $i=1, \ldots, N$}.
\label{eq:IAI1}
\end{equation}
From Lemma~\ref{lemma:basicderivativeinverse}, for $t$ to preserve $\lambda$,
\begin{equation}
\sum_{i=1}^N A_{i,j}=1
\label{eq:aji_t}
\end{equation}
for any $j$. If $\{|\mathcal{I}_1|, \ldots, |\mathcal{I}_N|\}$ and $\{|a_{i,j}|\}$ exist to satisfy (\ref{eq:IAI}), (\ref{eq:IAI1}) and (\ref{eq:aji_t}), then by construction, $t\in PA(\lambda)$ and $t^* = s^*$.  Furthermore, for $t\in \mathbb{G}$, $\{|a_{i,j}|\}$ are in the form of $\pm2^k$ with integer $k$ and $\{\mathcal{I}_i\}$ are dyadic numbers.

One can permute $A$ by reversely mapping indices $\{1, 2, \ldots, N\}\rightarrow \{N, N-1, \ldots, 1\}$ in (\ref{eq:IAI}). The resultant $t$ is such that $t^* = {}^*s$. For the sake of simplicity, ignore this case of permuted $A^*$ and ${}^*s$ in the reminder of this section, because such permutation does not affect the existence of $t$, as will be clear in Lemma~\ref{lemma:MCsolutions}.

From (\ref{eq:aji_t}), $A$ is a \emph{column stochastic matrix} as each entry is non-negative and each column sums to 1. Thus, $1$ is an eigenvalue of $A$. Matrix $A$ defines a Markov chain where $A_{i,j}$ represents the transition probability from state $j$ to $i$. The theory of Markov chains can be used to solve (\ref{eq:IAI}), (\ref{eq:IAI1}) and (\ref{eq:aji_t}). 

Specifically, let $|\mathcal{I}_1|, \ldots, |\mathcal{I}_N|$ represent the nodes of a directed graph and $A^*$ be the adjacency matrix. An arc exists from node $|\mathcal{I}_j|$ to $|\mathcal{I}_i|$ if $A^*_{i,j}=1$. Node $|\mathcal{I}_i|$ is \emph{reachable} from $|\mathcal{I}_j|$ if $(A^{*})^k_{i,j}>0$ for some $k\ge1$. Nodes $|\mathcal{I}_i|$ and $|\mathcal{I}_j|$ are said to \emph{communicate} if $|\mathcal{I}_i|$ is reachable from $|\mathcal{I}_j|$ and $|\mathcal{I}_j|$ is reachable from $|\mathcal{I}_i|$. The set of nodes $\{|\mathcal{I}_1|, \ldots, |\mathcal{I}_N|\}$ can be uniquely decomposed into $K$ disjoint subsets $C_k$, with
\[
\{|\mathcal{I}_1|, \ldots, |\mathcal{I}_N|\}=\bigcup_{k=1}^K C_k,
\]
for some positive integer $K$, such that nodes of each subset communicate and nodes of different subsets do not communicate. A node is \emph{recurrent} if the probability of ever returning to node $|\mathcal{I}_i|$ starting in node $|\mathcal{I}_i|$ is $1$, and is \emph{transient} otherwise. All nodes in a given subset $C_k$ are either recurrent or transient.

\begin{lemma} [\textsl{Sericola, 2013, \cite{sericola2013markov}}]
Let $\{\pi_1, \ldots, \pi_N\}$ be the limiting probability distribution of the Markov chain defined by $A$. Vector $[\pi_1, \ldots, \pi_N]$ is an eigenvector of $A$ corresponding to eigenvalue $1$ and sums to $1$. The solution $\{\pi_1, \ldots, \pi_N\}$ can be categorized into three types depending on $A^*$. 
\begin{itemize}
\item
If $K=1$, then every node of $\{|\mathcal{I}_1|, \ldots, |\mathcal{I}_N|\}$ is reachable from every other node. All nodes are recurrent and $A^*$ is said \emph{irreducible}. By Perron-Frobenius Theorem \cite[Theorem.\ 0.1]{notesonperronfrobenius}, a unique positive solution $\{\pi_1, \ldots, \pi_N\}$ exists. 
\item
If $K>1$ and every subset $C_k$ is recurrent, then $A^*$ can be decomposed into $K$ irreducible subsets. Infinitely many positive solutions $\{\pi_1, \ldots, \pi_N\}$ exist.
\item
If $K>1$ and at least one subset $C_k$ is transient, then no positive solution exists, because the limiting probability of any transient state is $0$, therefore violating (\ref{eq:IAI1}). 
\end{itemize}
\label{lemma:MCsolutions}
\end{lemma}

We will use Lemma~\ref{lemma:MCsolutions} to study the existence of $t \in PA(\lambda)$ that is topologically conjugate to $s$.

First consider the case where for any $j$,
\begin{equation}
\sum_{i=1}^N A^*_{i,j}>1.
\label{eq:sumA^*>1}
\end{equation}
An example is illustrated in Figure~\ref{fig:WeeklyDiscussion_20200504Page-5}(a) and described in Example~\ref{example:tinG}.   
\begin{theorem}
If $A^*$ is irreducible and satisfies (\ref{eq:sumA^*>1}), then for any $\{a_{i,j}\}$ satisfying (\ref{eq:aji_t}), a unique $t$ exists such that $t$ is expanding Markov, $t\in PA(\lambda)$, and $t$ and $s$ are topologically conjugate if $s$ is a linear or expanding Markov map. If $A^*$ can be decomposed into multiple irreducible subsets, then infinitely many such $t$ exist. If $A^*$ can be decomposed into multiple subsets, at least one of which are transient, then no such $t$ exists. 
\label{theorem:pfapp}
\end{theorem}
\begin{proof}
From (\ref{eq:sumA^*>1}), given any $j$, the number of $i$ for which $A^*_{i,j}=1$ for $i=1, \ldots, N$ is greater than $1$. Thus it is easy to select $\{|a_{i,j}|\}$, with $|a_{i,j}|>1$ for any $i,j$, to satisfy (\ref{eq:aji_t}). The graph of $t$ on $\mathcal{I}_i$ is thus expanding for any $i$. 

Consider the first case of $A^*$ being irreducible. From Lemma~\ref{lemma:MCsolutions}, a unique positive eigenvector $\boldsymbol{v}$ of $A$ exists corresponding to eigenvalue $1$ with $|\boldsymbol{v}|=1$. Let $[|\mathcal{I}_1|, \ldots, |\mathcal{I}_N|]^T=\boldsymbol{v}$. Let partition set $P=\{0=x_0< \cdots< x_N=1\}$ where $x_i=\sum_{l=1}^i |\mathcal{I}_l|$ for $i=1, \ldots, N$. Let $t(x_i)=x_{s^*_i}$. Interval $\mathcal{I}_i$ is partitioned into $\{\mathcal{I}_{i,j_0}<\cdots<\mathcal{I}_{i,j_1}\}$ where $A^*_{i,j}=1$ for $j=j_0,\ldots, j_1$. 
If $s^*_{i-1}<s^*_i$, then $|\mathcal{I}_{i,j}|=|a_{i,j}|^{-1} |\mathcal{I}_{j}|$ and $t$ is an affine segment with slope $a_{i,j}$; If $s^*_{i-1}>s^*_i$, then $|\mathcal{I}_{i,j}|=|a_{i,j}|^{-1} |\mathcal{I}_{j_1+j_0-j}|$ and $t$ is an affine segment with slope $-a_{i,j}$. Therefore $t$ is completely defined. By (\ref{eq:IAI}), $t$ is continuous at the boundary between the piecewise affine segments on adjacent intervals $\mathcal{I}_{i-1}$ and $\mathcal{I}_i$. By construction, $t$ is expanding Markov, $t^*=s^*$ and $t\in PA(\lambda)$. If $s$ is a linear or expanding Markov map, then by Theorem~\ref{theorem:s*1s*2}, $t$ and $s$ are topologically conjugate.

The other two cases of $A^*$ can be shown analogously.
\end{proof}
Additional conditions are required for $t\in\mathbb{G}$ as stated in the following corollary.
\begin{corollary}
Suppose that $A^*$ is irreducible and satisfies (\ref{eq:sumA^*>1}). For a set of $\{a_{i,j}\}$ satisfying (\ref{eq:aji_t}) and being in the form of $\pm 2^k$ for integer $k$, if (\ref{eq:IAI}) and (\ref{eq:IAI1}) have a dyadic solution, then unique $t$ exists such that $t$ is expanding Markov, $t\in \mathbb{G}$, and $t$ and $s$ are topologically conjugate.
\label{corollary:pfapp}
\end{corollary}
The following examples show that the choice of $\{a_{i,j}\}$ makes no difference in determining whether $t\in PA(\lambda)$ exists, as expected from Theorem~\ref{theorem:pfapp}, but plays an important role for $t\in\mathbb{G}$.
\begin{example}
Suppose that $N=6$, $s^*(0)=0, s^*(1)=2, s^*(2)=6, s^*(3)=5, s^*(4)=6, s^*(5)=1, s^*(6)=0$. From $s^*$, 
\[
A^*=\left[
\begin{array}{cccccc}
1 & 1 & 0 & 0 & 0 & 0 \\
0 & 0 & 1 & 1 & 1 & 1 \\
0 & 0 & 0 & 0 & 0 & 1 \\
0 & 0 & 0 & 0 & 0 & 1 \\
0 & 1 & 1 & 1 & 1 & 1 \\
1 & 0 & 0 & 0 & 0 & 0 
\end{array}
\right].
\]
$A^*$ is irreducible. Let
\[
A=\left[
\begin{array}{cccccc}
2^{-1} & 2^{-1} & 0 & 0 & 0 & 0 \\
0 & 0 & 2^{-1} & 2^{-1} & 2^{-1} & 2^{-1} \\
0 & 0 & 0 & 0 & 0 & 2^{-2} \\
0 & 0 & 0 & 0 & 0 & 2^{-3} \\
0 & 2^{-1} & 2^{-1} & 2^{-1} & 2^{-1} & 2^{-3} \\
2^{-1} & 0 & 0 & 0 & 0 & 0 
\end{array}
\right].
\]
The solution to (\ref{eq:IAI}) is given by 
\[
\left[
\begin{array}{ccc}
|\mathcal{I}_1|, &\ldots, &|\mathcal{I}_6|
\end{array}
\right]=
\left[
\begin{array}{cccccc}
\frac{1}{4} &\frac{1}{4}& \frac{1}{32} &\frac{1}{64}& \frac{21}{64} &\frac{1}{8}
\end{array}
\right].
\]
In this case, $t\in\mathbb{G}$. Figure~\ref{fig:WeeklyDiscussion_20200504Page-5}(a) plots $t$. However, for a slightly different choice of $A$
\[
A=\left[
\begin{array}{cccccc}
2^{-1} & 2^{-1} & 0 & 0 & 0 & 0 \\
0 & 0 & 2^{-1} & 2^{-1} & 2^{-1} & 2^{-2} \\
0 & 0 & 0 & 0 & 0 & 2^{-1} \\
0 & 0 & 0 & 0 & 0 & 2^{-3} \\
0 & 2^{-1} & 2^{-1} & 2^{-1} & 2^{-1} & 2^{-3} \\
2^{-1} & 0 & 0 & 0 & 0 & 0 
\end{array}
\right].
\]
The solution to (\ref{eq:IAI}) is given by 
\[
\left[
\begin{array}{ccc}
|\mathcal{I}_1|, &\ldots, &|\mathcal{I}_6|
\end{array}
\right]=
\left[
\begin{array}{cccccc}
\frac{4}{17} &\frac{4}{17}& \frac{1}{17} &\frac{1}{68}& \frac{23}{68} &\frac{2}{17}
\end{array}
\right].
\]
In this case, $t\in PA(\lambda)$ but $t\not\in\mathbb{G}$.
\label{example:tinG}
\end{example}
\begin{example}
$A^*$ is not irreducible in the following two cases. First,
\[
A^*=\left[
\begin{array}{cccccc}
0 & 1 & 1 & 0 & 0  \\
0 & 1 & 1 & 0 & 0  \\
1 & 0 & 0 & 0 & 0 \\
1 & 1 & 1 & 1 & 1  \\
1 & 1 & 1 & 1 & 1  
\end{array}
\right].
\]
No positive solution to (\ref{eq:IAI}) exists, because solving (\ref{eq:IAI}) for any $\{a_{i,j}\}$ leads to $|\mathcal{I}_1|=|\mathcal{I}_2|=|\mathcal{I}_3|=0$ because the first three states are transient. Next,
\[
A^*=\left[
\begin{array}{cccc}
1 & 1 & 0 & 0   \\
1 & 1 & 0 & 0   \\
0 & 0 & 1 & 1   \\
0 & 0 & 1 & 1 
\end{array}
\right].
\]
Infinitely many positive solutions to (\ref{eq:IAI}) exist for any $\{a_{i,j}\}$, because the first two states form an irreducible subset and the last two states form another irreducible subset.
\end{example}

Next consider the case where for some $j_0$,
\begin{equation}
\sum_{i=1}^N A^*_{i,j_0}=1.
\label{eq:sumA^*=1}
\end{equation}   
The difference from (\ref{eq:sumA^*>1}) is that in (\ref{eq:sumA^*=1}) a single $i_0$ exists such that $A^*_{i_0,j_0}=1$ and $A^*_{i,j_0}=0$ for all $i\neq i_0$. From (\ref{eq:aji_t}), $|a_{i_0,j_0}|=1$ to be $\lambda$-preserving. If $A^*_{i_0,j'}=1$ for some $j'\neq j_0$, then $t$ on $\mathcal{I}_{i_0}$ is neither an affine segment nor an expanding monotone piece as illustrated in Figure~\ref{fig:WeeklyDiscussion_20200525Page-3}(a), in which case Theorem~\ref{theorem:s*1s*2} cannot be applied. To avoid this problem, the remainder of the section assumes the following.
\begin{assumption}
First, one and only one $j_0$ exists such that a single $i_0$ exists where $A^*_{i_0,j_0}=1$ and $A^*_{i,j_0}=0$ for all $i\neq i_0$. Second, for $i_0$ obtained in part (1), $A^*_{i_0,j}=0$ for any $j\neq j_0$. 
\label{assumption:11}
\end{assumption}
Part (1) of Assumption~\ref{assumption:11} is for simplicity. The results can be easily extended to the case where multiple such $j_0$ exist. Part (2) is needed such that $t$ on any $\mathcal{I}_i$ is either expanding monotone or an affine segment with slope equal to $\pm1$. Specifically, a unique pair of indices $i_0, j_0$ exist such that 
\begin{equation}
\left\{
\begin{array}{cc}
|a_{i_0,j_0}|=1, & \\
|a_{i,j}|>1, &\mbox{whenever $A^*_{i,j}=1$ and if $i\neq i_0$ or $j\neq i_0$}.
\end{array}
\right.
\label{eq:a=1a>1}
\end{equation}
Even with Assumption~\ref{assumption:11}, Theorem~\ref{theorem:s*1s*2} cannot be directly applied because $t$ is not a linear or expanding Markov map: $t$ is not expanding on $\mathcal{I}_{i_0}$ and not necessarily\footnote{For $t$ to be a $\lambda$-preserving linear Markov map, $A_{i,j}$ must be the same for all $j$ given $i$ whenever $A^*_{i,j}=1$ and (\ref{eq:aji_t}) must be satisfied. Whether $\{a_{i,j}\}$ exists to meet both conditions depends on $A^*$. For the first two $A^*$ in Example~\ref{example:4042} $t$ is a linear Markov map, and for the last $A^*$ $t$ is not a linear or expanding Markov map.} linear on $\mathcal{I}_{i}$ for $i\neq i_0$. Next, we apply Definition~\ref{definition:topologicalconjugacy} of topological conjugacy directly to circumvent this technical issue.

\begin{figure}
 \centering
  \includegraphics[width=14cm]{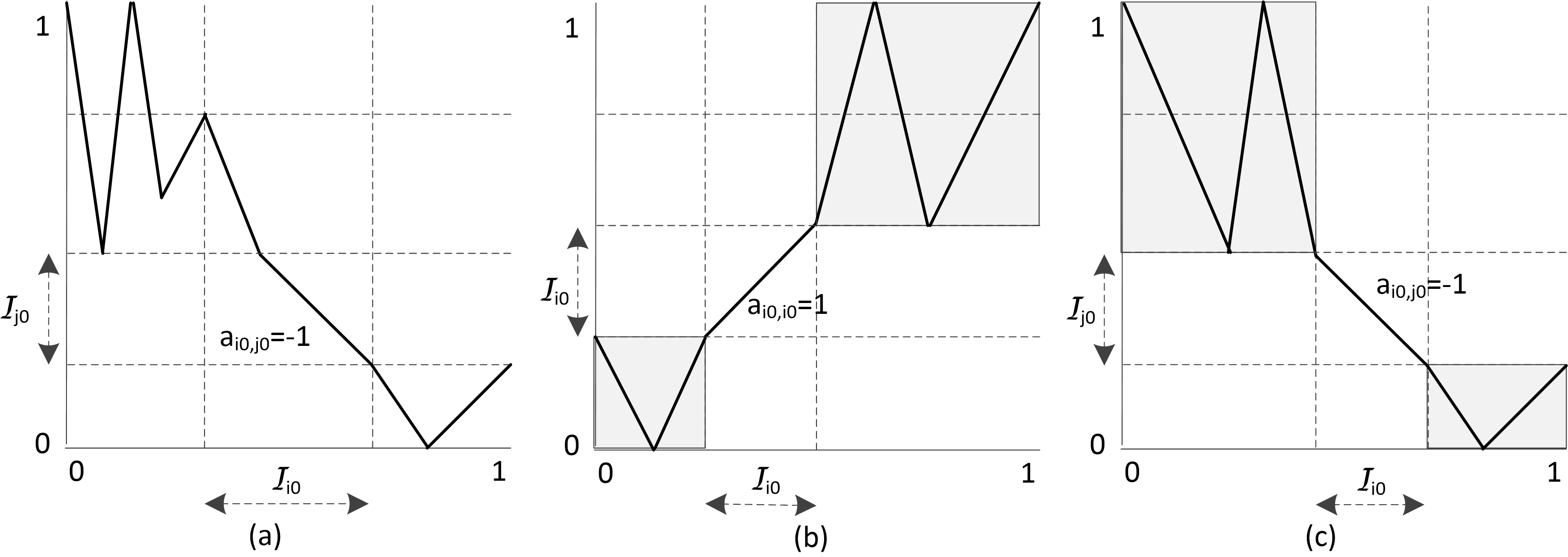}
  \caption{Illustration of the cases described by (\ref{eq:sumA^*=1}). Case (a) is not considered under Assumption~\ref{assumption:11}. Cases (b) and (c) are addressed in Corollary~\ref{corollary:pfapp+1} and Corollary~\ref{corollary:pfapp-1} respectively.}
  \label{fig:WeeklyDiscussion_20200525Page-3}
\end{figure}

First, suppose that $s^*_{i_0-1}<s^*_{i_0}$, as illustrated in Figure~\ref{fig:WeeklyDiscussion_20200525Page-3}(b). Then $a_{i_0,j_0}=1$. In this case, 
\[
t\left(\bigcup_{i=1}^{i_0-1} \mathcal{I}_i\right) = \bigcup_{i=1}^{j_0-1} \mathcal{I}_i \mbox{ and }
t^{-1}\left(\bigcup_{i=1}^{j_0-1} \mathcal{I}_i\right) = \bigcup_{i=1}^{i_0-1} \mathcal{I}_i. 
\]
$j_0$ must be equal to $i_0$, because 
\[
\lambda\left(\bigcup_{i=1}^{j_0-1} \mathcal{I}_i\right)=\sum_{i=1}^{j_0-1} \left|\mathcal{I}_i\right|=\lambda\left(\bigcup_{i=1}^{i_0-1} \mathcal{I}_i\right)=\sum_{i=1}^{i_0-1} \left|\mathcal{I}_i\right|.
\]
Therefore, 
\begin{equation}
\left\{
\begin{array}{cc}
s^*_{i}\le i_0-1, & \mbox{if }i<i_0-1,\\
s^*_{i}=i, & \mbox{if }i=i_0-1, i_0\\
s^*_{i}\ge i_0, & \mbox{if }i>i_0.
\end{array}
\right.
\label{eq:si0-1formula}
\end{equation}

\begin{figure}
 \centering
  \includegraphics[width=17cm]{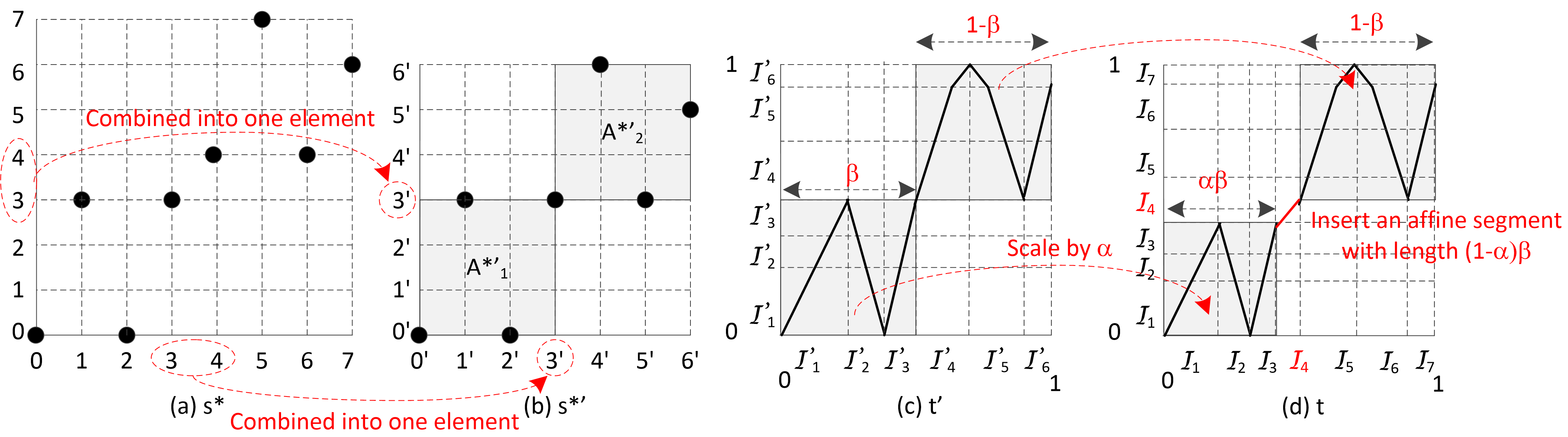}
  \caption{Construction of $t$ in the case of (\ref{eq:sumA^*=1}) with $a_{i_0,j_0}=1$. Here $N=7$ and $i_0=4$.}
  \label{fig:WeeklyDiscussion_20200525Page-5}
\end{figure} 
Figure~\ref{fig:WeeklyDiscussion_20200525Page-5} shows the construction of $t$. Recall that $\{0=\hat{x}_0<\cdots<\hat{x}_N=1\}$ is the partition of $s$. First, revise $s$ to obtain $s'$ by eliminating the portion on $[\hat{x}_{i_0-1},\hat{x}_{i_0}]$, scaling up the portion on $[0, \hat{x}_{i_0-1}]$ by a factor of $\frac{1}{\alpha}$ to fill up $[0,\hat{x}_{i_0}]$ and keeping the portion on $[\hat{x}_{i_0},1]$ unchanged. Specifically,
\begin{equation}
s'(x)=\left\{
\begin{array}{cc}
\frac{1}{\alpha} s\left(\alpha x\right), &\mbox{ if $0\le x<\hat{x}_{i_0}$},\\
s(x), &\mbox{ otherwise}.
\end{array}
\right.
\label{eq:s'definition}
\end{equation}
with 
\begin{equation}
\alpha=\frac{\hat{x}_{i_0-1}}{\hat{x}_{i_0}}.
\label{eq:alphascaling} 
\end{equation}
From (\ref{eq:si0-1formula}), $s(\hat{x}_{i_0-1})= \hat{x}_{i_0-1}$ and $s(\hat{x}_{i_0})= \hat{x}_{i_0}$. In (\ref{eq:s'definition}), $s'\left((\hat{x}_{i_0})^-\right)=\frac{1}{\alpha} s(\hat{x}_{i_0-1})=s(\hat{x}_{i_0})=s'\left((\hat{x}_{i_0})^+\right)$. Thus, $s'$ is continuous and expanding Markov. The change from $s$ to $s'$ in effect combines elements $i_0-1$ and $i_0$ of the $(N+1)$-element set $P^*=\{0, 1, \ldots, N\}$ into a single element to arrive at an $N$-element set ${P^*}'=\{0', 1', \ldots, (N-1)'\}$ where $0$ becomes $0'$, $1$ becomes $1'$, and so on, and $i_0-1$ and $i_0$ are combined to become $(i_0-1)'$, then $i_0+1$ becomes $i'_0$, $i_0+2$ becomes $(i_0+1)'$, and so on, and $N$ becomes $(N-1)'$. Revise the index map $s^*$ to become ${s^*}'$ when $P^*$ becomes ${P^*}'$, as shown in Figure~\ref{fig:WeeklyDiscussion_20200525Page-5}(a) and (b). 

The revised ${A^*}'$ obtained from the revised ${s^*}'$ satisfies (\ref{eq:sumA^*>1}) instead of (\ref{eq:sumA^*=1}). Matrix ${A^*}'$ is a block-diagonal one consisting of ${A^*}'_1$ and ${A^*}'_2$ where ${A^*}'_1$ is a map of $\{\mathcal{I}'_{1}, \ldots, \mathcal{I}'_{i_0-1}\}$ to itself and ${A^*}'_2$ is a map of $\{\mathcal{I}'_{i_0}, \ldots, \mathcal{I}'_{N-1}\}$ to itself, as shown in Figure~\ref{fig:WeeklyDiscussion_20200525Page-5}(c). If ${A^*}'$ can be decomposed into multiple irreducible subsets, then from Theorem~\ref{theorem:pfapp}, infinitely many $t'$ exists such that $t'$ is expanding Markov, $t'\in PA(\lambda)$, and $t'$ and $s'$ are topologically conjugate. Set $\{\mathcal{I}'_{1}, \ldots, \mathcal{I}'_{N-1}\}$ is obtained in the construction of $t'$ according to Theorem~\ref{theorem:pfapp}. Let homeomorphism $h'$ be such that 
\begin{equation}
s'=h'\circ t'\circ h'^{-1}.
\label{eq:s't'} 
\end{equation}
Comparing the scaling (\ref{eq:s'definition}) with (\ref{eq:s't'}), it follows that
\[
[\hat{x}_{i-1}, \hat{x}_i]= \left\{
\begin{array}{cc}
\alpha h'(\mathcal{I}'_i), & \mbox{ for $i=1, \ldots, i_0-1$},\\
h'(\mathcal{I}'_{i-1}), & \mbox{ for $i=i_0+1, \ldots, N$}.
\end{array}
\right.
 \]
In particular, $\hat{x}_{i_0}=h'(\beta)$ where $\beta = \sum_{i=1}^{i_0-1} |\mathcal{I}'_{i}|$.

Finally, revise $t'$ to obtain $t$ by scaling down $\mathcal{I}'_1, \ldots, \mathcal{I}'_{i_0-1}$ by a factor of $\alpha$, keeping $\mathcal{I}'_{i_0}, \ldots, \mathcal{I}'_{N-1}$ unchanged, and inserting a new interval of length $(1-\alpha)\beta$ between $\mathcal{I}'_{i_0-1}$ and $\mathcal{I}'_{i_0}$ to arrive at a set of $N$ intervals $\mathcal{I}_1, \ldots, \mathcal{I}_{N}$ with $\sum_{i=1}^{N} |\mathcal{I}_i|=1$, as shown in Figure~\ref{fig:WeeklyDiscussion_20200525Page-5}(d). Let $t$ be an affine segment with slope $1$ on $\mathcal{I}_{i_0}$ and continuous between $\mathcal{I}_{i_0-1},\mathcal{I}_{i_0}$ and between $\mathcal{I}_{i_0},\mathcal{I}_{i_0+1}$. Revise $h'$ to obtain $h$ correspondingly
\begin{equation}
h(x)=\left\{
\begin{array}{cc}
\alpha h'\left(\frac{1}{\alpha} x\right), &\mbox{ if $0\le x<\alpha\beta$},\\
\hat{x}_{i_0}+\frac{\hat{x}_{i_0}}{\beta}(x-\beta), &\mbox{ if $\alpha\beta\le x<\beta$},\\
h'(x), &\mbox{ otherwise}.
\end{array}
\right.
\label{eq:h'definition}
\end{equation}
Recall that $t$ on $\mathcal{I}_{i_0}$ is affine with slope $1$. If $s$ is affine with slope $1$ on $[\hat{x}_{i_0-1},\hat{x}_{i_0}]$, then by (\ref{eq:s't'}) and (\ref{eq:h'definition}), it follows that
\begin{equation}
s=h\circ t\circ h^{-1}.
\end{equation}
Then
\[
[\hat{x}_{i-1}, \hat{x}_i]= h(\mathcal{I}_i)
\]
for $i=1, \ldots, N$. Moreover, because $t'\in PA(\lambda)$ and $t$ on $\mathcal{I}_{i_0}$ is affine with slope $1$, $t\in PA(\lambda)$. Hence, the following corollary holds.

\begin{corollary}
Suppose that $A^*$ satisfies (\ref{eq:sumA^*=1}) and that Assumption~\ref{assumption:11} holds. Suppose that $j_0=i_0$ and $s^*_{i_0-1}<s^*_{i_0}$.  Suppose that $s$ is affine with slope $1$ on $[\hat{x}_{i_0-1},\hat{x}_{i_0}]$. If the revised ${A^*}'$ can be decomposed into irreducible subsets, then $t$ exists such that $t$ and $s$ are topologically conjugate and $t\in PA(\lambda)$. Furthermore, suppose that $t'$ is constructed from the revised $s'$ in (\ref{eq:s'definition}) and (\ref{eq:s't'}). If $t'\in\mathbb{G}$ and $\alpha$ of (\ref{eq:alphascaling}) is dyadic, then $t\in\mathbb{G}$.
\label{corollary:pfapp+1}
\end{corollary}
Clearly $t$ is not unique in the preceding construction because the solution of $\{\mathcal{I}'_1, \ldots, \mathcal{I}'_{N-1}\}$ in Theorem~\ref{theorem:pfapp} are not unique.

Next, suppose that $s^*_{i_0-1}>s^*_{i_0}$, as illustrated in Figure~\ref{fig:WeeklyDiscussion_20200525Page-3}(c). Then $a_{i_0,j_0}=-1$. If $j_0=i_0$, then one can construct $t$ analogously to the preceding case of $a_{i_0,j_0}=1$ except that in the final step when a new interval of length $(1-\alpha)\beta$ is inserted as shown in Figure~\ref{fig:WeeklyDiscussion_20200525Page-5}(d), $t$ is an affine segment with slope $-1$, instead of $1$, on $\mathcal{I}_{i_0}$. 

Now suppose $j_0\neq i_0$. Assume $i_0<j_0$. (The case of $i_0>j_0$ can be addressed analogously.) In this case,
\begin{equation}
t(\mathcal{I}_{i_0})=\mathcal{I}_{j_0}\mbox{ and } t^{-1}(\mathcal{I}_{j_0})=\mathcal{I}_{i_0}.
\label{eq:ti0j0_1}
\end{equation}
\begin{equation}
t\left(\bigcup_{i=1}^{i_0-1} \mathcal{I}_i\right) = \bigcup_{i=j_0+1}^{N} \mathcal{I}_i \mbox{ and } t^{-1}\left(\bigcup_{i=j_0+1}^{N} \mathcal{I}_i\right)=\bigcup_{i=1}^{i_0-1} \mathcal{I}_i.
\label{eq:ti0j0_3}
\end{equation}
$\lambda$-preservation leads to 
\begin{equation}
\left|\mathcal{I}_{j_0}\right|=\left|\mathcal{I}_{i_0}\right|.
\label{eq:ti0j0_2}
\end{equation}
\begin{equation}
\lambda\left(\bigcup_{i=j_0+1}^{N} \mathcal{I}_i\right)=\sum_{i=j_0+1}^{N} \left|\mathcal{I}_i\right|=\lambda\left(\bigcup_{i=1}^{i_0-1} \mathcal{I}_i\right)=\sum_{i=1}^{i_0-1} \left|\mathcal{I}_i\right|.
\label{eq:ti0j0_4}
\end{equation}

\begin{lemma}
Suppose that $A^*$ satisfies (\ref{eq:sumA^*=1}) and that Assumption~\ref{assumption:11} holds. Suppose that $a_{i_0,j_0}=-1$ and $i_0<j_0$. Let $\{0=x_0<\cdots<x_N=1\}$ be a partition of $[0,1]$ and $\mathcal{I}_i=[x_{i-1}, x_i]$ for $i=1, \ldots, N$. Let $t\in C(\lambda)$ be monotone on each $\mathcal{I}_i$ and be affine on $\mathcal{I}_i \cap t^{-1}(\mathcal{I}_j)$ where $A^*_{i,j}=1$, and the slope $a_{i,j}$ of the affine segment satisfies (\ref{eq:a=1a>1}). Then
\begin{equation}
\left|\frac{d (t\circ t)}{dx}\right|>1
\label{eq:tcirct>1}
\end{equation} 
for all $x\in[0,1]$ except for a finite number of points.
\label{lemma:t2_tot}
\end{lemma}
\begin{proof}
From (\ref{eq:ti0j0_2}) and (\ref{eq:ti0j0_4}), it is easy to show that intervals $\mathcal{I}_{i_0}$ and $\mathcal{I}_{j_0}$ are symmetric with respect to $\frac{1}{2}$. From (\ref{eq:a=1a>1}), $\left|\frac{d t}{dx}\right|=1$ on $\mathcal{I}_{i_0}$ and $\left|\frac{d t}{dx}\right|>1$ on $[0,1]\setminus\mathcal{I}_{i_0}$.

\begin{figure}
 \centering
  \includegraphics[width=14cm]{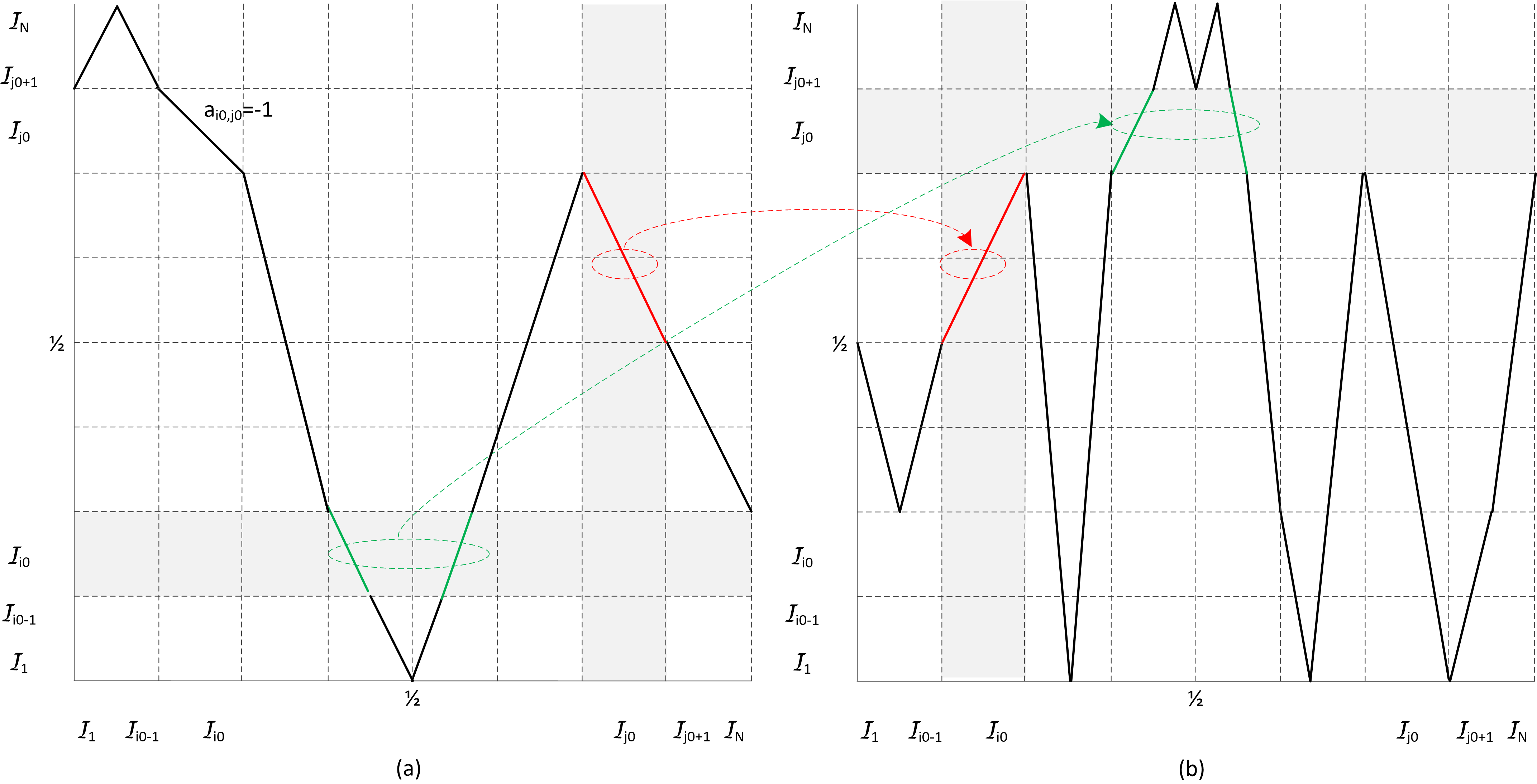}
  \caption{Illustration of $t$ in (a) and $t\circ t$ in (b) where $t$ satisfies (\ref{eq:ti0j0_1}) and (\ref{eq:ti0j0_3}). The graph of $t$ on $\mathcal{I}_{i_0}$ is an affine segment of slope $-1$. $t(\mathcal{I}_{i_0})=\mathcal{I}_{j_0}$. This affine segment transforms portions of $t$ to become portions of $t\circ t$ as highlighted as green and red segments.}
  \label{fig:WeeklyDiscussion_20200525Page-4}
\end{figure} 

The affine segment of $t$ with slope $-1$ on $\mathcal{I}_{i_0}$ affects $t\circ t$ in two ways, as illustrated in Figure~\ref{fig:WeeklyDiscussion_20200525Page-4}. First, $t$ on $\mathcal{I}_{j_0}$ flips horizontally along the $x=\frac{1}{2}$ axis to become $t\circ t$ on $\mathcal{I}_{i_0}$. Second, $t$ on $t^{-1}(\mathcal{I}_{i_0})$ flips vertically along the $y=\frac{1}{2}$ axis to become $t\circ t$ on $(t\circ t)^{-1}(\mathcal{I}_{j_0})$. On other portions of $[0,1]$, $t\circ t$ is obtained by the composition of two segments each with $\left|\frac{d t}{dx}\right|>1$, and thus $\left|\frac{d (t\circ t)}{dx}\right|$ is greater than either of the two. Hence, (\ref{eq:tcirct>1}) holds on $[0,1]$ whenever the derivative is defined.
\end{proof}

Lemma~\ref{lemma:t2_tot} states that $t\circ t$ is expanding Markov although $t$ is not. The proof of the expanding Markov case of Theorem~\ref{theorem:s*1s*2} in \cite[Theorem.\ 2.1]{10.2307/2000600} notes that the expanding property of $t$ is used only to make $\bigcup_n t^{-n} (P)$ dense and the theorem holds for maps for which some power is expanding Markov. Hence, the following corollary holds.
\begin{corollary}
Suppose that $A^*$ satisfies (\ref{eq:sumA^*=1}) and that Assumption~\ref{assumption:11} holds. Suppose that $a_{i_0,j_0}=-1$. If $i_0=j_0$, then the conclusion of Corollary~\ref{corollary:pfapp+1} holds. If $i_0\neq j_0$, then the conclusions of Theorem~\ref{theorem:pfapp} and Corollary~\ref{corollary:pfapp} hold.
\label{corollary:pfapp-1}
\end{corollary}

\begin{example}
First, infinitely many $t\in PA(\lambda)$ exist in Corollaries~\ref{corollary:pfapp+1} and  \ref{corollary:pfapp-1}, respectively, for 
\[
A^*=\left[
\begin{array}{ccccc}
0 & 0 & 0 & 1 & 1  \\
0 & 0 & 0 & 1 & 1  \\
0 & 0 & 1 & 0 & 0  \\
1 & 1 & 0 & 0 & 0  \\ 
1 & 1 & 0 & 0 & 0 
\end{array}
\right]
\mbox{ and }
A^*=\left[
\begin{array}{ccccc}
1 & 1 & 0 & 0 & 0  \\
1 & 1 & 0 & 0 & 0  \\
0 & 0 & 1 & 0 & 0  \\
0 & 0 & 0 & 1 & 1  \\
0 & 0 & 0 & 1 & 1  
\end{array}
\right],
\]
where $i_0=j_0=3$. $t\in\mathbb{G}$ exists. Second, a unique $t\in PA(\lambda)$ exists in Corollary~\ref{corollary:pfapp-1} for
\[
A^*=\left[
\begin{array}{cccccc}
0 & 0 & 1 & 1 & 0 & 0 \\
0 & 0 & 1 & 1 & 0 & 0 \\
0 & 1 & 0 & 0 & 0 & 0 \\
0 & 1 & 1 & 1 & 1 & 1 \\
0 & 1 & 1 & 1 & 1 & 1 \\
1 & 0 & 0 & 0 & 0 & 0
\end{array}
\right],
\]
where $i_0=6, j_0=1$. However, in this case, $t\not\in\mathbb{G}$.
\label{example:4042}
\end{example}

\section{Conclusion and Future Study}\label{sec:future}

This paper has introduced a new and interesting monoid, $\lambda$-preserving Thompson's monoid $\mathbb{G}$, modeled on Thompson's group $\mathbb{F}$, and studied a number of properties of $\mathbb{G}$. The main results of this paper improve several results of \cite{2019arXiv190607558B} and demonstrate an interesting interplay between algebraic and dynamical settings.

A few areas are worth further exploring.

First, ergodicity. This paper has studied the properties of TM, LEO and Markov, and characterized periods of periodic points of maps in $\mathbb{G}$. The next step is to systematically study statistical properties of long-term time averages of various functions along trajectories of the dynamical system governed by $g\in\mathbb{G}$.

Second, presentations. Thompson's group $\mathbb{F}$ admits infinite and finite presentations. As seen in this paper, $\lambda$-preserving Thompson's monoid $\mathbb{G}$ is more sophisticated and the notions of equivalence classes and sets of equivalence classes are useful in presenting $\mathbb{G}$. The next step is to construct various presentations of $\mathbb{G}$.

Third, analogue of $\mathbb{G}$ in a high dimensional space. This paper has studied interval maps, which exist in a one-dimensional space. It will be interesting to extend the study to a higher dimensional space and see if different conclusions will be drawn for high dimensional maps as opposed to interval maps.

\section*{Acknowledgments}

I would like to thank my mentor, Professor Sergiy Merenkov, for his continuous and insightful guidance and advice throughout the entire research process. He introduced me to the general fields of Lebesgue measure preserving interval maps and Thompson's groups, provided direction in my research,  informed me of the connection between my work and other results in the literature, and guided me through the writing of this paper. I would like to thank the MIT PRIMES-USA for giving me the opportunity and resources to work on this project and in particular Dr. Tanya Khovanova for proving great advice on enhancing the quality of this paper and Ms. Boya Song for her valuable comments on an earlier version of this paper. 

\bibliographystyle{unsrt}
\bibliography{arXiv_submission}

\end{document}